\documentclass[11pt,a4paper]{article}

\pdfoutput=1 
\usepackage{epsf,epsfig,amsfonts,amsgen,amsmath,amstext,amsbsy,amsopn,amsthm,cases,listings,color
}
\usepackage{ebezier,eepic}
\usepackage{color}
\usepackage{multirow}
\usepackage{epstopdf}
\usepackage{graphicx}
\usepackage{pgf,tikz}
\usepackage{mathrsfs}
\usepackage[marginal]{footmisc}
\usepackage{enumitem}
\usepackage[titletoc]{appendix}
\usepackage{booktabs}
\usepackage{url}
\usepackage{mathtools}

\usepackage{pgfplots}
\usepackage{authblk}
\usepackage{amssymb}

\usepackage{wasysym}

\usepackage{empheq}

\usepackage{dsfont}

\usepackage{subfigure}
\pgfplotsset{compat=1.18}
\usepackage{mathrsfs}
\usepackage{wasysym} 
\usetikzlibrary{arrows}
\usepackage{aligned-overset}
\usepackage{bm}
 \usepackage[backref=page]{hyperref}

\allowdisplaybreaks[1]

\definecolor{uuuuuu}{rgb}{0.27,0.27,0.27}
\definecolor{sqsqsq}{rgb}{0.1255,0.1255,0.1255}

\newtheorem{definition}{Definition} [section]
\newtheorem{theorem}[definition]{Theorem}
\newtheorem{lemma}[definition]{Lemma}
\newtheorem{proposition}[definition]{Proposition}

\newtheorem{corollary}[definition]{Corollary}
\newtheorem{conjecture}[definition]{Conjecture}
\newtheorem{claim}[definition]{Claim}
\newtheorem{problem}[definition]{Problem}

\newtheorem{fact}[definition]{Fact}
\newenvironment{pf}{\noindent{\bf Proof}}

\newcommand{\uproduct}{\mathbin{\;{\rotatebox{90}{\textnormal{$\small\Bowtie$}}}}}

\begin{document}
\title{\bf\Large Toward a density Corr\'{a}di--Hajnal theorem for degenerate hypergraphs}
\date{\today}
\author[1]{Jianfeng Hou\thanks{Research supported by National Key R\&D Program of China (Grant No. 2023YFA1010202), National Natural Science Foundation of China (Grant No. 12071077), the Central Guidance on Local Science and Technology Development Fund of Fujian Province (Grant No. 2023L3003). Email: \texttt{jfhou@fzu.edu.cn}}}
\author[1]{Caiyun Hu\thanks{Email: \texttt{hucaiyun.fzu@gmail.com}}}
\author[2]{Heng Li\thanks{Email: \texttt{heng.li@sdu.edu.cn}}}
\author[3]{Xizhi Liu\thanks{Research supported by ERC Advanced Grant 101020255 and Leverhulme Research Project Grant RPG-2018-424. Email: \texttt{xizhi.liu.ac@gmail.com}}}
\author[1]{Caihong Yang\thanks{Email: \texttt{chyang.fzu@gmail.com}}}
\author[1]{Yixiao Zhang\thanks{Email: \texttt{fzuzyx@gmail.com}}}
\affil[1]{Center for Discrete Mathematics,
            Fuzhou University, Fujian, 350003, China}
\affil[2]{School of Mathematics, Shandong University, 
            Shandong, 250100, China}
\affil[3]{Mathematics Institute and DIMAP,
            University of Warwick, 
            Coventry, CV4 7AL, UK}
\maketitle
\begin{abstract}
Given an $r$-graph $F$ with $r \ge 2$, let $\mathrm{ex}(n, (t+1) F)$ denote the maximum number of edges in an $n$-vertex $r$-graph with at most $t$ pairwise vertex-disjoint copies of $F$.
Extending several old results and complementing prior  work~\cite{HLLYZ23} on nondegenerate hypergraphs, we initiate a systematic study on $\mathrm{ex}(n, (t+1) F)$ for degenerate hypergraphs $F$. 

For a broad class of degenerate hypergraphs $F$, we present near-optimal upper bounds for $\mathrm{ex}(n, (t+1) F)$ when $n$ is sufficiently large and $t$ lies in intervals $\left[0, \frac{\varepsilon \cdot \mathrm{ex}(n,F)}{n^{r-1}}\right]$, $\left[\frac{\mathrm{ex}(n,F)}{\varepsilon n^{r-1}}, \varepsilon n \right]$, and $\left[ (1-\varepsilon)\frac{n}{v(F)}, \frac{n}{v(F)} \right]$, where $\varepsilon > 0$ is a constant depending only on $F$.
Our results reveal very different structures for extremal constructions across the three intervals, and we provide characterizations of extremal constructions within the first interval. 
Additionally, we characterize extremal constructions within the second interval for graphs.
Our proof for the first interval also applies to a special class of nondegenerate hypergraphs, including those with undetermined Tur{\' a}n densities, partially improving a result in~\cite{HLLYZ23}.
%
\end{abstract}
\section{Introduction}\label{SEC:Introduction}
\subsection{Motivation}\label{SUBSEC:Motivation}
Fix an integer $r\ge 2$, an $r$-graph $\mathcal{H}$ is a collection of $r$-subsets of some finite set $V$. We identify a hypergraph $\mathcal{H}$ with its edge set and use $V(\mathcal{H})$ to denote its vertex set. The size of $V(\mathcal{H})$ is denoted by $v(\mathcal{H})$. 
Given a vertex $v\in V(\mathcal{H})$, 
the \textbf{degree} $d_{\mathcal{H}}(v)$ of $v$ in $\mathcal{H}$ is the number of edges in $\mathcal{H}$ containing $v$.
We use $\delta(\mathcal{H})$, $\Delta(\mathcal{H})$, and $d(\mathcal{H})$ to denote the \textbf{minimum degree}, the \textbf{maximum degree}, and the \textbf{average degree} of $\mathcal{H}$, respectively.
We will omit the subscript $\mathcal{H}$ if it is clear from the context.

Given two $r$-graphs $F$ and $\mathcal{H}$, 
the \textbf{$F$-matching number} $\nu(F, \mathcal{H})$ of $\mathcal{H}$ is the maximum of vertex-disjoint copies of $F$ in $\mathcal{H}$. 
This is an extension of the matching number $\nu(\mathcal{H})$ of $\mathcal{H}$ as $\nu(\mathcal{H}) = \nu(K_r^r, \mathcal{H})$, where $K_{\ell}^r$ denotes the complete $r$-graph with $\ell$ vertices. 
The study of the following problem  encompasses several central topics in Extremal Combinatorics.
Given an $r$-graph $F$ and integers $n, t\in \mathbb{N}$:
\begin{align}\label{PROB:F-matching}
    \textit{What constraints on an $n$-vertex $r$-graph $\mathcal{H}$ force it to satisfy $\nu(F, \mathcal{H}) \ge t+1$?} \tag{$\ast$}
\end{align}
There are many classical results on this topic from the minimum degree perspective. 
For example, Corr\'{a}di and Hajnal~\cite{CH63} proved the following result for $K_3$. 

\begin{theorem}[Corr\'{a}di--Hajnal~\cite{CH63}]\label{THM:CH-min-deg}
    Suppose that $n, t \in \mathbb{N}$ are integers with $t\le n/3$. 
    Then every $n$-vertex graph $G$ with $\delta(G) \ge t + \left\lfloor \frac{n-t}{2} \right\rfloor$ contains at least $t$ vertex-disjoint copies of $K_3$, i.e.
    \begin{align*}
        \delta(G) \ge t + \left\lfloor \frac{n-t}{2} \right\rfloor 
        \quad\Rightarrow\quad
        \nu(K_3, G) \ge t. 
    \end{align*}
\end{theorem}
Theorem~\ref{THM:CH-min-deg} was extended to all complete graphs in 
the classical Hajnal--Szemer\'{e}di Theorem~\cite{HS70}, which implies that 
for all integers $n\ge \ell \ge 2$, $t \le \lfloor n/(\ell+1) \rfloor$, and for every $n$-vertex graph $G$,
\begin{align*}
    \delta(G) \ge  t + \left\lfloor \frac{\ell-1}{\ell}(n-t) \right\rfloor 
        \quad\Rightarrow\quad
        \nu(K_{\ell+1}, G) \ge t. 
\end{align*}
There are many further extensions of Theorem~\ref{THM:CH-min-deg} to general graphs (see e.g.~\cite{AY96,KSS01}) and hypergraphs (see e.g.~\cite{LS07}). 
We refer the reader to a survey~\cite{KO09} by K\"{u}hn and Osthus for further related results.  

Problem~\eqref{PROB:F-matching} becomes much harder when considered from the average degree perspective. 
For $r=2$ and $F = K_{2}$, 
the celebrated Erd\H{o}s--Gallai Theorem~\cite{EG59} states that for
all integers $n, \ell \in \mathbb{N}$ with $t+1 \le n/2$ and for every $n$-vertex graph $G$,
\begin{align*}
|G| > \max \left\{\binom{2t+1}{2}, \binom{n}{2}-\binom{n-t}{2}\right\}
\quad\Rightarrow\quad \nu(G) \ge t+1.      
\end{align*}
%
Extending the Erd\H{o}s--Gallai Theorem to $r$-graphs for $r\ge 3$ is a major open problem in Extremal Set Theory.  Despite substantial effort, the following conjecture of Erd{\H o}s is still open in general (see e.g. \cite{Frankl13,Frankl17A,Frankl17B, HLS12} for some recent progress on this topic). 

\begin{conjecture}[Erd\H{o}s~\cite{Erdos65}]\label{CONJ:Erdos-matching}
Suppose that $n, t, r \in \mathbb{N}$ satisfy $r\ge 3$ and $t+1 \le n/r$. 
Then for every $n$-vertex $r$-graph $\mathcal{H}$,
\begin{align*}
    |\mathcal{H}|> \max\left\{\binom{r(t+1)-1}{r}, \binom{n}{r}-\binom{n-t}{r}\right\}
    \quad\Rightarrow\quad \nu(\mathcal{H}) \ge t+1. 
\end{align*}
\end{conjecture}
In the special case where $t = 1$, Problem~\eqref{PROB:F-matching} exhibits a close relationship with the Tur\'{a}n problem. 
Fix an $r$-graph $F$, we say an $r$-graph $\mathcal{H}$ is \textbf{$F$-free} if $\nu(F,\mathcal{H}) = 0$. 
In other words, $\mathcal{H}$ does not contains $F$ as a subgraph. 
%
%
The \textbf{Tur\'{a}n number} $\mathrm{ex}(n,F)$ of $F$ is the maximum number of edges in an $F$-free $r$-graph on $n$ vertices.
The \textbf{Tur\'{a}n density} of $F$ is defined as $\pi(F) := \lim_{n\to \infty}\mathrm{ex}(n,F)/\binom{n}{r}$,
the existence of the limit follows from a simple averaging argument of Katona, Nemetz, and Simonovits~\cite{KNS64}. 
Let $\mathrm{EX}(n,F)$ denote the collection of all $n$-vertex $F$-free $r$-graphs with exactly $\mathrm{ex}(n,F)$ edges. 
Members in $\mathrm{EX}(n,F)$ are called the \textbf{extremal constructions} of $F$. 
The study of $\mathrm{ex}(n,F)$ and $\mathrm{EX}(n,F)$ is a central topic in Extremal Combinatorics. 
The celebrated Tur\'{a}n Theorem states that $\mathrm{EX}(n,K_{\ell+1}) = \{T(n,\ell)\}$ for all $n \ge \ell \ge 2$, where $T(n,\ell)$ denote the balanced $\ell$-partite graph with $n$ vertices. 

We call an $r$-graph $F$ \textbf{nondegenerate} if $\pi(F) > 0$, and \textbf{degenerate} if $\pi(F) = 0$. 
By a theorem of Erd{\H o}s~\cite{Erdos64}, an $r$-graph $F$ is degenerate iff it is \textbf{$r$-partite}, i.e. the vertex set of $F$ can be partitioned into $r$ parts such that every edge contains exactly one vertex from each part. 
The celebrated Erd\H{o}s--Stone--Simonovits Theorem~\cite{ES66,ES46} provides a satisfying bound for $\mathrm{ex}(n,F)$ when $F$ is a nondegenerate graph. 
However, determining $\mathrm{ex}(n,F)$ (even asymptotically) for degenerate graphs and $r$-graphs with $r\ge 3$ remains a challenging problem in general. 
We refer the reader to surveys~\cite{FS13,KE11} for more related results. 

Our focus in this work lies in exploring average degree (i.e. edge density) constraints that force an $r$-graph to have large $F$-matching number, especially when $F$ is a degenerate $r$-graph. 
Since our results are closely related to the Tur\'{a}n problem of $F$, 
we abuse the use of notation by letting $\mathrm{ex}\left(n, (t+1)F\right)$ denote the maximum number of edges in an $n$-vertex $r$-graph $\mathcal{H}$ with $\nu(F, \mathcal{H}) < t+1$. 
Let us introduce some definitions and known results before stating our results. 

Given two $r$-graphs $\mathcal{G}$ and $\mathcal{H}$, we use $\mathcal{G}\sqcup \mathcal{H}$ to denote the vertex-disjoint union of $\mathcal{G}$ and $\mathcal{H}$. 
The \textbf{join} $\mathcal{G} \uproduct \mathcal{H}$ of $\mathcal{G}$ and $\mathcal{H}$ is the $r$-graph obtained from $\mathcal{G}\sqcup \mathcal{H}$ by adding all $r$-sets that have nonempty intersection with both $V(\mathcal{G})$ and $V(\mathcal{H})$. 
For simplicity, we define 
the join of an $r$-graph $\mathcal{H}$ and a family $\mathcal{F}$ of $r$-graphs as 
$\mathcal{H} \uproduct \mathcal{F} := \left\{\mathcal{H} \uproduct \mathcal{G} \colon \mathcal{G} \in \mathcal{F}\right\}$. 
Similarly, let $\mathcal{H} \sqcup \mathcal{F} := \left\{\mathcal{H} \sqcup \mathcal{G} \colon \mathcal{G} \in \mathcal{F}\right\}$

In~\cite{Erdos62}, Erd\H{o}s considered $\mathrm{ex}(n,(t+1)K_3)$ and proved the following result. 

\begin{theorem}[Erd\H{o}s~\cite{Erdos62}]\label{THM:Erdos-disjoint-triangle}
    Suppose that $n, t\in \mathbb{N}$ are integers satisfying $t\le \sqrt{n/400}$. 
    Then $\mathrm{EX}\left(n, (t+1)K_3\right) = \{K_{t} \uproduct T(n-t,2)\}$. 
\end{theorem}

Later, Moon~\cite{Moon68} extended it to all complete graphs. 

\begin{theorem}[Moon~\cite{Moon68}]\label{THM:Moon-disjoint-clique}
    Suppose that $n, t, \ell\in \mathbb{N}$ are integers satisfying $\ell \ge 2$, $t\le \frac{2n-3\ell^2+2\ell}{\ell^3+2\ell^2+\ell+1}$, and $\ell \mid (n-t)$. 
    Then $\mathrm{EX}\left(n, (t+1)K_{\ell+1}\right) = \left\{K_{t}\uproduct T(n-t,\ell)\right\}$. 
\end{theorem}

It is worth mentioning that Simonovits~\cite{SI68} also proved Theorem~\ref{THM:Moon-disjoint-clique} when $t\ge 1$ and $\ell \ge 2$ are fixed and $n$ is sufficiently large. 

Extending Theorem~\ref{THM:Moon-disjoint-clique} to larger $t$ becomes much more challenging.
Indeed, a full density version of the Corr\'{a}di--Hajnal Theorem was obtained only recently by Allen, B\"{o}ttcher, Hladk\'{y}, and Piguet~\cite{ABHP15} when $n$ is large. 
Their results show that, interestingly, there are four different extremal constructions for four different regimes of $t$, and the construction $K_{t} \uproduct T(n-t, 2)$ is extremal only for $t \le \frac{2n-6}{9}$. For the other three extremal constructions, we refer the reader to their paper for details.
A full-density version of the Corr\'{a}di--Hajnal Theorem for larger complete graphs still seems out of reach, and it appears that there are even no conjectures for the extremal constructions in general (see remarks in the last section of~\cite{ABHP15}). 

In~\cite{HLLYZ23}, the authors initiated a systematic study of  $\mathrm{ex}\left(n, (t+1)F\right)$ for nondegenerate hypergraphs $F$. 
It was showed in~\cite{HLLYZ23} that if a nondegenerate hypergraphs $F$ satisfies the `Smoothness' and `Boundedness' constraints (refer to~\cite{HLLYZ23} for definitions),  then for sufficiently large $n$ and all $t \le \varepsilon n$, where $\varepsilon > 0$ is a constant depending only on $F$, we have $\mathrm{EX}\left(n, (t+1)F\right) = \left\{K_{t}\uproduct \mathrm{EX}\left(n, F\right)\right\}$.

To complement the work of~\cite{HLLYZ23}, we initiate a systematic study of $\mathrm{ex}\left(n, (t+1)F\right)$ for degenerate hypergraphs $F$ in the present paper. 
Let us introduce some definitions related to the lower bound first. 

\begin{figure}[htbp]
\centering
\tikzset{every picture/.style={line width=0.8pt}} 
\begin{tikzpicture}[x=0.75pt,y=0.75pt,yscale=-0.8,xscale=0.8]

\draw  [fill={rgb, 255:red, 0; green, 0; blue, 0 }  ,fill opacity=1 ] (315.45,37.65) .. controls (315.45,37.15) and (315.8,36.75) .. (316.25,36.75) .. controls (316.69,36.75) and (317.05,37.15) .. (317.05,37.65) .. controls (317.05,38.15) and (316.69,38.55) .. (316.25,38.55) .. controls (315.8,38.55) and (315.45,38.15) .. (315.45,37.65) -- cycle ;
\draw  [fill={rgb, 255:red, 0; green, 0; blue, 0 }  ,fill opacity=1 ] (326.1,37.65) .. controls (326.1,37.15) and (326.46,36.75) .. (326.9,36.75) .. controls (327.34,36.75) and (327.7,37.15) .. (327.7,37.65) .. controls (327.7,38.15) and (327.34,38.55) .. (326.9,38.55) .. controls (326.46,38.55) and (326.1,38.15) .. (326.1,37.65) -- cycle ;
\draw  [fill={rgb, 255:red, 0; green, 0; blue, 0 }  ,fill opacity=1 ] (320.77,27.07) .. controls (320.77,26.58) and (321.13,26.17) .. (321.57,26.17) .. controls (322.01,26.17) and (322.37,26.58) .. (322.37,27.07) .. controls (322.37,27.57) and (322.01,27.97) .. (321.57,27.97) .. controls (321.13,27.97) and (320.77,27.57) .. (320.77,27.07) -- cycle ;
\draw  [line width=0.75]  (321.57,27.07) -- (326.9,37.65) -- (316.25,37.65) -- cycle ;
\draw  [fill={rgb, 255:red, 0; green, 0; blue, 0 }  ,fill opacity=1 ] (338.2,37.65) .. controls (338.2,37.15) and (338.55,36.75) .. (339,36.75) .. controls (339.44,36.75) and (339.8,37.15) .. (339.8,37.65) .. controls (339.8,38.15) and (339.44,38.55) .. (339,38.55) .. controls (338.55,38.55) and (338.2,38.15) .. (338.2,37.65) -- cycle ;
\draw  [fill={rgb, 255:red, 0; green, 0; blue, 0 }  ,fill opacity=1 ] (348.85,37.65) .. controls (348.85,37.15) and (349.2,36.75) .. (349.65,36.75) .. controls (350.09,36.75) and (350.45,37.15) .. (350.45,37.65) .. controls (350.45,38.15) and (350.09,38.55) .. (349.65,38.55) .. controls (349.2,38.55) and (348.85,38.15) .. (348.85,37.65) -- cycle ;
\draw  [fill={rgb, 255:red, 0; green, 0; blue, 0 }  ,fill opacity=1 ] (343.52,27.07) .. controls (343.52,26.58) and (343.88,26.17) .. (344.32,26.17) .. controls (344.76,26.17) and (345.12,26.58) .. (345.12,27.07) .. controls (345.12,27.57) and (344.76,27.97) .. (344.32,27.97) .. controls (343.88,27.97) and (343.52,27.57) .. (343.52,27.07) -- cycle ;
\draw  [line width=0.75]  (344.32,27.07) -- (349.65,37.65) -- (339,37.65) -- cycle ;
\draw  [fill={rgb, 255:red, 0; green, 0; blue, 0 }  ,fill opacity=1 ] (315.45,59.83) .. controls (315.45,59.33) and (315.8,58.93) .. (316.25,58.93) .. controls (316.69,58.93) and (317.05,59.33) .. (317.05,59.83) .. controls (317.05,60.32) and (316.69,60.73) .. (316.25,60.73) .. controls (315.8,60.73) and (315.45,60.32) .. (315.45,59.83) -- cycle ;
\draw  [fill={rgb, 255:red, 0; green, 0; blue, 0 }  ,fill opacity=1 ] (326.1,59.83) .. controls (326.1,59.33) and (326.46,58.93) .. (326.9,58.93) .. controls (327.34,58.93) and (327.7,59.33) .. (327.7,59.83) .. controls (327.7,60.32) and (327.34,60.73) .. (326.9,60.73) .. controls (326.46,60.73) and (326.1,60.32) .. (326.1,59.83) -- cycle ;
\draw  [fill={rgb, 255:red, 0; green, 0; blue, 0 }  ,fill opacity=1 ] (320.77,49.25) .. controls (320.77,48.75) and (321.13,48.35) .. (321.57,48.35) .. controls (322.01,48.35) and (322.37,48.75) .. (322.37,49.25) .. controls (322.37,49.75) and (322.01,50.15) .. (321.57,50.15) .. controls (321.13,50.15) and (320.77,49.75) .. (320.77,49.25) -- cycle ;
\draw  [line width=0.75]  (321.57,49.25) -- (326.9,59.83) -- (316.25,59.83) -- cycle ;
\draw  [fill={rgb, 255:red, 0; green, 0; blue, 0 }  ,fill opacity=1 ] (315.45,119.27) .. controls (315.45,118.77) and (315.8,118.37) .. (316.25,118.37) .. controls (316.69,118.37) and (317.05,118.77) .. (317.05,119.27) .. controls (317.05,119.77) and (316.69,120.17) .. (316.25,120.17) .. controls (315.8,120.17) and (315.45,119.77) .. (315.45,119.27) -- cycle ;
\draw  [fill={rgb, 255:red, 0; green, 0; blue, 0 }  ,fill opacity=1 ] (326.1,119.27) .. controls (326.1,118.77) and (326.46,118.37) .. (326.9,118.37) .. controls (327.34,118.37) and (327.7,118.77) .. (327.7,119.27) .. controls (327.7,119.77) and (327.34,120.17) .. (326.9,120.17) .. controls (326.46,120.17) and (326.1,119.77) .. (326.1,119.27) -- cycle ;
\draw  [fill={rgb, 255:red, 0; green, 0; blue, 0 }  ,fill opacity=1 ] (320.77,108.69) .. controls (320.77,108.2) and (321.13,107.79) .. (321.57,107.79) .. controls (322.01,107.79) and (322.37,108.2) .. (322.37,108.69) .. controls (322.37,109.19) and (322.01,109.59) .. (321.57,109.59) .. controls (321.13,109.59) and (320.77,109.19) .. (320.77,108.69) -- cycle ;
\draw  [line width=0.75]  (321.57,108.69) -- (326.9,119.27) -- (316.25,119.27) -- cycle ;
\draw  [fill={rgb, 255:red, 0; green, 0; blue, 0 }  ,fill opacity=1 ] (338.2,119.27) .. controls (338.2,118.77) and (338.55,118.37) .. (339,118.37) .. controls (339.44,118.37) and (339.8,118.77) .. (339.8,119.27) .. controls (339.8,119.77) and (339.44,120.17) .. (339,120.17) .. controls (338.55,120.17) and (338.2,119.77) .. (338.2,119.27) -- cycle ;
\draw  [fill={rgb, 255:red, 0; green, 0; blue, 0 }  ,fill opacity=1 ] (348.85,119.27) .. controls (348.85,118.77) and (349.2,118.37) .. (349.65,118.37) .. controls (350.09,118.37) and (350.45,118.77) .. (350.45,119.27) .. controls (350.45,119.77) and (350.09,120.17) .. (349.65,120.17) .. controls (349.2,120.17) and (348.85,119.77) .. (348.85,119.27) -- cycle ;
\draw  [fill={rgb, 255:red, 0; green, 0; blue, 0 }  ,fill opacity=1 ] (343.52,108.69) .. controls (343.52,108.2) and (343.88,107.79) .. (344.32,107.79) .. controls (344.76,107.79) and (345.12,108.2) .. (345.12,108.69) .. controls (345.12,109.19) and (344.76,109.59) .. (344.32,109.59) .. controls (343.88,109.59) and (343.52,109.19) .. (343.52,108.69) -- cycle ;
\draw  [line width=0.75]  (344.32,108.69) -- (349.65,119.27) -- (339,119.27) -- cycle ;
\draw  [fill={rgb, 255:red, 0; green, 0; blue, 0 }  ,fill opacity=1 ] (338.2,59.83) .. controls (338.2,59.33) and (338.55,58.93) .. (339,58.93) .. controls (339.44,58.93) and (339.8,59.33) .. (339.8,59.83) .. controls (339.8,60.32) and (339.44,60.73) .. (339,60.73) .. controls (338.55,60.73) and (338.2,60.32) .. (338.2,59.83) -- cycle ;
\draw  [fill={rgb, 255:red, 0; green, 0; blue, 0 }  ,fill opacity=1 ] (348.85,59.83) .. controls (348.85,59.33) and (349.2,58.93) .. (349.65,58.93) .. controls (350.09,58.93) and (350.45,59.33) .. (350.45,59.83) .. controls (350.45,60.32) and (350.09,60.73) .. (349.65,60.73) .. controls (349.2,60.73) and (348.85,60.32) .. (348.85,59.83) -- cycle ;
\draw  [fill={rgb, 255:red, 0; green, 0; blue, 0 }  ,fill opacity=1 ] (343.52,49.25) .. controls (343.52,48.75) and (343.88,48.35) .. (344.32,48.35) .. controls (344.76,48.35) and (345.12,48.75) .. (345.12,49.25) .. controls (345.12,49.75) and (344.76,50.15) .. (344.32,50.15) .. controls (343.88,50.15) and (343.52,49.75) .. (343.52,49.25) -- cycle ;
\draw  [line width=0.75]  (344.32,49.25) -- (349.65,59.83) -- (339,59.83) -- cycle ;
\draw  [fill={rgb, 255:red, 0; green, 0; blue, 0 }  ,fill opacity=1 ] (332.44,76.67) .. controls (332.44,76.43) and (332.62,76.23) .. (332.84,76.23) .. controls (333.05,76.23) and (333.23,76.43) .. (333.23,76.67) .. controls (333.23,76.92) and (333.05,77.11) .. (332.84,77.11) .. controls (332.62,77.11) and (332.44,76.92) .. (332.44,76.67) -- cycle ;
\draw  [fill={rgb, 255:red, 0; green, 0; blue, 0 }  ,fill opacity=1 ] (332.42,87.31) .. controls (332.42,87.07) and (332.59,86.87) .. (332.81,86.87) .. controls (333.03,86.87) and (333.2,87.07) .. (333.2,87.31) .. controls (333.2,87.56) and (333.03,87.76) .. (332.81,87.76) .. controls (332.59,87.76) and (332.42,87.56) .. (332.42,87.31) -- cycle ;
\draw  [fill={rgb, 255:red, 0; green, 0; blue, 0 }  ,fill opacity=1 ] (332.44,97.96) .. controls (332.44,97.71) and (332.62,97.52) .. (332.84,97.52) .. controls (333.05,97.52) and (333.23,97.71) .. (333.23,97.96) .. controls (333.23,98.2) and (333.05,98.4) .. (332.84,98.4) .. controls (332.62,98.4) and (332.44,98.2) .. (332.44,97.96) -- cycle ;
\draw    (0,190) -- (640,190.08) ;
\draw [shift={(640,190.08)}, rotate = 180.01] [color={rgb, 255:red, 0; green, 0; blue, 0 }  ][line width=0.75]    (10.93,-3.29) .. controls (6.95,-1.4) and (3.31,-0.3) .. (0,0) .. controls (3.31,0.3) and (6.95,1.4) .. (10.93,3.29)   ;
\draw [line width=0.75]    (0,190) -- (0,182) ;
\draw [line width=0.75]    (149.56,190) -- (149.56,182) ;
\draw [line width=0.75]    (400.56,190) -- (400.56,182) ;
\draw [line width=0.75]    (600,190) -- (600,182) ;
%
\draw [line width=3]    (0,190) -- (30,190.32) ;
\draw [line width=3]    (279.33,190) -- (309.02,190.32) ;
\draw [line width=3]    (570.02,190.32) -- (600,190.08) ;
\draw    (10,175.84) -- (25,151.33) ;
\draw [shift={(26,151.33)}, rotate = 131.55] [fill={rgb, 255:red, 0; green, 0; blue, 0 }  ][line width=0.08]  [draw opacity=0] (5.36,-2.57) -- (0,0) -- (5.36,2.57) -- cycle    ;
\draw    (294.56,175.08) -- (294.56,146.08) ;
\draw [shift={(294.56,143.08)}, rotate = 90] [fill={rgb, 255:red, 0; green, 0; blue, 0 }  ][line width=0.08]  [draw opacity=0] (5.36,-2.57) -- (0,0) -- (5.36,2.57) -- cycle    ;
\draw    (580,175.84) -- (560,150.08) ;
\draw [shift={(559,150.08)}, rotate = 48.24] [fill={rgb, 255:red, 0; green, 0; blue, 0 }  ][line width=0.08]  [draw opacity=0] (5.36,-2.57) -- (0,0) -- (5.36,2.57) -- cycle    ;
\draw  [fill={rgb, 255:red, 155; green, 155; blue, 155 }  ,fill opacity=1 ] (15.33,78.5) .. controls (15.33,57.79) and (22.05,41) .. (30.33,41) .. controls (38.62,41) and (45.33,57.79) .. (45.33,78.5) .. controls (45.33,99.21) and (38.62,116) .. (30.33,116) .. controls (22.05,116) and (15.33,99.21) .. (15.33,78.5) -- cycle ;
\draw [fill={rgb, 255:red, 155; green, 155; blue, 155 }  ,fill opacity=0.3 ]  (74.33,77.5) .. controls (74.33,38.56) and (92.69,7) .. (115.33,7) .. controls (137.98,7) and (156.33,38.56) .. (156.33,77.5) .. controls (156.33,116.44) and (137.98,148) .. (115.33,148) .. controls (92.69,148) and (74.33,116.44) .. (74.33,77.5) -- cycle ;
\draw    (37.02,44.32) -- (97.24,13.71) ;
\draw    (42.33,57) -- (97.24,13.71) ;
\draw    (42.33,57) -- (79.24,42.71) ;
\draw    (44.51,67.31) -- (79.24,42.71) ;
\draw    (44.51,67.31) -- (73.84,73.31) ;
\draw    (45.33,81) -- (73.84,73.31) ;
\draw    (45.33,81) -- (76.51,97.97) ;
\draw    (44.33,93.67) -- (76.51,97.97) ;
\draw    (44.33,93.67) -- (92.24,135.71) ;
\draw    (39.02,110.32) -- (92.24,135.71) ;
\draw  [fill={rgb, 255:red, 155; green, 155; blue, 155 }  ,fill opacity=1 ] (227.33,77) .. controls (227.33,50.49) and (238.08,29) .. (251.33,29) .. controls (264.59,29) and (275.33,50.49) .. (275.33,77) .. controls (275.33,103.51) and (264.59,125) .. (251.33,125) .. controls (238.08,125) and (227.33,103.51) .. (227.33,77) -- cycle ;
\draw   (303.33,75.8) .. controls (303.33,38.8) and (316.76,8.8) .. (333.33,8.8) .. controls (349.9,8.8) and (363.33,38.8) .. (363.33,75.8) .. controls (363.33,112.81) and (349.9,142.8) .. (333.33,142.8) .. controls (316.76,142.8) and (303.33,112.81) .. (303.33,75.8) -- cycle ;
\draw    (262.17,33.31) -- (324.45,11.17) ;
\draw    (270.33,48) -- (324.45,11.17) ;
\draw    (270.33,48) -- (308.45,38.17) ;
\draw    (274.51,63.31) -- (308.45,38.17) ;
\draw    (274.51,63.31) -- (303.33,75.5) ;
\draw    (275.33,86) -- (303.33,75.8) ;
\draw    (275.33,86) -- (309.45,115.17) ;
\draw    (271.45,102.17) -- (309.45,115.17) ;
\draw    (271.45,102.17) -- (321.45,137.17) ;
\draw    (263.51,117.64) -- (321.45,137.17) ;
\draw  [fill={rgb, 255:red, 155; green, 155; blue, 155 }  ,fill opacity=1] 
(427.52,75.55) .. controls (427.52,37.99) and (446.54,7.55) .. (470.02,7.55) .. controls (493.49,7.55) and (512.52,37.99) .. (512.52,75.55) .. controls (512.52,113.1) and (493.49,143.55) .. (470.02,143.55) .. controls (446.54,143.55) and (427.52,113.1) .. (427.52,75.55) -- cycle ;
\draw  [fill={rgb, 255:red, 155; green, 155; blue, 155 }  ,fill opacity=0.3 ] (539.83,74.78) .. controls (539.83,56.12) and (548.45,40.99) .. (559.1,40.99) .. controls (569.74,40.99) and (578.37,56.12) .. (578.37,74.78) .. controls (578.37,93.44) and (569.74,108.57) .. (559.1,108.57) .. controls (548.45,108.57) and (539.83,93.44) .. (539.83,74.78) -- cycle ;

\draw (17.33,72.33) node [anchor=north west][inner sep=0.75pt]  [font=\scriptsize] [align=left] {$K_{t}$};
\draw (80,72.67) node [anchor=north west][inner sep=0.75pt]  [font=\scriptsize] [align=left] {$K_{2,3}$-free};
\draw (231.67,70) node [anchor=north west][inner sep=0.75pt]  [font=\scriptsize] [align=left] {$K_{2t+1}$};
\draw (449.31,72.32) node [anchor=north west][inner sep=0.75pt]  [font=\scriptsize] [align=left] {$K_{5t+4}$};
\draw (532,116) node [anchor=north west][inner sep=0.75pt]  [font=\scriptsize] [align=left] {$K_{2,3}$-free};
\draw (-5,200) node [anchor=north west][inner sep=0.75pt]   [align=left] {\footnotesize{$0$}};
\draw (630,200) node [anchor=north west][inner sep=0.75pt]   [align=left] {\footnotesize{$t$}};
\draw (115,194) node [anchor=north west][inner sep=0.75pt]   [align=left] {\footnotesize{$\sim \sqrt{\frac{n}{2}}$}};
\draw (375,195) node [anchor=north west][inner sep=0.75pt]   [align=left] {\footnotesize{$\sim \frac{4n}{29}$}};
\draw (593,198) node [anchor=north west][inner sep=0.75pt]   [align=left] {\footnotesize{$\frac{n}{5}$}};
%
\end{tikzpicture}
\caption{Three conjectured asymptotic extremal constructions for $\mathrm{ex}(n,(t+1)K_{2,3})$ for $t$ in different intervals. We highlighted the three small intervals in which the constructions are proved to be asymptotically optimal.}
\label{fig:K23-3-segments}
\end{figure}
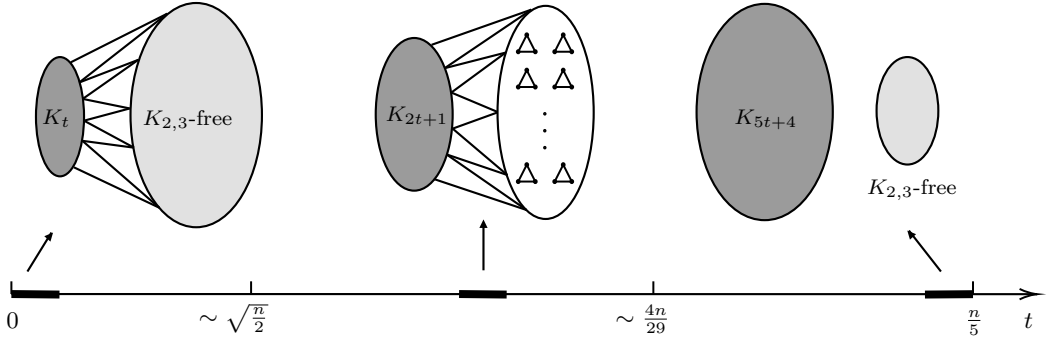

Given an $r$-graph $F$ and an integer $s \in [0, v(F)]$, 
let  
\begin{align*}
    F[s] := \left\{F[S] \colon S\subset V(F) \text{ and } |S| = s\right\},     
\end{align*}
where $F[S]$ denotes the induced subgraph of $F$ on $S$. 
The \textbf{covering number} $\tau(F)$ of $F$ is defined as  
\begin{align*}
    \tau(F) 
    := \min \left\{|S| \colon \text{$S \subset V(F)$ and $S\cap e\neq \emptyset$ for all $e\in F$} \right\}.
\end{align*}
For an $m$-vertex $r$-graph $F$ with covering number $\tau$, let 
\begin{enumerate}[label=(\roman*)]
    \item $\mathcal{G}_1(n,t,F) := K_t^{r} \uproduct \mathrm{EX}(n-t, F)$, 
    \item $\mathcal{G}_2(n,t,F) := K_{\tau(t+1)-1}^r \uproduct \mathrm{EX}\left(n-\tau(t+1)+1, F[m-\tau+1]\right)$, 
    \item $\mathcal{G}_3(n,t,F) := K_{m(t+1)-1}^r \sqcup \mathrm{EX}\left(n-m(t+1)+1, F\right)$. 
\end{enumerate}
Let $g_i(n,t,F)$ denote the size of a member in $\mathcal{G}_i(n,t,F)$ for $i\in \{1,2,3\}$. 
Simple calculations show that 
\begin{align*}
    g_{1}(n,t,F) 
    & = \binom{n}{r}-\binom{n-t}{r}+\mathrm{ex}(n-t,F), \\
    g_{2}(n,t,F) 
    & = \binom{n}{r}-\binom{n-\tau(t+1)+1}{r}+\mathrm{ex}\left(n-\tau(t+1)+1,F[m-\tau+1]\right), \quad\text{and}\\
    g_{3}(n,t,F) 
    & = \binom{m(t+1)-1}{r}+ \mathrm{ex}\left(n-m(t+1)+1,F\right). 
\end{align*}
An easy observation is that the every member in the families defined above has $n$ vertices and is $(t+1)F$-free. 
Therefore, we have the following lower bound for $\mathrm{ex}\left(n, (t+1)F\right)$. 
\begin{proposition}\label{PROP:three-lower-bounds}
    Suppose that $F$ is an $m$-vertex $r$-graph with covering number $\tau$. 
    Then 
    \begin{align*}
        \mathrm{ex}\left(n, (t+1)F\right)
        \ge \max\left\{g_i(n,t,F) \colon i\in \{1,2,3\}\right\}
    \end{align*}
    for all $t \ge 1$ and $n \ge mt$. 
\end{proposition}
Our main results, which will be presented shortly, show that for an $r$-partite $r$-graph $F$ with part sizes $s_r \ge \cdots \ge s_1 = \tau(F)$, the constructions $\mathcal{G}_1(n,t,F)$, $\mathcal{G}_2(n,t,F)$, and $\mathcal{G}_3(n,t,F)$ are asymptotically optimal for $\mathrm{ex}\left(n, (t+1)F\right)$ when $n$ is large, and $t$ lies in intervals $\left[0, \frac{\varepsilon \cdot \mathrm{ex}(n,F)}{n^{r-1}}\right]$, $\left[\frac{\mathrm{ex}(n,F)}{\varepsilon n^{r-1}}, \varepsilon n \right]$, and $\left[ (1-\varepsilon)\frac{n}{v(F)}, \frac{n}{v(F)} \right]$, respectively (see Figure~\ref{fig:K23-3-segments}).  
Here $\varepsilon > 0$ is a constant depending only on $F$.
We conjecture that when $F$ is the balanced complete $r$-partite $r$-graph $K_{s,\ldots, s}^{r}$,  the asymptotic behavior of $\mathrm{ex}\left(n, (t+1)F\right)$ is governed by the three constructions defined above for all feasible $t$. 
For general degenerate $r$-graphs, there are could be more extremal constructions (see Section~\ref{SEC:Remarks} for more details). 
We hope that our results together with results in~\cite{HLLYZ23} may provide some insights towards a comprehensive generalization of the density version of the  Corr\'{a}di--Hajnal Theorem for all hypergraphs. 

\subsection{The first interval}
In this subsection, we state the result for $\mathrm{ex}\left(n, (t+1)F\right)$ when $t$ lies in $\left[0, \frac{\varepsilon \cdot \mathrm{ex}(n,F)}{n^{r-1}}\right]$. 
This result shares similarities with the main result in~\cite{HLLYZ23}, but here we only need the following weaker property of $F$.

\begin{definition}[\textbf{Boundedness}]
      Let $c_1, c_2 > 0$ be two real numbers. 
      An $r$-graph $F$ is $(c_1, c_2)$-bounded\footnote{One could extend the definition of `Boundedness' by allowing $c_1$ and $c_2$ to be functions of $n$, but constant functions are sufficient for our purpose in the present paper.} if the following holds for all sufficiently large $n$. 
      Every $n$-vertex $r$-graph $\mathcal{H}$ with
      \begin{align*}
          \Delta(\mathcal{H}) \ge c_1 \binom{n-1}{r-1} 
          \quad\text{and}\quad
          |\mathcal{H}| \ge (1-c_2) \cdot \mathrm{ex}(n,F)
      \end{align*}
      contains a copy of $F$.  
      We say $F$ is bounded if it is $(c_1, c_2)$-bounded for some $c_1, c_2 > 0$.
\end{definition}
We will show in~\cite{HHLLYZ23c} that many well-studied degenerate hypergraphs such as even cycles $C_{2k}$ for $k \ge 2$, complete bipartite graphs $K_{s,t}$ for $t > s(s-1) \ge 2$, and the expansion of complete bipartite graphs $K_{s,t}^{+}$ for $t > (s-1)!\ge 2$ are bounded. 
\begin{theorem}\label{THM:1st-interval}
    Let $m \ge r \ge 2$ be integers and $F$ be an $r$-graph with $m$ vertices. 
    Suppose that $F$ is $(c_1, c_2)$-bounded for some $c_1 < 1/m$ and $c_2 >0$. 
    Then for sufficiently large $n$,  
        \begin{align*}
            \mathrm{EX}\left(n, (t+1)F\right)
                = K_{t}^r \uproduct \mathrm{EX}(n-t, F)
        \end{align*}
    holds for all integers $t$ satisfying 
    \begin{align*}
        0 \le t 
        \le \min\left\{\frac{\delta \cdot \mathrm{ex}(n,F)}{m\binom{n-1}{r-1}},\  \frac{\delta (n-1)}{8m(r-1)}\right\}, 
        \quad\text{where}\quad 
        \delta := \min\left\{\frac{1}{4}\left(\frac{1}{m}-c_1\right),\ \frac{c_2}{4}\right\}. 
    \end{align*}
\end{theorem}
One interesting consequence of Theorem~\ref{THM:1st-interval} is as follows. 
\begin{definition}[\textbf{Suspensions}]
    Let $r \ge 2$ and $F$ be an $(r-1)$-graph. 
    The suspension $\widehat{F}$ of $F$ is the $r$-graph defined as $\widehat{F} := \left\{\{v\}\cup e \colon e\in F\right\}$, 
    where $v$ is a new vertex not contained in $V(F)$. 
\end{definition}
Observe that the suspension $\widehat{F}$ of every hypergraph $F$ is $(c_1, 1)$-bounded for all $c_1 > \pi(F)$.
Thus we have the following corollary. 

\begin{corollary}\label{CORO:1st-interval-suspension}
    Let $m \ge r \ge 2$ be integers and $F$ be an $(r-1)$-graph on $m-1$ vertices. 
    Suppose that $\pi(F) < 1/m$. 
    Then for sufficiently large $n$,  
    \begin{align*}
        \mathrm{EX}(n, (t+1)\widehat{F})
                = K_{t}^r \uproduct \mathrm{EX}(n-t, \widehat{F})
    \end{align*}
    holds for all integers $t$ satisfying 
    \begin{align*}
        0 \le t 
        \le \min\left\{\frac{1-m\cdot \pi(F)}{5m^2}\frac{\mathrm{ex}(n,F)}{\binom{n-1}{r-1}},\  \frac{1-m\cdot \pi(F)}{40m^2}\frac{n-1}{r-1}\right\}. 
    \end{align*}
\end{corollary}
\textbf{Remarks.}
\begin{itemize}
        \item Note that in Theorem~\ref{THM:1st-interval} we do not require $F$ to be a degenerate hypergraph,  
        so it improves~{\cite[Theorem~1.7]{HLLYZ23}} for nondegenerate $r$-graphs $F$ with $\pi(F) < 1/v(F)$ (since we do not need the `Smoothness' constraint in Theorem~\ref{THM:1st-interval}). 
        Also note that Corollary~\ref{CORO:1st-interval-suspension} improves~{\cite[Theorem~1.7]{HLLYZ23}} for the suspension of nondegenerate $r$-graphs $F$ with $\pi(F) < \frac{1}{v(F)+1}$, and an illustrative example is provided in the subsequent  Corollary~\ref{CORO:1st-generalized-triangle-suspension}. 
        \item For integers $r > b \ge 0$ let $Y_{r,b}$ denote the $r$-graph consistsing of two edges that intersect in exactly $b$ vertices. 
        Gan, Han, Sun and Wang~\cite{GHSW21} (see also~\cite{HSW23}) proved some asymptotic upper bounds for $\mathrm{ex}\left(n, (t+1)Y_{r,b}\right)$ for $r > b \ge 1$. 
        Since $Y_{r,b}$ is the suspension of $Y_{r-1,b-1}$ when $b \ge 1$, Corollary~\ref{CORO:1st-interval-suspension} improves their results for $t$ in the interval as stated in Corollary~\ref{CORO:1st-interval-suspension}. 
\end{itemize}

Let $r \ge 2$ be an integer. 
A notable example in hypergraph Tur\'{a}n problem is the  \textbf{generalized triangle} $T_r$, defined as 
    \begin{align*}
        T_{r} 
        := \left\{\{1,\ldots,r-1, r\}, \{1,\ldots,r-1,r+1\},\{r,r+1,\ldots,2r-1\}\right\}. 
    \end{align*}
The study of $\pi(T_r)$ has a long history (see e.g.~\cite{Mantel07,Bol74,FF83,Sido87,She96,KM04,PI08,NY17,LIU19,LMR23unif}), although the precise value of $\pi(T_r)$ remains undetermined for all $r\ge 7$. 
Determining the $\pi(\widehat{T}_r)$ appears to be even more challenging in general, and the longstanding problem of determining $\pi(\widehat{T}_3)$, i.e. $\pi(K_{4}^{3-})$, remains open (see~\cite{FF84,Mub03}). 

In~\cite{FF89}, Frankl--F\"{u}redi showed that $\pi(T_r) \le \frac{1}{e\binom{r-1}{2}}$, 
which is smaller that $\frac{1}{2r}$ when $r\ge 4$. 
Therefore, by Corollary~\ref{CORO:1st-interval-suspension}\footnote{Together with facts $\mathrm{ex}(n,\widehat{T}_r)) \ge \pi(\widehat{T}_r) \binom{n}{r}$, $\pi(\widehat{T}_r) \le \pi(T_r) \le \frac{1}{e\binom{r-1}{2}}$, and some simple calculations.}, we obtain the following result, which appears to be the first of its kind for a hypergraph with an undetermined Tur\'{a}n density. 

\begin{corollary}\label{CORO:1st-generalized-triangle-suspension}
    Suppose that $r \ge 4$ is an integer. 
    Then for sufficiently large $n$, 
    \begin{align*}
        \mathrm{EX}(n, (t+1)\widehat{T}_r)
                = K_{t}^r \uproduct \mathrm{EX}(n-t, \widehat{T}_t), 
    \end{align*}
    holds for all integers $t$ satisfying 
    \begin{align*}
        0 \le t 
        \le \frac{\left(1-(2r-1)\cdot \pi(\widehat{T}_r)\right)\cdot \pi(\widehat{T}_r)}{5 r (2r-1)^2} n. 
    \end{align*}
\end{corollary}


The proof of Theorem~\ref{THM:1st-interval} is presented in Section~\ref{SEC:proof-1st-interval}. 

\subsection{The second interval}\label{SUBSEC:2nd-segment}
In this subsection, we present the result for $\mathrm{ex}\left(n, (t+1)F\right)$ when $t$ lies in $\left[\frac{\mathrm{ex}(n,F)}{\varepsilon n^{r-1}}, \varepsilon n\right]$. 

\begin{theorem}\label{THM:2nd-interval}
    Let $r \ge 2$ and $s_r \ge \cdots \ge s_1 \ge 2$ be integers.
    Let $s:= s_1+ \cdots+ s_r$ and $\mathbb{B} := B_{s_1, \ldots, s_r}^{r}$ be an $r$-partite $r$-graph with part sizes $s_1, \ldots, s_r$. 
    For sufficiently large $n$, 
    \begin{align*}
        \mathrm{ex}(n,(t+1)\mathbb{B})
        \le \binom{n}{r} - \binom{n-s_1(t+1)+1}{r} + \mathrm{ex}(n-s_1(t+1)+1, \mathbb{B}). 
    \end{align*}
    holds for all integers $t$ satisfying 
    \begin{align*}
        320es_1s\left(\frac{\mathrm{ex}(n,\mathbb{B})}{\binom{n-1}{r-1}} + 320 es_1s^2 r!\right)
        \le t 
        \le \frac{n}{512 e s_1s^3 r!}. 
    \end{align*}
\end{theorem}

For graphs, we are able to obtain the following tight bound. 

\begin{theorem}\label{THM:2nd-interval-graph}
    Let $s_2\ge s_1 \ge 2$ be integers. 
    Let $s:=s_1+s_2$ and $\mathbb{B}:= B_{s_1, s_2}$ be an $s_1$ by $s_2$ bipartite connected graph with $\tau(B) = s_1$. 
    For sufficiently large $n$,  
    \begin{align*}
        \mathrm{ex}\left(n,(t+1)\mathbb{B}\right)
        = \binom{n}{2} - \binom{n-s_1(t+1)+1}{2} + \mathrm{ex}\left(n-s_1(t+1)+1, \mathbb{B}[s_2+1]\right)
    \end{align*}
    holds for all integers $t$ satisfying 
    \begin{align*}
        \max \left\{\sqrt{32s_1sn},\  \frac{12800 e s^5}{s_1}\left(\frac{\mathrm{ex}(n,\mathbb{B})}{n-1} + 288 es_1s^2r!\right)\right\} 
        \le t 
        \le \frac{n}{1024 e s_1s^3}. 
    \end{align*}
\end{theorem}

Since the complete bipartite graph $K_{s_1, s_2}$ with $s_2\ge s_1 \ge 2$ satisfies $\mathrm{ex}(n,K_{s_1,s_2}) \ge \mathrm{ex}(n,K_{2,2}) = (1/2-o(1))n^{3/2}$ (see~\cite{Brown66,ERS66,Fur83,Fur94,Fur96}), we obtain the following corollary of Theorem~\ref{THM:2nd-interval-graph}. 

\begin{corollary}
    Let $s_2\ge s_1 \ge 2$ be integers. 
    Let $s := s_1+s_2$ and $\mathbb{K} :=  K_{s_1,s_2}$. 
    For sufficiently large $n$,  
    \begin{align*}
        \mathrm{ex}\left(n,(t+1)\mathbb{K}\right)
        = \binom{n}{2} - \binom{n-s_1(t+1)+1}{2} + \mathrm{ex}\left(n-s_1(t+1)+1, \mathbb{K}[s_2+1]\right)
    \end{align*}
    holds for all integers $t$ satisfying 
    \begin{align*}
        \frac{12801 e s^5 \cdot \mathrm{ex}(n,\mathbb{K})}{s_1(n-1)}
        \le t 
        \le  \frac{n}{1024 e s_1s^3}.
    \end{align*}
\end{corollary}

The proof of Theorem~\ref{THM:2nd-interval} is presented in Section~\ref{SUBSEC:proof-2nd-hypergraph}. 
The proof of Theorem~\ref{THM:2nd-interval-graph} is presented in Section~\ref{SUBSEC:proof-2nd-graph}. 

\subsection{The third interval}\label{SUBSEC:3rd-segment}
In this subsection, we present the result for $\mathrm{ex}\left(n, (t+1)F\right)$ when $t$ lies in $\left[ (1-\varepsilon)\frac{n}{v(F)}, \frac{n}{v(F)} \right]$. 

We say an $n$-vertex $r$-graph $\mathcal{H}$ is an \textbf{$m$-star} for some $m \le n$ if there exists a set $V_1\subset V(\mathcal{H})$ of size $m$ such that every edge in $\mathcal{H}$ contains at least one vertex in $V_1$. 
Extending the definition of $\mathrm{ex}\left(n,F\right)$, 
we let $\mathrm{ex}_{\mathrm{star}}\left(m,n,F\right)$ denote the maximum number of edges in an $F$-free $r$-uniform $m$-star with $n$ vertices. 
Observe that $\mathrm{ex}_{\mathrm{star}}\left(m,n,F\right) \le \mathrm{ex}\left(n,F\right)$ for all $m \le n$, and   $\mathrm{ex}_{\mathrm{star}}\left(n,n,F\right) = \mathrm{ex}\left(n,F\right)$. 
For the complete $r$-partite $r$-graph $K_{s_1, \ldots, s_r}^{r}$ with part sizes $s_1, \ldots, s_r$, 
it is not hard to obtain the following improved upper bound (when $m$ is small) using Erd{\H o}s' argument in~\cite{Erdos64}. 

\begin{proposition}\label{PROP:star-Turan-problem}
    Suppose that $r \ge 2$ and $s_r \ge \cdots \ge s_1 \ge 1$ are integers.
    Then 
    \begin{align*}
        \mathrm{ex}_{\mathrm{star}}\left(m,n,K_{s_1, \ldots, s_r}^r\right)
        \le \frac{(s_2+\cdots+s_r-r+1)^{\frac{1}{s_1}}}{r}mn^{r-1-\frac{1}{s_1\cdots s_{r-1}}} + \frac{s_1-1}{r}\binom{n}{r-1}. 
    \end{align*}
\end{proposition}

Our main result in this subsection is as follows. 

\begin{theorem}\label{THM:3rd-interval}
Let $r \ge 2$ and $s_r \ge \cdots \ge s_1 \ge 1$ be integers.
Let $s:= s_1+ \cdots+ s_r$ and $\mathbb{B} := B_{s_1, \ldots, s_r}^{r}$ be an $r$-partite $r$-graph with part sizes $s_1, \ldots, s_r$.
For sufficiently large $n$, 
\begin{align*}
    \mathrm{ex}(n,(t+1)\mathbb{B}) 
    \le \binom{s(t+1)-1}{r} 
        + \mathrm{ex}_{\mathrm{star}}\left(n-st, n, s\mathbb{B}\right). 
\end{align*}
holds for all integers $t$ satisfying 
\begin{align*}
    \frac{n}{s} - \frac{n}{16 e^2 r^4 s^2 \prod_{i\in [r]}s_i}
    \le t 
    \le \frac{n}{s}. 
\end{align*}
\end{theorem}

For graphs, we can prove the following improved bound.

\begin{theorem}\label{THM:3rd-interval-graph}
    Let $s_2 \ge s_1 \ge 2$ be integers. 
    Let $s:= s_1 + s_2$ and $\mathbb{B} := B_{s_1, s_2}$ be an $s_1$ by $s_2$ bipartite graph. 
   For sufficiently large $n$, 
   \begin{align*}
       \mathrm{ex}(n,(t+1)\mathbb{B})
       \le \binom{s(t+1)-1}{2} + \mathrm{ex}\left(n-s(t+1)+1, \mathbb{B}\right) + s_1sn
   \end{align*}
   holds for all integers $t$ satisfying 
    \begin{align*}
        \frac{n}{s} - \frac{n}{65s_1s^2}
        \le t 
        \le \frac{n}{s}. 
    \end{align*}
\end{theorem}

The proof of Theorem~\ref{THM:3rd-interval-graph} is presented in Section~\ref{SUBSEC:proof-3rd-interval-gp}. 
The proof of Theorem~\ref{THM:3rd-interval} is presented in Section~\ref{SUBSEC:proof-3rd-interval-hygp}. 
In the next section, we present some definitions and preliminary results. 
\section{Preliminaries}\label{SEC:prelim}
Let $r\ge 2$ be an integer and $\mathcal{H}$ be an $r$-graph. 
Given a set $W \subset V(\mathcal{H})$, we use $\mathcal{H}[W]$ to denote the \textbf{induced subgraph} of $\mathcal{H}$ on $W$, and use $\mathcal{H}-W$ to denote the induced subgraph of $\mathcal{H}$ on $V(\mathcal{H})\setminus W$. 
Given two disjoint sets $S, T \subset V(\mathcal{H})$, we use $\mathcal{H}[S,T]$ to denote the collection of edges in $\mathcal{H}$ that have nonempty intersection with both $S$ and $T$.  
We use $\overline{\mathcal{H}}$ to denote the \textbf{complement} of $\mathcal{H}$. 

The following simple inequalities will be useful. 

\begin{fact}\label{FACT:binom-inequality-a}
    Suppose that $n \ge r \ge 1$ and $t \in [n]$ are integers. Then 
    \begin{align*}
        \binom{n-t}{r} 
         \le \left(1-\frac{t}{n}\right)^{r} \binom{n}{r} \quad\text{and}\quad 
        \binom{n}{r} 
         \le \left(\frac{n}{n-t-r}\right)^{r}\binom{n-t}{r}. 
    \end{align*}
\end{fact}

\begin{fact}[{\cite[Lemma~3.4]{HLLYZ23}}]\label{LEMMA:binom-inequ-b}
    Suppose that  $n,b,r \ge 1$ are integers satisfying $b \le \frac{n-r^2+1}{r+1}$. Then 
        \begin{align*}
            \binom{n}{r} 
            \le \frac{n^{r}}{r!} 
            \le e \binom{n-b}{r}. 
        \end{align*}    
\end{fact}

\begin{fact}\label{FACT:binom-inequality-b}
    Suppose that $n, \ell, x \ge 0$ are integers satisfying $\ell + x \le n$. 
    Then 
    \begin{align*}
        \binom{n}{2} - \binom{n-\ell-x}{2} 
        \ge \binom{n}{2} - \binom{n-\ell}{2} + x (n-\ell-x). 
    \end{align*}
\end{fact}

\begin{fact}\label{FACT:inequality-c}
    Suppose that $r \ge 1$ is an integer and $0 \le x \le \frac{1}{4r}$ is a real number. 
    Then 
    \begin{align*}
        \left(\frac{1}{1-x}\right)^r
        \le 1+ 4rx. 
    \end{align*}
\end{fact}
\begin{proof}
    This is due to
    \begin{align*}
        \left(\frac{1}{1-x}\right)^r
        \le (1+2x)^r 
        \le 1+ 2rx + \sum_{i=2}^{r}(2rx)^{i}
        \le 1+ 2rx + 2rx\cdot  \sum_{i=1}^{r-1}\left(\frac{1}{2}\right)^{i}
        \le 1+ 4rx. 
    \end{align*}
\end{proof}

\begin{fact}\label{FACT:inequality-d}
    Suppose that $x \le 1$ and $r \ge 1$ are real numbers. 
    Then 
    \begin{align*}
        (1 - x)^{\frac{1}{r}} \le 1 - \frac{x}{r}. 
    \end{align*}
\end{fact}
\subsection{Graphs}
%
Given integers $m,n,s_1,s_2\ge 1$ and an $s_1$ by $s_2$ bipartite graph $B_{s_1,s_2} = B[W_1, W_2]$, the \textbf{Zarankiewicz number} $Z(m,n,B_{s_1, s_2})$ is the maximum number of edges in an $m$ by $n$ bipartite graph without a copy of $B_{s_1,s_2}$ with the set $W_1$ contained in the part of size $m$ and the set $W_2$ contained in the part of size $n$.  
For convenience, we use $Z(m,n,s_1,s_2)$ to denote $Z(m,n,K_{s_1, s_2})$. 

We will use the following classical result of K\H{o}v\'{a}ri, S\'{o}s and Tur\'{a}n~\cite{KST54}. 

\begin{theorem}[K{\H o}v\'{a}ri--S\'{o}s--Tur\'{a}n~\cite{KST54}]\label{THM:graph-KST}
    Let $m,n,s_1,s_2\ge 1$ be integers. 
    Then 
    \begin{align*}
        \mathrm{ex}(n,K_{s_1,s_2})
        & \le \frac{(s_2-1)^{\frac{1}{s_1}}}{2}n^{2-\frac{1}{s_1}} + \frac{s_1-1}{2}n, \quad\text{and}\\
        Z(m,n,s_1,s_2) 
        & \le (s_2-1)^{\frac{1}{s_1}} mn^{1-\frac{1}{s_1}} + (s_1-1)n. 
    \end{align*}
\end{theorem}

The following result of Alon and Yuster~\cite{AY96} will be useful in the proof of Theorem~\ref{THM:3rd-interval-graph}. 

\begin{theorem}[Alon--Yuster~\cite{AY96}]\label{THM:AY-graph-factor}
    Let $F$ be a graph with $m$ vertices. 
   For every $\varepsilon>0$ there exists $n_0=n_0(\varepsilon,F)$ such that the following holds for all $n \ge n_0$. 
   Every $n$-vertex graph $G$ with $\delta(G) \ge \left(1-\frac{1}{\chi (H)}+\varepsilon\right)n$ contains $\lfloor n/m \rfloor$ pairwise vertex-disjoint copies of $F$. 
\end{theorem}

We also need the following simple result on $\frac{\mathrm{ex}(n,F)}{n}$ in the proof of Theorem~\ref{THM:3rd-interval-graph}. 

\begin{proposition}\label{PROP:Turan-ratio}
    Let $F$ be a connected graph. 
    For all integers $n \ge m \ge 1$ we have 
    \begin{align*}
        \frac{\mathrm{ex}(n,F)}{n}
        \ge \left(1 - \frac{m}{n}\right)\frac{\mathrm{ex}(m,F)}{m}. 
    \end{align*}
\end{proposition}
\begin{proof}
    Let $G_m$ be an $m$-vertex $F$-free graph with $\mathrm{ex}(m,F)$ edges. 
    Let $G_n$ be the union of $\lfloor n/m \rfloor$ vertex-disjoint copies of $G_m$ with a set of $n - m \lfloor n/m \rfloor$ isolated vertices. 
    Since $F$ is connected, the graph $G_n$ is $F$-free. 
    Therefore, we have 
    \begin{align*}
        \frac{\mathrm{ex}(n,F)}{n}
        \ge \frac{|G_n|}{n}
        = \frac{\lfloor n/m \rfloor \times \mathrm{ex}(m,F)}{n}
        \ge \left(\frac{n}{m}-1\right) \frac{m}{n}\frac{\mathrm{ex}(m,F)}{m}
        = \left(1 - \frac{m}{n}\right)\frac{\mathrm{ex}(m,F)}{m}, 
    \end{align*}
    proving Proposition~\ref{PROP:Turan-ratio}. 
\end{proof}

\subsection{Hypergraphs}
Given a $r$-graph $\mathcal{H}$ and a vertex $v\in V(\mathcal{H})$, the \textbf{link} $L_{\mathcal{H}}(v)$ of $v$ in $\mathcal{H}$ is defined as 
\begin{align*}
    L_{\mathcal{H}}(v)
    := \left\{A\in \binom{V(\mathcal{H})}{r-1} \colon A\cup \{v\} \in \mathcal{H}\right\}. 
\end{align*}
Recall that the degree $d_{\mathcal{H}}(v)$ of $v$ is $d_{\mathcal{H}}(v) := |L_{\mathcal{H}}(v)|$. 
Given a set $S \subset V(\mathcal{H})$ let 
\begin{align*}
    L_{\mathcal{H}}(S) := \bigcap_{v\in S}L_{\mathcal{H}}(v)
\end{align*}
denote the \textbf{common link} of $S$ in $\mathcal{H}$, 
and let $d_{\mathcal{H}}(S) := |L_{\mathcal{H}}(S)|$. 
Let $T \subset V(\mathcal{H})$ be a set of size $r-1$, the \textbf{neighborhood} of $T$ is 
\begin{align*}
    N_{\mathcal{H}}(T)
    := \left\{v\in V(\mathcal{H}) \colon T\cup \{v\} \in \mathcal{H}\right\}. 
\end{align*}
We will omit the subscript $\mathcal{H}$ from the notations defined above if it is clear from the context. 

Extending the K{\H o}v\'{a}ri--S\'{o}s--Tur\'{a}n Theorem to hypergraphs, Erd\H{o}s~\cite{Erdos64} proved the following theorem. 

\begin{theorem}[Erd\H{o}s~\cite{Erdos64}]\label{THM:Erdos-hypergraph-KST}
    Let $n \ge r \ge 3$ and $s_r\ge \cdots \ge s_1 \ge 1$ be integers. 
    There exists a constant $C:= C(r,s_1, \ldots, s_r)$ such that 
    \begin{align*}
        \mathrm{ex}(n,K_{s_1, \ldots, s_r}^{r})
        \le Cn^{r-\frac{1}{s_1\cdots s_{r-1}}}. 
    \end{align*}
\end{theorem}

\textbf{Remark.} A detailed analysis 
of the proof of Erd\H{o}s in~\cite{Erdos64} shows that 
    \begin{align*}
        \mathrm{ex}(n,K_{s_1, \ldots, s_r}^{r})
        \le \frac{(s_2 + \cdots + s_r-r+1)^{\frac{1}{s_1}}}{r}n^{r-\frac{1}{s_1\cdots s_{r-1}}} + \frac{s_1 -1}{r} \binom{n}{r-1}. 
    \end{align*}

We also need the following lower bound for $\mathrm{ex}(n, K_{s_1,\ldots,s_r}^{r})$ which comes from a simple application of the probabilistic deletion method. 

\begin{fact}\label{FACT:Ks1sr-lower-bound}
    Let $r \ge 2$ and $s_r \geq s_{r-1} \geq \cdots \geq s_1 \geq 2$ be integers. 
    Then 
    \begin{align}\label{lower-bound-K}
          \mathrm{ex}(n, K_{s_1,\ldots,s_r}^{r}) 
          \geq {\rm{ex}}(n, K_{2,\ldots,2}^{r}) 
          = \Omega\left(n^{r- \frac{r}{2^r -1}}\right)
          = \Omega\left(n^{r- \frac{2}{3}}\right). 
    \end{align}
\end{fact}

An $r$-graph $\mathcal{H}$ is \textbf{bipartite} if there exists a bipartition $V_1 \cup V_2 = V(\mathcal{H})$ such that every edge in $\mathcal{H}$ has nonempty intersection with both $V_1$ and $V_2$. 
A bipartite $r$-graph $\mathcal{H} = \mathcal{H}[V_1, V_2]$ is \textbf{semibipartite} if every edge in $\mathcal{H}$ contains exactly one vertex from $V_1$. 
Let $\mathbb{B}:= B[W_1, \ldots, W_r]$ be an $r$-partite $r$-graph with parts $W_1, \ldots, W_r$. 
An \textbf{ordered copy} of $\mathbb{B}$ in a semibipartite $r$-graph $\mathcal{H}[V_1, V_2]$ is a copy of $\mathbb{B}$ with $W_1$ contained in $V_1$ and $W_2, \ldots, W_r$ contained in $V_2$. 
Extending the graph Zarankiewicz number to hypergraphs, we use $Z(m,n,\mathbb{B})$ to denote the maximum number of edges in an $m$ by $n$ semibipartite $r$-graph without any ordered copy of $\mathbb{B}$. 
For convenience, let $Z(m,n,s_1, \ldots, s_r) := Z(m,n,K_{s_1, \ldots, s_r}^{r})$. 

Extending results of K{\H o}v\'{a}ri--S\'{o}s--Tur\'{a}n and Erd\H{o}s, we prove the following upper bound for $Z(m,n,s_1, \ldots, s_r)$. Since the proof is essentially the same as that of Erd\H{o}s in~\cite{Erdos64}, we omit it here. 

\begin{proposition}\label{PROP:hypergraph-KST-Zaran}
    Suppose that $r \ge 3$, $s_r\ge \cdots \ge s_1 \ge 1$, and $m, n\ge 1$ are integers.  
    Then 
    \begin{align*}
        Z(m,n,s_1, \ldots, s_r)
        \le \frac{(s_2+\cdots+s_r-r+1)^{\frac{1}{s_1}}}{r-1}mn^{r-1-\frac{1}{s_1\cdots s_{r-1}}} + (s_1-1)\binom{n}{r-1}. 
    \end{align*}    
\end{proposition}

The following theorem of Lu and Sz\'{e}kely~\cite{LS07} will be useful in the proof of Theorem~\ref{THM:3rd-interval}. 

\begin{theorem}[Lu--Sz\'{e}kely~\cite{LS07}]\label{THM:Packing-F-mindegree}
    Let $F$ be an $r$-graph such that every edge in $F$ has nonempty intersection with at most $d$ other edges in $F$.
    Suppose that an $n$-vertex $r$-graph $\mathcal{H}$ satisfies 
    \begin{align*}
        v(F) \mid v(\mathcal{H}) \quad\text{and}\quad
        \delta(\mathcal{H}) \ge \left(1 - \frac{1}{e\left(d+1+r^2\frac{|F|}{v(F)}\right)}\right)\binom{n-1}{r-1}. 
    \end{align*}
    Then $\nu(F, \mathcal{H}) = v(\mathcal{H})/v(F)$.
    In particular, if an $n$-vertex $r$-graph $\mathcal{H}$ satisfies 
    \begin{align*}
        \sum_{i\in [r]}s_i \mid v(\mathcal{H}) 
        \quad\text{and}\quad
        \delta(\mathcal{H}) \ge \left(1 - \frac{1}{2er^2\prod_{i\in [r]}s_i}\right)\binom{n-1}{r-1}. 
    \end{align*}
    Then $\nu(K_{s_1, \ldots, s_r}^{r}, \mathcal{H}) = v(\mathcal{H})/s$.
\end{theorem}

The following simple upper bound for the number of edges in an $(t+1)F$-free $r$-graph with bounded degree will be useful. 

\begin{lemma}\label{LEMMA:trivial-max-degree}
    Let $m \ge r\ge 2$ be integers and $F$ be an $r$-graph on $m$ vertices. 
    For every $t \ge 0$, every $n$-vertex $(t+1)F$-free $r$-graph $\mathcal{H}$ satisfies $|\mathcal{H}| \le mt \cdot \Delta(\mathcal{H}) + \mathrm{ex}(n-mt, F)$. 
\end{lemma}
\begin{proof}
    Let $\mathcal{H}$ be an $n$-vertex $(t+1)F$-free $r$-graph. 
    Let $\mathcal{F} = \{F_1, \ldots, F_{\ell}\}$ be a maximum collection of pairwise vertex-disjoint copies of $F$ in $\mathcal{H}$. 
    Let $B:= \bigcup_{i\in [\ell]}V(F_i)$. 
    Since $\mathcal{H}$ is $(t+1)F$-free, we have $\ell \le t$ and hence, $|B|\le m\ell$. 
    Let $U:= V(\mathcal{H})\setminus B$. 
    By moving vertices from $U$ to $B$, we may assume that $|B| = mt$. 
    Observe that $\mathcal{H}[U]$ is $F$-free. 
    So we have 
    \begin{align*}
        |\mathcal{H}|
        \le \sum_{v\in B}d_{\mathcal{H}}(v) + |\mathcal{H}[U]|
        \le mt \cdot \Delta(\mathcal{H}) + \mathrm{ex}(n-mt, F), 
    \end{align*}
    proving Lemma~\ref{LEMMA:trivial-max-degree}. 
\end{proof}

We will need the following simple fact regarding the monotonicity of $\binom{n}{r} - \binom{n-x}{r} + \mathrm{ex}(n-x,F)$. 

\begin{fact}\label{FACT:increasing-f(ell)}
    Let $F$ be an $r$-graph and let
    \begin{align*}
        f(x) := \binom{n}{r} - \binom{n-x}{r} + \mathrm{ex}(n-x,F)
    \end{align*}
    for all integers $x \in [n]$. 
    Then $f(\ell) \le f(\ell+1)$ for all integers $\ell \in [n-1]$. 
\end{fact}
\section{Proofs for theorems in the first interval}\label{SEC:proof-1st-interval}
We prove Theorem~\ref{THM:1st-interval} in this section. 
Let us start with some technical lemmas. 

\begin{lemma}\label{LEMMA:1st-interval-bounded-max-degree}
    Let $n \ge m \ge r \ge 2$ be integers and $\delta \in[0, 1/m]$ be a real number. 
    Let $F$ be an $m$-vertex $r$-graph. 
    Suppose that $t$ is a positive integer satisfying 
    \begin{align*}
        t 
        \le \min\left\{\frac{\delta m n}{r-1}-r,\ \frac{n}{m}\right\}.
    \end{align*}
   Then every $n$-vertex $(t+1)F$-free $r$-graph $\mathcal{H}$ with $\Delta(\mathcal{H}) \le \left(1/m-\delta\right)\binom{n-1}{r-1}$ satisfies 
    \begin{align*}
        |\mathcal{H}|
        \le \binom{n}{r} - \binom{n-t}{r} + \mathrm{ex}(n-t, F). 
    \end{align*}
\end{lemma}
\begin{proof}[Proof of Lemma~\ref{LEMMA:1st-interval-bounded-max-degree}]
    Since $t \le \frac{\delta m n}{r-1}-r \le n - \left(1 - \delta m\right)^{\frac{1}{r-1}}n-r$, we have   
    \begin{align*}
        \left(1- \delta m\right)\left(\frac{n}{n-t-r}\right)^{r-1}
        \le 1. 
    \end{align*}
    Combined with Lemma~\ref{LEMMA:trivial-max-degree} and Fact~\ref{FACT:binom-inequality-a}, we obtain 
    \begin{align*}
        |\mathcal{H}|
        & \le mt \left(\frac{1}{m}-\delta\right)\binom{n-1}{r-1} + \mathrm{ex}(n-mt, F) \\
        & \le \left(1- \delta m\right)\left(\frac{n}{n-t-r}\right)^{r-1} t\binom{n-t}{r-1} + \mathrm{ex}(n-t, F) \\
        & \le t\binom{n-t}{r-1} + \mathrm{ex}(n-t, F) 
         \le \binom{n}{r} - \binom{n-t}{r} + \mathrm{ex}(n-t, F), 
    \end{align*}
    which proves Lemma~\ref{LEMMA:1st-interval-bounded-max-degree}. 
\end{proof}
\begin{lemma}\label{LEMMA:1st-interval-avoid-B}
    Let $n \ge m \ge r \ge 2$ be integers and $\delta_1, \delta_2 \ge 0$ be real numbers. 
     Let $F$ be a $(c_1, c_2)$-bounded $r$-graph. 
     Suppose that $\mathcal{H}$ is an $n$-vertex $r$-graph satisfying 
     \begin{enumerate}[label=(\roman*)]
         \item $|\mathcal{H}| \ge \left(1-c_2 + \delta_1\right)\mathrm{ex}(n,F)$, and 
         \item there exists a vertex $v\in V(\mathcal{H})$ with $d_{\mathcal{H}}(v) \ge (c_1 + \delta_2)\binom{n-1}{r-1}$. 
     \end{enumerate}
     Then for every set $B\subset V(\mathcal{H})\setminus\{v\}$ of size at most $\min\left\{\frac{\delta_1 \cdot \mathrm{ex}(n,F)}{\binom{n-1}{r-1}},\  \frac{\delta_2(n-1)}{r-1}\right\}$, there exists a copy of $F$ in $\mathcal{H}$ that have empty intersection with $B$. 
\end{lemma}
\begin{proof}
    Fix $B\subset V(\mathcal{H})\setminus\{v\}$ of size at most $\min\left\{\frac{\delta_1 \cdot \mathrm{ex}(n,F)}{\binom{n-1}{r-1}},\  \frac{\delta_2(n-1)}{r-1}\right\}$. 
    Let $U := V(\mathcal{H}) \setminus B$, $n_1 := |U|$, and $\mathcal{H}_1:= \mathcal{H}[U]$. 
    It follows from the assumptions that 
    \begin{align*}
        |\mathcal{H}_1|
        \ge \left(1-c_2 + \delta_1\right)\mathrm{ex}(n,F) - |B|\binom{n-1}{r-1}
        \ge (1-c_2)\cdot \mathrm{ex}(n,F), 
    \end{align*}
    and 
    \begin{align*}
        d_{\mathcal{H}_1}(v) 
        \ge d_{\mathcal{H}}(v) -|B|\binom{n-2}{r-2}
         \ge (c_1 + \delta_2)\binom{n-1}{r-1} - |B|\frac{r-1}{n-1}\binom{n-1}{r-1} 
         \ge c_1\binom{n-1}{r-1}. 
    \end{align*}
    So, it follows from the $(c_1, c_2)$-boundedness of $F$ that $F\subset \mathcal{H}_1$. 
\end{proof}

Now we are ready to prove Theorem~\ref{THM:1st-interval}. 

\begin{proof}[Proof of Theorem~\ref{THM:1st-interval}]
     Fix integers $m \ge r \ge 2$ and fix an $m$-vertex $(c_1, c_2)$-bounded $r$-graph $F$ with $c_1 < 1/m$ and $c_2>0$. 
     Let $n$ be a sufficiently large integer and $\mathcal{H}$ be an $n$-vertex $(t+1)F$-free $r$-graph with the maximum number of edges, where $t$ is an integer satisfying 
     \begin{align*}
         0 \le t 
         \le \min\left\{\frac{\delta \cdot \mathrm{ex}(n,F)}{m\binom{n-1}{r-1}},\  \frac{\delta (n-1)}{8m(r-1)}\right\}
         \quad\text{with}\quad
         \delta 
         := \min\left\{\frac{1}{4}\left(\frac{1}{m}-c_1\right),\ \frac{c_2}{4}\right\}.
     \end{align*}
     It follows from Proposition~\ref{PROP:three-lower-bounds} that 
     \begin{align}\label{equ:1st-H-lower-bound}
         |\mathcal{H}|
         \ge g_{1}(n,t,F)
         = \binom{n}{r}-\binom{n-t}{r}+\mathrm{ex}(n-t, F). 
     \end{align}
    Let 
    \begin{align*}
        V:= V(\mathcal{H}), \quad
        L:= \left\{v\in V \colon d_{\mathcal{H}}(v) \ge \left(c_1+ 2 \delta\right) \binom{n-1}{r-1}\right\}, \quad
        U:= V\setminus L, 
        \quad\text{and}\quad \ell:= L
    \end{align*}
    For convenience, let us assume that $L = \{v_1, \ldots, v_{\ell}\}$. 
    \begin{claim}\label{CLAIM:1st-nondegenerate-size-L-upper-bound}
        We have $\ell \le t$. 
    \end{claim}
    \begin{proof}
        Suppose that this is not true. 
        We may assume that $\ell = t+1$, since otherwise we can replace $L$ by a $(t+1)$-subset. 
        For $i\in [t+1]$ let $\mathcal{H}_i:= \mathcal{H}[U\cup \{v_i\}]$. 
        Let $n_1:= n-t$. 
        Notice from~\eqref{equ:1st-H-lower-bound} that for each $i\in [t+1]$, the $r$-graph $\mathcal{H}_i$ satisfies
        \begin{align}\label{equ:1st-Hi-lower-bound}
            |\mathcal{H}_i| 
            \ge |\mathcal{H}| - \left(\binom{n}{r} -\binom{n-t}{r}\right)
            \ge \mathrm{ex}(n, F)
             \ge (1-c_2 + \delta)\cdot \mathrm{ex}(n, F). 
        \end{align}
        %
        For every $i\in [t+1]$ we will find a copy of $F$, denoted by $F_i$, in $\mathcal{H}_i$  such that $F_1, \ldots, F_{t+1}$ are pairwise vertex-disjoint. 
        Define $B_0:= \emptyset$. 
        Suppose that we have defined $B_i \subset V(\mathcal{H})$ with $|B_i| \le im$ for some $i\in [0,t]$. 
        Consider the $r$-graph $\mathcal{H}_{i+1}$. 
        It follows from the definition of $L$ that  
        \begin{align}\label{equ:1st-Hi-max-deg}
            d_{\mathcal{H}_{i+1}}(v_{i+1})
            \ge d_{\mathcal{H}}(v_{i+1}) - t \binom{n-2}{r-2}
            & \ge (c_1 + 2\delta)\binom{n-1}{r-1} - t\frac{r-1}{n-1}\binom{n-1}{r-1} \notag \\
            & \ge (c_1 + \delta)\binom{n-1}{r-1}. 
        \end{align}
        Since 
        \begin{align*}
            |B_i| 
            \le i m 
            \le tm 
            \le \min\left\{\frac{\delta \cdot \mathrm{ex}(n,F)}{\binom{n-1}{r-1}},\  \frac{\delta (n-1)}{r-1}\right\}, 
        \end{align*}
        it follows from~\eqref{equ:1st-Hi-lower-bound},~\eqref{equ:1st-Hi-max-deg}, and Lemma~\ref{LEMMA:1st-interval-avoid-B} that there exists a copy of $F$ in $\mathcal{H}_{i+1}$ that have empty intersection with $B_i$. Denote this copy of $F$ by $F_{i+1}$ and let $B_{i+1}:= B_i \cup V(F_{i+1})$. 
        Then $|B_{i+1}| =  |B_i|+m \le (i+1)m$. 
        Inductively, we can find $F_1, \ldots, F_{t+1}$ as desired, contradicting the $(t+1)F$-freeness of $\mathcal{H}$. 
    \end{proof}
    Let $t_1 := t-\ell$. It follows from Claim~\ref{CLAIM:1st-nondegenerate-size-L-upper-bound} that $t_1 \ge 0$.
    
    \begin{claim}\label{CLAIM:1st-nondegenerate-H[U]-t1F-free}
        The induced subgraph $\mathcal{H}[U]$ is $(t_1+1)F$-free. 
    \end{claim}
    \begin{proof}
        Suppose to the contrary that there exists $t_1+1$ pairwise vertex-disjoint copies of $F$, denoted by $F_1', \ldots, F_{t_1+1}'$, in $\mathcal{H}[U]$. 
        Let $B_0:= \bigcup_{i\in [t_1+1]}V(F_i')$ and 
        notice that $|B_0| = (t_1+1) m$. 
        For $i\in [\ell]$ let $\mathcal{H}_i:= \mathcal{H}[U\cup \{v_i\}]$. 
        Repeating the argument as in the proof of Claim~\ref{CLAIM:1st-nondegenerate-size-L-upper-bound}, we find a copy of $F$, denoted by $F_i$, in $\mathcal{H}_i$ for every $i\in [\ell]$,  such that $F_1, \ldots, F_{\ell}$ are pairwise vertex-disjoint. 
        Moreover, we can also guarantee that each $F_i$ have empty intersection with $B_0$. 
        However, this contradicts the $(t+1)F$-freeness of $\mathcal{H}$, since $F_1, \ldots, F_{\ell}, F_1', \ldots, F_{t_1}'$ are $t+1$ pairwise vertex-disjoint copies of $F$ in $\mathcal{H}$. 
    \end{proof}
    It follows from $\ell \le t \le \frac{\delta (n-1)}{8m(r-1)}$ (by Claim~\ref{CLAIM:1st-nondegenerate-size-L-upper-bound}),  the definition of $L$, Facts~\ref{FACT:binom-inequality-a}, and~\ref{FACT:inequality-c} that 
    \begin{align*}
        \Delta(\mathcal{H}[U]) 
        \le (c_1+2\delta)\binom{n-1}{r-1}
        & \le (c_1+2\delta) \left(\frac{n}{n-\ell-r}\right)^{r-1}\binom{n-\ell-1}{r-1} \\
        & \le (c_1+2\delta) \left(1+ \frac{4(r-1)(\ell+r)}{n}\right)\binom{n-\ell-1}{r-1} \\
        & \le (c_1+2\delta) \left(1+ \frac{\delta}{m}\right)\binom{n-\ell-1}{r-1} \\
        & \le (c_1+3\delta)\binom{n-\ell-1}{r-1} 
        \le \left(\frac{1}{m} - \delta\right)\binom{n-\ell-1}{r-1},
    \end{align*}
    where the last inequality follows from the assumption that $\delta \le \frac{1}{4}\left(\frac{1}{m}-c_1\right)$. 
    In addition, since $n$ is large, we have $\frac{\delta m (n-\ell)}{r-1} - r \ge \frac{\delta (n-1)}{8m(r-1)} \ge t$. 
    So it follows from Claim~\ref{CLAIM:1st-nondegenerate-size-L-upper-bound} and Lemma~\ref{LEMMA:1st-interval-bounded-max-degree} that 
    \begin{align*}
        |\mathcal{H}[U]|
        & \le \binom{n-\ell}{r} - \binom{n-\ell-(t-\ell)}{r} + \mathrm{ex}(n-\ell-(t-\ell), F) \\
        & = \binom{n-\ell}{r} - \binom{n-t}{r} + \mathrm{ex}(n-t, F). 
    \end{align*}
    Consequently, we have 
    \begin{align*}
        |\mathcal{H}| 
        \le \binom{n}{r}-\binom{n-\ell}{r}+ |\mathcal{H}[U]|
        \le \binom{n}{r} - \binom{n-t}{r} + \mathrm{ex}(n-t, F),  
    \end{align*}
    proving Theorem~\ref{THM:1st-interval}. 
\end{proof}
\section{Proofs for theorems in the third interval}\label{SEC:proof-3rd-interval}
In this section, we prove Theorems~\ref{THM:3rd-interval} and~\ref{THM:3rd-interval-graph}. 
Both proofs share a similar strategy, and we begin with the proof of Theorem~\ref{THM:3rd-interval-graph} since it is more straightforward to comprehend compared to the hypergraph case. 

\subsection{Graphs: proof of Theorem~\ref{THM:3rd-interval-graph}}\label{SUBSEC:proof-3rd-interval-gp}
Let us provide a concise overview of the proof strategy for Theorem~\ref{THM:3rd-interval-graph}. 
We will start with an extremal $(t+1)\mathbb{B}$-free graph $G$, take a maximum collection of pairwise vertex-disjoint copies of $\mathbb{B}$ in $G$, and denote the union of their vertex sets by $V_1$. 
We will show that the set $L \subset V_1$ of large degree vertices contains most vertices in $V_1$. 
Then, a key step is to use the Alon--Yuster Theorem (Theorem~\ref{THM:AY-graph-factor}) to show that the number of edges crossing $L$ and $V(G)\setminus V_1$ is small. 
Using this key property and applying some relatively trivial estimates, we will obtain the desired bound on the size of $|G|$.

\begin{proof}[Proof of Theorem~\ref{THM:3rd-interval-graph}]
   Fix integers $s_2 \ge s_1 \ge 2$ and an $s_1$ by $s_2$ bipartite graph $B_{s_1, s_2} =: \mathbb{B}$. 
   Let
   \begin{align*}
       s:= s_1 + s_2 \quad\text{and}\quad
       \varepsilon := \frac{1}{65s_1 s^2}. 
   \end{align*}
   Let $n$ be a sufficiently large integer, $t \in \left[\frac{n}{s}-\varepsilon n, \frac{n}{s}\right]$ be an integer, and $G$ be an $n$-vertex $(t+1)\mathbb{B}$-free graph with maximum number of edges. 
   Let $\hat{t} := s(t+1)-1$ for simplicity. 
   Notice from Proposition~\ref{PROP:three-lower-bounds} and Fact~\ref{FACT:increasing-f(ell)} that 
   \begin{align}\label{equ:3rd-graph-G-lower-bound}
       |G| 
        \ge g_3(n,t,\mathbb{B})
        = \binom{\hat{t}}{2} 
            + \mathrm{ex}\left(n-\hat{t}, \mathbb{B}\right) 
         \ge \binom{st}{2} + \mathrm{ex}(n-st,\mathbb{B}). 
   \end{align}
   Let $\mathcal{B} = \{\mathbb{B}_1, \ldots, \mathbb{B}_t\}$ be a collection of pairwise vertex-disjoint copies of $\mathbb{B}$ in $G$ (the existence of $\mathcal{B}$ is guaranteed by the maximality of $G$). 
    Let 
    \begin{align*}
        V := V(G), \quad 
        V_1 := \bigcup_{i\in [t]}V(B_i),
        \quad
        V_2:= V\setminus V_1, 
        \quad\text{and}\quad 
        G_i := G[V_i]\quad\text{for $i\in \{1,2\}$}. 
    \end{align*}
    Note that $|V_1| = st$ and $G_2$ is $\mathbb{B}$-free, so we have (by Theorem~\ref{THM:graph-KST})
    \begin{align}\label{equ:3rd-G2-upper-bound}
        |G_2| 
        \le \mathrm{ex}(n-st, \mathbb{B})
        \le C (n-st)^{2-\frac{1}{s_1}}. 
    \end{align}
    Trivially, we have 
    \begin{align}\label{equ:3rd-G12-upper-bound}
        |G[V_1, V_2]| \le |V_1||V_2| = st(n-st). 
    \end{align}
    Let 
    \begin{align*}
        \alpha := 1-\frac{1}{2s_1}, \quad 
        L := \left\{v\in V_1 \colon d_{G_1}(v) \ge \alpha (st-1)\right\}, 
        \quad
        S:= V_1 \setminus L, 
        \quad\text{and}\quad 
        \ell := |L|. 
    \end{align*}
    It follows from the definition of $L$ that 
    \begin{align}\label{equ:3rd-G1-upper-bound}
        |G_1|
        & = \frac{1}{2}\left(\sum_{v\in L}d_{G_1}(v) + \sum_{v\in S}d_{G_1}(v)\right) \notag \\
        & \le \frac{1}{2}\left(\ell (st-1) + (st-\ell) \alpha (st-1)\right) 
         = \binom{st}{2} - \frac{st-\ell}{2}(1-\alpha)(st-1). 
    \end{align}

    \begin{claim}\label{CLAIM:3rd-graph-S-upper-bound}
        We have $|S| \le \frac{4(n-st)}{(1-\alpha)}$. 
    \end{claim}
    \begin{proof}
        It follows from~\eqref{equ:3rd-G2-upper-bound},~\eqref{equ:3rd-G12-upper-bound}, and~\eqref{equ:3rd-G1-upper-bound} that 
        \begin{align*}
            |G|
            & = |G_1|+ |G[V_1, V_2]| + |G_2| \\
            & \le \binom{st}{2} - \frac{st-\ell}{2}(1-\alpha)(st-1) + st(n-st)
                + \mathrm{ex}(n-st, \mathbb{B}). 
        \end{align*}
        Combined with~\eqref{equ:3rd-graph-G-lower-bound}, we obtain 
        \begin{align*}
            - \frac{st-\ell}{2}(1-\alpha)(st-1) + st(n-st) \ge 0,
        \end{align*}
        which implies that $|S|
            = st-\ell 
            \le \frac{2st(n-st)}{(st-1)(1-\alpha)}
            \le \frac{4(n-st)}{1-\alpha}$. 
    \end{proof}
    Let $\mathcal{B}_{S}$ be the collection of members in $\mathcal{B}$ that have nonempty intersection with $S$. Let 
    \begin{align*}
        S_1:= \bigcup_{B_i\in \mathcal{B}_{S}} V(B_i)
        \quad\text{and}\quad
        L_1:= V_1\setminus S_1. 
    \end{align*}
    Note that $S\subset S_1$ and $L_1\subset L$. 
    Also, notice from Claim~\ref{CLAIM:3rd-graph-S-upper-bound} that 
    \begin{align}\label{equ:3rd-graph-S1-upper-bound}
            |S_1| 
            \le  s|S|
            \le \frac{4s(n-st)}{1-\alpha}. 
    \end{align}

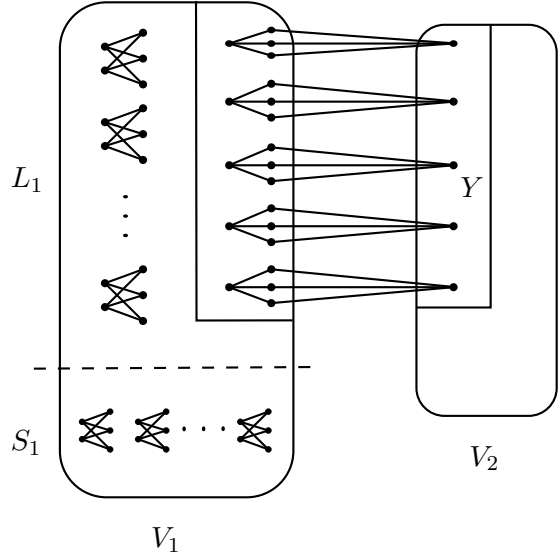
\begin{figure}[htbp]
\centering
\tikzset{every picture/.style={line width=0.8pt}} 
\begin{tikzpicture}[x=0.75pt,y=0.75pt,yscale=-1,xscale=1]
\draw   (116,449.52) .. controls (116,436.65) and (126.44,426.21) .. (139.31,426.21) -- (209.25,426.21) .. controls (222.12,426.21) and (232.56,436.65) .. (232.56,449.52) -- (232.56,651.86) .. controls (232.56,664.74) and (222.12,675.18) .. (209.25,675.18) -- (139.31,675.18) .. controls (126.44,675.18) and (116,664.74) .. (116,651.86) -- cycle ;
\draw  [fill={rgb, 255:red, 0; green, 0; blue, 0 }  ,fill opacity=1 ] (137,448.5) .. controls (137,447.67) and (137.67,447) .. (138.5,447) .. controls (139.33,447) and (140,447.67) .. (140,448.5) .. controls (140,449.33) and (139.33,450) .. (138.5,450) .. controls (137.67,450) and (137,449.33) .. (137,448.5) -- cycle ;
\draw  [fill={rgb, 255:red, 0; green, 0; blue, 0 }  ,fill opacity=1 ] (156,441.5) .. controls (156,440.67) and (156.67,440) .. (157.5,440) .. controls (158.33,440) and (159,440.67) .. (159,441.5) .. controls (159,442.33) and (158.33,443) .. (157.5,443) .. controls (156.67,443) and (156,442.33) .. (156,441.5) -- cycle ;
\draw  [fill={rgb, 255:red, 0; green, 0; blue, 0 }  ,fill opacity=1 ] (137,460.5) .. controls (137,459.67) and (137.67,459) .. (138.5,459) .. controls (139.33,459) and (140,459.67) .. (140,460.5) .. controls (140,461.33) and (139.33,462) .. (138.5,462) .. controls (137.67,462) and (137,461.33) .. (137,460.5) -- cycle ;
\draw  [fill={rgb, 255:red, 0; green, 0; blue, 0 }  ,fill opacity=1 ] (156,467.5) .. controls (156,466.67) and (156.67,466) .. (157.5,466) .. controls (158.33,466) and (159,466.67) .. (159,467.5) .. controls (159,468.33) and (158.33,469) .. (157.5,469) .. controls (156.67,469) and (156,468.33) .. (156,467.5) -- cycle ;
\draw  [fill={rgb, 255:red, 0; green, 0; blue, 0 }  ,fill opacity=1 ] (156,454.5) .. controls (156,453.67) and (156.67,453) .. (157.5,453) .. controls (158.33,453) and (159,453.67) .. (159,454.5) .. controls (159,455.33) and (158.33,456) .. (157.5,456) .. controls (156.67,456) and (156,455.33) .. (156,454.5) -- cycle ;
\draw    (138.5,448.5) -- (157.5,441.5) ;
\draw    (138.5,460.5) -- (157.5,454.5) ;
\draw    (138.5,448.5) -- (157.5,454.5) ;
\draw    (138.5,460.5) -- (157.5,441.5) ;
\draw    (138.5,448.5) -- (157.5,467.5) ;
\draw    (138.5,460.5) -- (157.5,467.5) ;
\draw  [fill={rgb, 255:red, 0; green, 0; blue, 0 }  ,fill opacity=1 ] (137,486.5) .. controls (137,485.67) and (137.67,485) .. (138.5,485) .. controls (139.33,485) and (140,485.67) .. (140,486.5) .. controls (140,487.33) and (139.33,488) .. (138.5,488) .. controls (137.67,488) and (137,487.33) .. (137,486.5) -- cycle ;
\draw  [fill={rgb, 255:red, 0; green, 0; blue, 0 }  ,fill opacity=1 ] (156,479.5) .. controls (156,478.67) and (156.67,478) .. (157.5,478) .. controls (158.33,478) and (159,478.67) .. (159,479.5) .. controls (159,480.33) and (158.33,481) .. (157.5,481) .. controls (156.67,481) and (156,480.33) .. (156,479.5) -- cycle ;
\draw  [fill={rgb, 255:red, 0; green, 0; blue, 0 }  ,fill opacity=1 ] (137,498.5) .. controls (137,497.67) and (137.67,497) .. (138.5,497) .. controls (139.33,497) and (140,497.67) .. (140,498.5) .. controls (140,499.33) and (139.33,500) .. (138.5,500) .. controls (137.67,500) and (137,499.33) .. (137,498.5) -- cycle ;
\draw  [fill={rgb, 255:red, 0; green, 0; blue, 0 }  ,fill opacity=1 ] (156,505.5) .. controls (156,504.67) and (156.67,504) .. (157.5,504) .. controls (158.33,504) and (159,504.67) .. (159,505.5) .. controls (159,506.33) and (158.33,507) .. (157.5,507) .. controls (156.67,507) and (156,506.33) .. (156,505.5) -- cycle ;
\draw  [fill={rgb, 255:red, 0; green, 0; blue, 0 }  ,fill opacity=1 ] (156,492.5) .. controls (156,491.67) and (156.67,491) .. (157.5,491) .. controls (158.33,491) and (159,491.67) .. (159,492.5) .. controls (159,493.33) and (158.33,494) .. (157.5,494) .. controls (156.67,494) and (156,493.33) .. (156,492.5) -- cycle ;
\draw    (138.5,486.5) -- (157.5,479.5) ;
\draw    (138.5,498.5) -- (157.5,492.5) ;
\draw    (138.5,486.5) -- (157.5,492.5) ;
\draw    (138.5,498.5) -- (157.5,479.5) ;
\draw    (138.5,486.5) -- (157.5,505.5) ;
\draw    (138.5,498.5) -- (157.5,505.5) ;
\draw  [fill={rgb, 255:red, 0; green, 0; blue, 0 }  ,fill opacity=1 ] (137,567.5) .. controls (137,566.67) and (137.67,566) .. (138.5,566) .. controls (139.33,566) and (140,566.67) .. (140,567.5) .. controls (140,568.33) and (139.33,569) .. (138.5,569) .. controls (137.67,569) and (137,568.33) .. (137,567.5) -- cycle ;
\draw  [fill={rgb, 255:red, 0; green, 0; blue, 0 }  ,fill opacity=1 ] (156,560.5) .. controls (156,559.67) and (156.67,559) .. (157.5,559) .. controls (158.33,559) and (159,559.67) .. (159,560.5) .. controls (159,561.33) and (158.33,562) .. (157.5,562) .. controls (156.67,562) and (156,561.33) .. (156,560.5) -- cycle ;
\draw  [fill={rgb, 255:red, 0; green, 0; blue, 0 }  ,fill opacity=1 ] (137,579.5) .. controls (137,578.67) and (137.67,578) .. (138.5,578) .. controls (139.33,578) and (140,578.67) .. (140,579.5) .. controls (140,580.33) and (139.33,581) .. (138.5,581) .. controls (137.67,581) and (137,580.33) .. (137,579.5) -- cycle ;
\draw  [fill={rgb, 255:red, 0; green, 0; blue, 0 }  ,fill opacity=1 ] (156,586.5) .. controls (156,585.67) and (156.67,585) .. (157.5,585) .. controls (158.33,585) and (159,585.67) .. (159,586.5) .. controls (159,587.33) and (158.33,588) .. (157.5,588) .. controls (156.67,588) and (156,587.33) .. (156,586.5) -- cycle ;
\draw  [fill={rgb, 255:red, 0; green, 0; blue, 0 }  ,fill opacity=1 ] (156,573.5) .. controls (156,572.67) and (156.67,572) .. (157.5,572) .. controls (158.33,572) and (159,572.67) .. (159,573.5) .. controls (159,574.33) and (158.33,575) .. (157.5,575) .. controls (156.67,575) and (156,574.33) .. (156,573.5) -- cycle ;
\draw    (138.5,567.5) -- (157.5,560.5) ;
\draw    (138.5,579.5) -- (157.5,573.5) ;
\draw    (138.5,567.5) -- (157.5,573.5) ;
\draw    (138.5,579.5) -- (157.5,560.5) ;
\draw    (138.5,567.5) -- (157.5,586.5) ;
\draw    (138.5,579.5) -- (157.5,586.5) ;
\draw  [fill={rgb, 255:red, 0; green, 0; blue, 0 }  ,fill opacity=1 ] (126,637.21) .. controls (126,636.6) and (126.49,636.11) .. (127.1,636.11) .. controls (127.71,636.11) and (128.2,636.6) .. (128.2,637.21) .. controls (128.2,637.81) and (127.71,638.3) .. (127.1,638.3) .. controls (126.49,638.3) and (126,637.81) .. (126,637.21) -- cycle ;
\draw  [fill={rgb, 255:red, 0; green, 0; blue, 0 }  ,fill opacity=1 ] (139.96,632.1) .. controls (139.96,631.49) and (140.46,631) .. (141.06,631) .. controls (141.67,631) and (142.17,631.49) .. (142.17,632.1) .. controls (142.17,632.7) and (141.67,633.19) .. (141.06,633.19) .. controls (140.46,633.19) and (139.96,632.7) .. (139.96,632.1) -- cycle ;
\draw  [fill={rgb, 255:red, 0; green, 0; blue, 0 }  ,fill opacity=1 ] (126,645.97) .. controls (126,645.36) and (126.49,644.87) .. (127.1,644.87) .. controls (127.71,644.87) and (128.2,645.36) .. (128.2,645.97) .. controls (128.2,646.57) and (127.71,647.06) .. (127.1,647.06) .. controls (126.49,647.06) and (126,646.57) .. (126,645.97) -- cycle ;
\draw  [fill={rgb, 255:red, 0; green, 0; blue, 0 }  ,fill opacity=1 ] (139.96,651.08) .. controls (139.96,650.48) and (140.46,649.99) .. (141.06,649.99) .. controls (141.67,649.99) and (142.17,650.48) .. (142.17,651.08) .. controls (142.17,651.69) and (141.67,652.18) .. (141.06,652.18) .. controls (140.46,652.18) and (139.96,651.69) .. (139.96,651.08) -- cycle ;
\draw  [fill={rgb, 255:red, 0; green, 0; blue, 0 }  ,fill opacity=1 ] (139.96,641.59) .. controls (139.96,640.98) and (140.46,640.49) .. (141.06,640.49) .. controls (141.67,640.49) and (142.17,640.98) .. (142.17,641.59) .. controls (142.17,642.19) and (141.67,642.68) .. (141.06,642.68) .. controls (140.46,642.68) and (139.96,642.19) .. (139.96,641.59) -- cycle ;
\draw    (127.1,637.21) -- (141.06,632.1) ;
\draw    (127.1,645.97) -- (141.06,641.59) ;
\draw    (127.1,637.21) -- (141.06,641.59) ;
\draw    (127.1,645.97) -- (141.06,632.1) ;
\draw    (127.1,637.21) -- (141.06,651.08) ;
\draw    (127.1,645.97) -- (141.06,651.08) ;
\draw  [fill={rgb, 255:red, 0; green, 0; blue, 0 }  ,fill opacity=1 ] (154,637.21) .. controls (154,636.6) and (154.49,636.11) .. (155.1,636.11) .. controls (155.71,636.11) and (156.2,636.6) .. (156.2,637.21) .. controls (156.2,637.81) and (155.71,638.3) .. (155.1,638.3) .. controls (154.49,638.3) and (154,637.81) .. (154,637.21) -- cycle ;
\draw  [fill={rgb, 255:red, 0; green, 0; blue, 0 }  ,fill opacity=1 ] (167.96,632.1) .. controls (167.96,631.49) and (168.46,631) .. (169.06,631) .. controls (169.67,631) and (170.17,631.49) .. (170.17,632.1) .. controls (170.17,632.7) and (169.67,633.19) .. (169.06,633.19) .. controls (168.46,633.19) and (167.96,632.7) .. (167.96,632.1) -- cycle ;
\draw  [fill={rgb, 255:red, 0; green, 0; blue, 0 }  ,fill opacity=1 ] (154,645.97) .. controls (154,645.36) and (154.49,644.87) .. (155.1,644.87) .. controls (155.71,644.87) and (156.2,645.36) .. (156.2,645.97) .. controls (156.2,646.57) and (155.71,647.06) .. (155.1,647.06) .. controls (154.49,647.06) and (154,646.57) .. (154,645.97) -- cycle ;
\draw  [fill={rgb, 255:red, 0; green, 0; blue, 0 }  ,fill opacity=1 ] (167.96,651.08) .. controls (167.96,650.48) and (168.46,649.99) .. (169.06,649.99) .. controls (169.67,649.99) and (170.17,650.48) .. (170.17,651.08) .. controls (170.17,651.69) and (169.67,652.18) .. (169.06,652.18) .. controls (168.46,652.18) and (167.96,651.69) .. (167.96,651.08) -- cycle ;
\draw  [fill={rgb, 255:red, 0; green, 0; blue, 0 }  ,fill opacity=1 ] (167.96,641.59) .. controls (167.96,640.98) and (168.46,640.49) .. (169.06,640.49) .. controls (169.67,640.49) and (170.17,640.98) .. (170.17,641.59) .. controls (170.17,642.19) and (169.67,642.68) .. (169.06,642.68) .. controls (168.46,642.68) and (167.96,642.19) .. (167.96,641.59) -- cycle ;
\draw    (155.1,637.21) -- (169.06,632.1) ;
\draw    (155.1,645.97) -- (169.06,641.59) ;
\draw    (155.1,637.21) -- (169.06,641.59) ;
\draw    (155.1,645.97) -- (169.06,632.1) ;
\draw    (155.1,637.21) -- (169.06,651.08) ;
\draw    (155.1,645.97) -- (169.06,651.08) ;
\draw  [fill={rgb, 255:red, 0; green, 0; blue, 0 }  ,fill opacity=1 ] (205,637.21) .. controls (205,636.6) and (205.49,636.11) .. (206.1,636.11) .. controls (206.71,636.11) and (207.2,636.6) .. (207.2,637.21) .. controls (207.2,637.81) and (206.71,638.3) .. (206.1,638.3) .. controls (205.49,638.3) and (205,637.81) .. (205,637.21) -- cycle ;
\draw  [fill={rgb, 255:red, 0; green, 0; blue, 0 }  ,fill opacity=1 ] (218.96,632.1) .. controls (218.96,631.49) and (219.46,631) .. (220.06,631) .. controls (220.67,631) and (221.17,631.49) .. (221.17,632.1) .. controls (221.17,632.7) and (220.67,633.19) .. (220.06,633.19) .. controls (219.46,633.19) and (218.96,632.7) .. (218.96,632.1) -- cycle ;
\draw  [fill={rgb, 255:red, 0; green, 0; blue, 0 }  ,fill opacity=1 ] (205,645.97) .. controls (205,645.36) and (205.49,644.87) .. (206.1,644.87) .. controls (206.71,644.87) and (207.2,645.36) .. (207.2,645.97) .. controls (207.2,646.57) and (206.71,647.06) .. (206.1,647.06) .. controls (205.49,647.06) and (205,646.57) .. (205,645.97) -- cycle ;
\draw  [fill={rgb, 255:red, 0; green, 0; blue, 0 }  ,fill opacity=1 ] (218.96,651.08) .. controls (218.96,650.48) and (219.46,649.99) .. (220.06,649.99) .. controls (220.67,649.99) and (221.17,650.48) .. (221.17,651.08) .. controls (221.17,651.69) and (220.67,652.18) .. (220.06,652.18) .. controls (219.46,652.18) and (218.96,651.69) .. (218.96,651.08) -- cycle ;
\draw  [fill={rgb, 255:red, 0; green, 0; blue, 0 }  ,fill opacity=1 ] (218.96,641.59) .. controls (218.96,640.98) and (219.46,640.49) .. (220.06,640.49) .. controls (220.67,640.49) and (221.17,640.98) .. (221.17,641.59) .. controls (221.17,642.19) and (220.67,642.68) .. (220.06,642.68) .. controls (219.46,642.68) and (218.96,642.19) .. (218.96,641.59) -- cycle ;
\draw    (206.1,637.21) -- (220.06,632.1) ;
\draw    (206.1,645.97) -- (220.06,641.59) ;
\draw    (206.1,637.21) -- (220.06,641.59) ;
\draw    (206.1,645.97) -- (220.06,632.1) ;
\draw    (206.1,637.21) -- (220.06,651.08) ;
\draw    (206.1,645.97) -- (220.06,651.08) ;
\draw  [dash pattern={on 4.5pt off 4.5pt}]  (103,610.47) -- (247.02,609.63) ;
\draw  [fill={rgb, 255:red, 0; green, 0; blue, 0 }  ,fill opacity=1 ] (149.42,523.99) .. controls (149.42,523.81) and (149.12,523.67) .. (148.75,523.67) .. controls (148.38,523.67) and (148.08,523.81) .. (148.08,523.99) .. controls (148.08,524.17) and (148.38,524.32) .. (148.75,524.32) .. controls (149.12,524.32) and (149.42,524.17) .. (149.42,523.99) -- cycle ;
\draw  [fill={rgb, 255:red, 0; green, 0; blue, 0 }  ,fill opacity=1 ] (149.42,543.47) .. controls (149.42,543.29) and (149.12,543.15) .. (148.75,543.15) .. controls (148.38,543.15) and (148.08,543.29) .. (148.08,543.47) .. controls (148.08,543.65) and (148.38,543.8) .. (148.75,543.8) .. controls (149.12,543.8) and (149.42,543.65) .. (149.42,543.47) -- cycle ;
\draw  [fill={rgb, 255:red, 0; green, 0; blue, 0 }  ,fill opacity=1 ] (149.42,533.73) .. controls (149.42,533.55) and (149.12,533.41) .. (148.75,533.41) .. controls (148.38,533.41) and (148.08,533.55) .. (148.08,533.73) .. controls (148.08,533.91) and (148.38,534.06) .. (148.75,534.06) .. controls (149.12,534.06) and (149.42,533.91) .. (149.42,533.73) -- cycle ;
\draw  [fill={rgb, 255:red, 0; green, 0; blue, 0 }  ,fill opacity=1 ] (197.49,641.34) .. controls (197.67,641.34) and (197.82,641.04) .. (197.81,640.67) .. controls (197.81,640.3) and (197.66,640.01) .. (197.49,640.01) .. controls (197.31,640.01) and (197.16,640.31) .. (197.16,640.68) .. controls (197.17,641.04) and (197.31,641.34) .. (197.49,641.34) -- cycle ;
\draw  [fill={rgb, 255:red, 0; green, 0; blue, 0 }  ,fill opacity=1 ] (178.01,641.46) .. controls (178.19,641.45) and (178.34,641.16) .. (178.34,640.79) .. controls (178.33,640.42) and (178.19,640.12) .. (178.01,640.12) .. controls (177.83,640.12) and (177.68,640.42) .. (177.69,640.79) .. controls (177.69,641.16) and (177.84,641.46) .. (178.01,641.46) -- cycle ;
\draw  [fill={rgb, 255:red, 0; green, 0; blue, 0 }  ,fill opacity=1 ] (187.75,641.4) .. controls (187.93,641.4) and (188.08,641.1) .. (188.07,640.73) .. controls (188.07,640.36) and (187.93,640.06) .. (187.75,640.06) .. controls (187.57,640.07) and (187.42,640.36) .. (187.43,640.73) .. controls (187.43,641.1) and (187.57,641.4) .. (187.75,641.4) -- cycle ;
\draw   (232.17,586.18) -- (184.2,586.25) -- (183.95,425.52) ;
\draw   (293.98,451.51) .. controls (293.98,443.78) and (300.25,437.51) .. (307.98,437.51) -- (349.98,437.51) .. controls (357.71,437.51) and (363.98,443.78) .. (363.98,451.51) -- (363.98,620.22) .. controls (363.98,627.95) and (357.71,634.22) .. (349.98,634.22) -- (307.98,634.22) .. controls (300.25,634.22) and (293.98,627.95) .. (293.98,620.22) -- cycle ;
\draw   (331.02,437.62) -- (331.02,579.71) -- (294.24,579.71) ;
\draw  [fill={rgb, 255:red, 0; green, 0; blue, 0 }  ,fill opacity=1 ] (199,446.94) .. controls (199,446.32) and (199.67,445.81) .. (200.5,445.81) .. controls (201.33,445.81) and (202,446.32) .. (202,446.94) .. controls (202,447.57) and (201.33,448.08) .. (200.5,448.08) .. controls (199.67,448.08) and (199,447.57) .. (199,446.94) -- cycle ;
\draw  [fill={rgb, 255:red, 0; green, 0; blue, 0 }  ,fill opacity=1 ] (220,440.13) .. controls (220,439.51) and (220.67,439) .. (221.5,439) .. controls (222.33,439) and (223,439.51) .. (223,440.13) .. controls (223,440.76) and (222.33,441.27) .. (221.5,441.27) .. controls (220.67,441.27) and (220,440.76) .. (220,440.13) -- cycle ;
\draw  [fill={rgb, 255:red, 0; green, 0; blue, 0 }  ,fill opacity=1 ] (313.97,446.94) .. controls (313.97,446.32) and (313.3,445.81) .. (312.47,445.81) .. controls (311.64,445.81) and (310.97,446.32) .. (310.97,446.94) .. controls (310.97,447.57) and (311.64,448.08) .. (312.47,448.08) .. controls (313.3,448.08) and (313.97,447.57) .. (313.97,446.94) -- cycle ;
\draw  [fill={rgb, 255:red, 0; green, 0; blue, 0 }  ,fill opacity=1 ] (220,452.99) .. controls (220,452.37) and (220.67,451.86) .. (221.5,451.86) .. controls (222.33,451.86) and (223,452.37) .. (223,452.99) .. controls (223,453.62) and (222.33,454.13) .. (221.5,454.13) .. controls (220.67,454.13) and (220,453.62) .. (220,452.99) -- cycle ;
\draw  [fill={rgb, 255:red, 0; green, 0; blue, 0 }  ,fill opacity=1 ] (220,446.94) .. controls (220,446.32) and (220.67,445.81) .. (221.5,445.81) .. controls (222.33,445.81) and (223,446.32) .. (223,446.94) .. controls (223,447.57) and (222.33,448.08) .. (221.5,448.08) .. controls (220.67,448.08) and (220,447.57) .. (220,446.94) -- cycle ;
\draw    (200.5,446.94) -- (221.5,440.13) ;
\draw    (312.47,446.94) -- (221.5,446.94) ;
\draw    (200.5,446.94) -- (221.5,446.94) ;
\draw    (313.97,446.94) -- (221.5,440.13) ;
\draw    (200.5,446.94) -- (221.5,452.99) ;
\draw    (312.47,446.94) -- (221.5,452.99) ;
\draw  [fill={rgb, 255:red, 0; green, 0; blue, 0 }  ,fill opacity=1 ] (199.03,476.17) .. controls (199.03,475.34) and (199.7,474.67) .. (200.53,474.67) .. controls (201.36,474.67) and (202.03,475.34) .. (202.03,476.17) .. controls (202.03,477) and (201.36,477.67) .. (200.53,477.67) .. controls (199.7,477.67) and (199.03,477) .. (199.03,476.17) -- cycle ;
\draw  [fill={rgb, 255:red, 0; green, 0; blue, 0 }  ,fill opacity=1 ] (220.03,467.17) .. controls (220.03,466.34) and (220.7,465.67) .. (221.53,465.67) .. controls (222.36,465.67) and (223.03,466.34) .. (223.03,467.17) .. controls (223.03,468) and (222.36,468.67) .. (221.53,468.67) .. controls (220.7,468.67) and (220.03,468) .. (220.03,467.17) -- cycle ;
\draw  [fill={rgb, 255:red, 0; green, 0; blue, 0 }  ,fill opacity=1 ] (314,476.17) .. controls (314,475.34) and (313.33,474.67) .. (312.5,474.67) .. controls (311.67,474.67) and (311,475.34) .. (311,476.17) .. controls (311,477) and (311.67,477.67) .. (312.5,477.67) .. controls (313.33,477.67) and (314,477) .. (314,476.17) -- cycle ;
\draw  [fill={rgb, 255:red, 0; green, 0; blue, 0 }  ,fill opacity=1 ] (220.03,484.17) .. controls (220.03,483.34) and (220.7,482.67) .. (221.53,482.67) .. controls (222.36,482.67) and (223.03,483.34) .. (223.03,484.17) .. controls (223.03,485) and (222.36,485.67) .. (221.53,485.67) .. controls (220.7,485.67) and (220.03,485) .. (220.03,484.17) -- cycle ;
\draw  [fill={rgb, 255:red, 0; green, 0; blue, 0 }  ,fill opacity=1 ] (220.03,476.17) .. controls (220.03,475.34) and (220.7,474.67) .. (221.53,474.67) .. controls (222.36,474.67) and (223.03,475.34) .. (223.03,476.17) .. controls (223.03,477) and (222.36,477.67) .. (221.53,477.67) .. controls (220.7,477.67) and (220.03,477) .. (220.03,476.17) -- cycle ;
\draw    (200.53,476.17) -- (221.53,467.17) ;
\draw    (312.5,476.17) -- (221.53,476.17) ;
\draw    (200.53,476.17) -- (221.53,476.17) ;
\draw    (314,476.17) -- (221.53,467.17) ;
\draw    (200.53,476.17) -- (221.53,484.17) ;
\draw    (312.5,476.17) -- (221.53,484.17) ;
\draw  [fill={rgb, 255:red, 0; green, 0; blue, 0 }  ,fill opacity=1 ] (199.03,508.17) .. controls (199.03,507.34) and (199.7,506.67) .. (200.53,506.67) .. controls (201.36,506.67) and (202.03,507.34) .. (202.03,508.17) .. controls (202.03,509) and (201.36,509.67) .. (200.53,509.67) .. controls (199.7,509.67) and (199.03,509) .. (199.03,508.17) -- cycle ;
\draw  [fill={rgb, 255:red, 0; green, 0; blue, 0 }  ,fill opacity=1 ] (220.03,499.17) .. controls (220.03,498.34) and (220.7,497.67) .. (221.53,497.67) .. controls (222.36,497.67) and (223.03,498.34) .. (223.03,499.17) .. controls (223.03,500) and (222.36,500.67) .. (221.53,500.67) .. controls (220.7,500.67) and (220.03,500) .. (220.03,499.17) -- cycle ;
\draw  [fill={rgb, 255:red, 0; green, 0; blue, 0 }  ,fill opacity=1 ] (314,508.17) .. controls (314,507.34) and (313.33,506.67) .. (312.5,506.67) .. controls (311.67,506.67) and (311,507.34) .. (311,508.17) .. controls (311,509) and (311.67,509.67) .. (312.5,509.67) .. controls (313.33,509.67) and (314,509) .. (314,508.17) -- cycle ;
\draw  [fill={rgb, 255:red, 0; green, 0; blue, 0 }  ,fill opacity=1 ] (220.03,516.17) .. controls (220.03,515.34) and (220.7,514.67) .. (221.53,514.67) .. controls (222.36,514.67) and (223.03,515.34) .. (223.03,516.17) .. controls (223.03,517) and (222.36,517.67) .. (221.53,517.67) .. controls (220.7,517.67) and (220.03,517) .. (220.03,516.17) -- cycle ;
\draw  [fill={rgb, 255:red, 0; green, 0; blue, 0 }  ,fill opacity=1 ] (220.03,508.17) .. controls (220.03,507.34) and (220.7,506.67) .. (221.53,506.67) .. controls (222.36,506.67) and (223.03,507.34) .. (223.03,508.17) .. controls (223.03,509) and (222.36,509.67) .. (221.53,509.67) .. controls (220.7,509.67) and (220.03,509) .. (220.03,508.17) -- cycle ;
\draw    (200.53,508.17) -- (221.53,499.17) ;
\draw    (312.5,508.17) -- (221.53,508.17) ;
\draw    (200.53,508.17) -- (221.53,508.17) ;
\draw    (314,508.17) -- (221.53,499.17) ;
\draw    (200.53,508.17) -- (221.53,516.17) ;
\draw    (312.5,508.17) -- (221.53,516.17) ;
\draw  [fill={rgb, 255:red, 0; green, 0; blue, 0 }  ,fill opacity=1 ] (199.03,538.83) .. controls (199.03,538) and (199.7,537.33) .. (200.53,537.33) .. controls (201.36,537.33) and (202.03,538) .. (202.03,538.83) .. controls (202.03,539.66) and (201.36,540.33) .. (200.53,540.33) .. controls (199.7,540.33) and (199.03,539.66) .. (199.03,538.83) -- cycle ;
\draw  [fill={rgb, 255:red, 0; green, 0; blue, 0 }  ,fill opacity=1 ] (220.03,529.83) .. controls (220.03,529) and (220.7,528.33) .. (221.53,528.33) .. controls (222.36,528.33) and (223.03,529) .. (223.03,529.83) .. controls (223.03,530.66) and (222.36,531.33) .. (221.53,531.33) .. controls (220.7,531.33) and (220.03,530.66) .. (220.03,529.83) -- cycle ;
\draw  [fill={rgb, 255:red, 0; green, 0; blue, 0 }  ,fill opacity=1 ] (314,538.83) .. controls (314,538) and (313.33,537.33) .. (312.5,537.33) .. controls (311.67,537.33) and (311,538) .. (311,538.83) .. controls (311,539.66) and (311.67,540.33) .. (312.5,540.33) .. controls (313.33,540.33) and (314,539.66) .. (314,538.83) -- cycle ;
\draw  [fill={rgb, 255:red, 0; green, 0; blue, 0 }  ,fill opacity=1 ] (220.03,546.83) .. controls (220.03,546) and (220.7,545.33) .. (221.53,545.33) .. controls (222.36,545.33) and (223.03,546) .. (223.03,546.83) .. controls (223.03,547.66) and (222.36,548.33) .. (221.53,548.33) .. controls (220.7,548.33) and (220.03,547.66) .. (220.03,546.83) -- cycle ;
\draw  [fill={rgb, 255:red, 0; green, 0; blue, 0 }  ,fill opacity=1 ] (220.03,538.83) .. controls (220.03,538) and (220.7,537.33) .. (221.53,537.33) .. controls (222.36,537.33) and (223.03,538) .. (223.03,538.83) .. controls (223.03,539.66) and (222.36,540.33) .. (221.53,540.33) .. controls (220.7,540.33) and (220.03,539.66) .. (220.03,538.83) -- cycle ;
\draw    (200.53,538.83) -- (221.53,529.83) ;
\draw    (312.5,538.83) -- (221.53,538.83) ;
\draw    (200.53,538.83) -- (221.53,538.83) ;
\draw    (314,538.83) -- (221.53,529.83) ;
\draw    (200.53,538.83) -- (221.53,546.83) ;
\draw    (312.5,538.83) -- (221.53,546.83) ;
\draw  [fill={rgb, 255:red, 0; green, 0; blue, 0 }  ,fill opacity=1 ] (199.03,569.5) .. controls (199.03,568.67) and (199.7,568) .. (200.53,568) .. controls (201.36,568) and (202.03,568.67) .. (202.03,569.5) .. controls (202.03,570.33) and (201.36,571) .. (200.53,571) .. controls (199.7,571) and (199.03,570.33) .. (199.03,569.5) -- cycle ;
\draw  [fill={rgb, 255:red, 0; green, 0; blue, 0 }  ,fill opacity=1 ] (220.03,560.5) .. controls (220.03,559.67) and (220.7,559) .. (221.53,559) .. controls (222.36,559) and (223.03,559.67) .. (223.03,560.5) .. controls (223.03,561.33) and (222.36,562) .. (221.53,562) .. controls (220.7,562) and (220.03,561.33) .. (220.03,560.5) -- cycle ;
\draw  [fill={rgb, 255:red, 0; green, 0; blue, 0 }  ,fill opacity=1 ] (314,569.5) .. controls (314,568.67) and (313.33,568) .. (312.5,568) .. controls (311.67,568) and (311,568.67) .. (311,569.5) .. controls (311,570.33) and (311.67,571) .. (312.5,571) .. controls (313.33,571) and (314,570.33) .. (314,569.5) -- cycle ;
\draw  [fill={rgb, 255:red, 0; green, 0; blue, 0 }  ,fill opacity=1 ] (220.03,577.5) .. controls (220.03,576.67) and (220.7,576) .. (221.53,576) .. controls (222.36,576) and (223.03,576.67) .. (223.03,577.5) .. controls (223.03,578.33) and (222.36,579) .. (221.53,579) .. controls (220.7,579) and (220.03,578.33) .. (220.03,577.5) -- cycle ;
\draw  [fill={rgb, 255:red, 0; green, 0; blue, 0 }  ,fill opacity=1 ] (220.03,569.5) .. controls (220.03,568.67) and (220.7,568) .. (221.53,568) .. controls (222.36,568) and (223.03,568.67) .. (223.03,569.5) .. controls (223.03,570.33) and (222.36,571) .. (221.53,571) .. controls (220.7,571) and (220.03,570.33) .. (220.03,569.5) -- cycle ;
\draw    (200.53,569.5) -- (221.53,560.5) ;
\draw    (312.5,569.5) -- (221.53,569.5) ;
\draw    (200.53,569.5) -- (221.53,569.5) ;
\draw    (314,569.5) -- (221.53,560.5) ;
\draw    (200.53,569.5) -- (221.53,577.5) ;
\draw    (312.5,569.5) -- (221.53,577.5) ;

\draw (90,507.8) node [anchor=north west][inner sep=0.75pt]   [align=left] {$L_1$};
\draw (90,640) node [anchor=north west][inner sep=0.75pt]   [align=left] {$S_1$};
\draw (160,688.8) node [anchor=north west][inner sep=0.75pt]   [align=left] {$V_1$};
\draw (319,647.8) node [anchor=north west][inner sep=0.75pt]   [align=left] {$V_2$};
\draw (314.5,512.67) node [anchor=north west][inner sep=0.75pt]   [align=left] {$Y$};
\end{tikzpicture}
\caption{Supplementary diagram for Claim~\ref{CLAIM:3rd-graph-Y-upper-bound} when $F = K_{2,3}$.}
\label{fig:3rd-part}
\end{figure}

    Let 
    \begin{align*}
        Y
        := \left\{v\in V_2 \colon |N_{G}(v) \cap L_1| \ge s_1 s\right\}. 
    \end{align*}
  \begin{claim}\label{CLAIM:3rd-graph-Y-upper-bound}
        We have $|Y| \le s-1$. 
   \end{claim}
   \begin{proof}
        Suppose to the contrary that $|Y| \ge s$. 
        Fix $s$ vertices $u_1, \ldots, u_s \in Y$. 
        We will show that there exists a collection $\mathcal{B}' = \{B_1', \ldots, B_s'\}$ of pairwise vertex-disjoint copies of $\mathbb{B}$ in $G[L_1\cup V_2]$ such that $u_i \in B_i'$ for $i\in [s]$. 
        Indeed, it follows from the definition of $Y$ that there exists an $s_1$-set $N_i \subset N_{G}(u_i) \cap L_1 \subset L$ for every $i\in [s]$ such that sets $N_1, \ldots, N_s$ are pairwise disjoint. 
        It follows from the definition of $L$ that 
        \begin{align*}
            |\bigcap_{v\in N_i} N_{G}(v) \cap L_1|
            & \ge st - s_1\left(st - \alpha (st-1)\right) - |S_1| \\
            & \ge  \left( 1-s_1(1-\alpha)\right) st - s_1s - \frac{4s(n-st)}{1-\alpha}
             \ge s_2s. 
        \end{align*}
        Then, a simple greedy argument shows that such a collection $\mathcal{B}'$ exists. 

        Let $L':= L_1\setminus \bigcup_{i \in [s]}V(B_i')$. Observe that 
        $|L_1|= |L|-|S_1| - s(s-1)$ is divisible by $s$. 
        In addition, it follows from the definition of $L$ that
        \begin{align*}
            \delta(G[L_1'])
            & \ge \alpha(st-1) - \left(|S_1|+s(s-1)\right) \\
            & \ge \alpha(st-1) - \left(\frac{4s(n-st)}{1-\alpha}+s(s-1)\right)  \\
            & \ge \left(1-\frac{1}{2s_1}\right)(st-1) 
                    - \frac{2s_1\cdot 4s \cdot \varepsilon s}{1-\varepsilon s}st - s(s-1) \\
            & > \left(1-\frac{1}{2s_1}\right)(st-1) 
                    - \frac{1}{8}st - s(s-1)
            > \left(\frac{1}{2}+\frac{1}{16}\right)(st-1). 
        \end{align*}
        Here, we used the assumption that $t \ge \frac{n}{s} - \varepsilon n$, $s_1 \ge 2$, and $\varepsilon = \frac{1}{65s_1s^2}$. 
        So it follows from Theorem~\ref{THM:AY-graph-factor} that $G[L_1']$ contains a perfect $B_{s_1,s_2}$-matching, i.e. $\frac{|L_1|}{s} \mathbb{B} \subset G[L_1']$. 
        These copies of $\mathbb{B}$ together with $\mathcal{B}_{S} \cup \mathcal{B}'$ shows that $(t+1)\mathbb{B} \subset G$ (see Figure~\ref{fig:3rd-part}), which is a contradiction. 
   \end{proof}
    Claim~\ref{CLAIM:3rd-graph-Y-upper-bound} implies that the induced bipartite graph $G[L_1, V_2]$ of $G$ on $L_1$ and $V_2$ satisfies 
    \begin{align}\label{equ:3rd-graph-G-L1V2}
        |G[L_1, V_2]|
        \le (s-1)|L_1| + \left(|V_2|-s+1\right)(s_1s-1)
        < (s-1)st + s_1s (n-st). 
    \end{align}
    It follows from~\eqref{equ:3rd-G2-upper-bound},~\eqref{equ:3rd-G1-upper-bound},~\eqref{equ:3rd-graph-S1-upper-bound}, and ~\eqref{equ:3rd-graph-G-L1V2} that 
    \begin{align*}
        |G|
        & = |G_1| + |G_2| + |G[L_1, V_2]| + |G[S_1, V_2]| \\
        & \le \binom{st}{2} - \frac{st-\ell}{2}(1-\alpha)(st-1) 
                + \mathrm{ex}(n-st, \mathbb{B}) \\
        &\quad     + (s-1)st + s_1s (n-st) 
                + s(st-\ell)(n-st) \\
        & \le \binom{st}{2} + \mathrm{ex}(n-st, \mathbb{B}) + s_1s n 
                - \left(\frac{1-\alpha}{2}(st-1) - s(n-st)\right). 
    \end{align*}
    Since $t \ge \frac{n}{s} - \varepsilon n$ and $\varepsilon = \frac{1}{65s_1s^2}$, we have 
    \begin{align*}
        \frac{1-\alpha}{2}(st-1) - s(n-st)
        \ge \frac{1}{4s_1}(1-\varepsilon s)n - 1 - s \cdot \varepsilon s n 
        > 0. 
    \end{align*}
    Consequently, 
    \begin{align*}
        |G|
         < \binom{st}{2} + \mathrm{ex}\left(n-st, \mathbb{B}\right) + s_1s n 
         \le \binom{\hat{t}}{2} + \mathrm{ex}\left(n-\hat{t}, \mathbb{B}\right) + s_1sn,  
    \end{align*}
    completing the proof of Theorem~\ref{THM:3rd-interval-graph}. 
\end{proof}
\subsection{Hypergraphs: proof of Theorem~\ref{THM:3rd-interval}}\label{SUBSEC:proof-3rd-interval-hygp}
The proof strategy for Theorem~\ref{THM:3rd-interval} is similar to that of Theorem~\ref{THM:3rd-interval-graph},  except in the hypergraph case we substitute Claim~\ref{CLAIM:3rd-graph-Y-upper-bound} with a less precise estimate, as presented in Claim~\ref{CLAIM:3rd-HL1V2}.
\begin{proof}[Proof of Theorem~\ref{THM:3rd-interval}]
Fix integers $r\ge 3$, $s_r \ge \cdots \ge s_1\ge 2$, and an $r$-partite $r$-graph $\mathbb{B} := B_{s_1, \ldots, s_r}^{r}$  with part sizes $s_1, \ldots, s_r$.  
Let 
\begin{align*}
    s:= \sum_{i\in [r]}s_i, \quad 
    A:= 2er^2\prod_{i\in [r]}s_i,
    \quad\text{and}\quad
    \varepsilon := \frac{1}{8er^2s^2A}. 
\end{align*}
Let $n$ be a sufficiently large integer and $t \in \left[\frac{n}{s}-\varepsilon n, \frac{n}{s}\right]$ be an integer. 
Let $\mathcal{H}$ be a maximum $n$-vertex $(t+1)\mathbb{B}$-free $r$-graph. 
Let $\hat{t} := s(t+1)-1$ for simplicity. 
It follows from Proposition~\ref{PROP:three-lower-bounds} and Fact~\ref{FACT:increasing-f(ell)} that 
\begin{align}\label{equ:3rd-lower-bound-H}
    |\mathcal{H}|
    \ge g_3(n,t,\mathbb{B})
    = \binom{\hat{t}}{r} 
        + \mathrm{ex}\left(n- \hat{t}, \mathbb{B}\right) 
    \ge \binom{st}{r} + \mathrm{ex}(n-st, \mathbb{B}). 
\end{align}
Let $\mathcal{B} = \{\mathbb{B}_1, \ldots, \mathbb{B}_{t}\}$ be a collection of vertex-disjoint copies of $\mathbb{B}$ in $\mathcal{H}$ (the existence of such a collection follows from the maximality of $\mathcal{H}$). 
Let 
\begin{align*}
    V:= V(\mathcal{H}), \quad 
    V_1 := \bigcup_{i\in [t]}V(\mathbb{B}_i), \quad 
    V_2:= V \setminus V_1, \quad\text{and}\quad 
    \mathcal{H}_i := \mathcal{H}[V_i]\quad\text{for $i\in \{1,2\}$}. 
\end{align*}
Observe that $|V_1| = st$ and $\mathcal{H}_2$ is $\mathbb{B}$-free. 
Hence, by Theorem~\ref{THM:Erdos-hypergraph-KST}, we have 
\begin{align}\label{equ:3rd-H2}
    |\mathcal{H}_2| 
    \le \mathrm{ex}\left(n-st, \mathbb{B}\right)
    \le C (n-st)^{r-\frac{1}{s_1\cdots s_{r-1}}}. 
\end{align}
%
Since $n-st \le \varepsilon s n \le \frac{n}{2r}$, it follows from Fact~\ref{LEMMA:binom-inequ-b} that 
\begin{align}\label{equ:3rd-H12}
    |\mathcal{H}[V_1, V_2]| 
    \le \sum_{v\in V_2}d_{\mathcal{H}}(v)
    \le (n-st)\binom{n-1}{r-1} 
    < e (n-st)\binom{st-1}{r-1}. 
\end{align}

Let 
\begin{align*}
    \alpha := 1-\frac{1}{2A}, \quad
    L:= \left\{v\in V_1 \colon d_{\mathcal{H}_1}(v) \ge \alpha \binom{st-1}{r-1}\right\}, 
    \quad
    S:= V_1\setminus L, \quad\text{and}\quad 
    \ell := |L|. 
\end{align*}
It follows from the definition of $L$ that 
\begin{align}\label{equ:3rd-H1}
    |\mathcal{H}_1|
    & = \frac{1}{r}\left(\sum_{v\in L}d_{\mathcal{H}_1}(v) + \sum_{v\in S}d_{\mathcal{H}_1}(v)\right) \notag \\
    & \le \frac{1}{r}\left(\ell \binom{st-1}{r-1} + (st-\ell)\alpha \binom{st-1}{r-1}\right)
    = \binom{st}{r} - (1-\alpha)\frac{st-\ell}{r} \binom{st-1}{r-1}. 
\end{align}
\begin{claim}\label{CLAIM:3rd-L-lower-bound}
    We have $|S| < \frac{er(n-st)}{1-\alpha}$. 
\end{claim}
\begin{proof}
    It follows from~\eqref{equ:3rd-H2},~\eqref{equ:3rd-H12}, and~\eqref{equ:3rd-H1} that 
    \begin{align*}
        |\mathcal{H}|
        & = |\mathcal{H}_1| + |\mathcal{H}[V_1, V_2]| + |\mathcal{H}_2| \\
        & <  \binom{st}{r} - (1-\alpha)\frac{st-\ell}{r} \binom{st-1}{r-1}       + e (n-st)\binom{st-1}{r-1}
            + \mathrm{ex}\left(n-st, \mathbb{B}\right).  
    \end{align*}
    Combined with~\eqref{equ:3rd-lower-bound-H}, we obtain 
    \begin{align*}
        - (1-\alpha)\frac{st-\ell}{r} \binom{st-1}{r-1}       + e (n-st)\binom{st-1}{r-1} 
        > 0, 
    \end{align*}
    which implies that $|S| = st-\ell <  \frac{er(n-st)}{1-\alpha}$. 
\end{proof}

Let $\mathcal{B}_{S}$ be the collection of elements in $\mathcal{B}$ that have nonempty intersection with $S$. Let 
\begin{align*}
    S_1 
    := \bigcup_{\mathbb{B}_i \in \mathcal{B}_{S}}V(\mathbb{B}_i)  
    \quad\text{and}\quad
    L_1 := V_1 \setminus S_1. 
\end{align*}
Observe that $S\subset S_1$ and $L_1\subset L$. 
Also observe from Claim~\ref{CLAIM:3rd-L-lower-bound} that 
\begin{align}\label{equ:3rd-S1-upper-bound}
    |S_1| \le s|S| \le \frac{ers(n-st)}{1-\alpha}. 
\end{align}
%
%
\begin{claim}\label{CLAIM:3rd-HL1V2}
   The $r$-graph $\mathcal{H}':= \mathcal{H}[L_1\cup V_2] \setminus \mathcal{H}[L_1]$ is $s  \mathbb{B}$-free. 
   In particular, we have 
   \begin{align*}
       |\mathcal{H}'|
       \le \mathrm{ex}_{\mathrm{star}}(n-st, n, s \mathbb{B}). 
   \end{align*}
\end{claim}
\begin{proof}
    Suppose to the contrary that there exists a collection $\mathcal{B}'= \{\mathbb{B}_1', \ldots, \mathbb{B}_{s}'\}$ of pairwise vertex-disjoint copies of $\mathbb{B}$ in $\mathcal{H}'$. 
    Let $W:= \bigcup_{i\in [s]}V(\mathbb{B}'_i)$, $W_1:= W \cap L_1$, and $W_2:= W\cap V_2$. 
    Observe that every $\mathbb{B}_i'$ contains at least one vertex from $V_2$ (since $L_1$ is an independent set in $\mathcal{H}'$). 
    So we have $|W_1| \le (s-1)s$. 
    It follows that 
    \begin{align*}
        |V_1\setminus W_1|
        \ge st - (s-1)s
        = s\left(t+1-s\right). 
    \end{align*}
    Fix an arbitrary subset $L_1' \subset L_1\setminus W_1$ of size exactly $s\left(t+1-s\right) - |S_1|$. 
    It follows from the definition of $L$ that the induced subgraph $\mathcal{H}[L_1']$ satisfies 
    \begin{align*}
        \delta(\mathcal{H}[L_1'])
        & \ge \alpha \binom{st-1}{t-1} - |V_1\setminus L_1'| \binom{st-2}{r-2} \\
        & \ge \alpha \binom{st-1}{t-1} - \left(|S_1| + (s-1)s\right)\frac{r-1}{st-1} \binom{st-1}{r-1} \\
        & > \left(\alpha - \left(\frac{ers(n-st)}{1-\alpha}+ s^2\right)\frac{r-1}{st-1}\right) \binom{st-1}{r-1},
    \end{align*}
    where the last inequality follows from~\eqref{equ:3rd-S1-upper-bound}. 
    Since $n$ is sufficiently large and $t \ge n/s - \varepsilon n$, simple calculations show that $\left(\frac{ers(n-st)}{1-\alpha}+ s^2\right)\frac{r-1}{st-1} < \frac{1}{2A}$. 
    Therefore, we have $\delta(\mathcal{H}[L_1']) \ge \left(1-\frac{1}{A}\right)\binom{st-1}{r-1}$. 
    It follows from Theorem~\ref{THM:Packing-F-mindegree} that $\left(t+1-s - |S_1|/s\right)\mathbb{B} \subset \mathcal{H}[L_1']$. 
    However, these $t+1-s - |S_1|/s$ copies of $\mathbb{B}$ together with $\mathcal{B}_S\cup \mathcal{B}'$ show that $(t+1)\mathbb{B} \subset \mathcal{H}$, a contradiction. 
\end{proof}

Let $\mathcal{H}'' := \left\{e\in \mathcal{H} \colon e\cap S_1 \neq \emptyset \text{ and } e\cap V_2 \neq \emptyset\right\}$. 
Using Fact~\ref{LEMMA:binom-inequ-b} and Claim~\ref{CLAIM:3rd-L-lower-bound}, we obtain 
\begin{align*}
    |\mathcal{H}''|
    \le |S_1||V_2|\binom{n-2}{r-2}
    & \le s |S| \times (n-st) \times e \binom{st-2}{r-2} \\
    & \le es(n-st)|S| \times \frac{r}{st} \binom{st-1}{r-1} \\
    & = \frac{\varepsilon er^2 s n}{(1/s-\varepsilon)n} \frac{st-\ell}{r} \binom{st-1}{r-1}
    \le \frac{1}{4A} \frac{st-\ell}{r} \binom{st-1}{r-1}, 
\end{align*}
where the last inequality follows from $\varepsilon = \frac{1}{8er^2s^2A}$. 
Combined with~\eqref{equ:3rd-H1} and Claim~\ref{CLAIM:3rd-HL1V2}, we obtain 
\begin{align*}
    |\mathcal{H}|
    & = |\mathcal{H}_1|  + |\mathcal{H}''| + |\mathcal{H}'| \\
    & \le \binom{st}{r} - (1-\alpha)\frac{st-\ell}{r} \binom{st-1}{r-1}
      + \frac{1}{4A} \frac{st-\ell}{r} \binom{st-1}{r-1} 
        + \mathrm{ex}_{\mathrm{star}}(n-st, n, s \mathbb{B})\\
    & = \binom{st}{r} - \left(\frac{1}{2A} - \frac{1}{4A}\right)\frac{st-\ell}{r} \binom{st-1}{r-1} + + \mathrm{ex}_{\mathrm{star}}(n-st, n, s \mathbb{B}) \\
    & < \binom{s(t+1)-1}{r} + + \mathrm{ex}_{\mathrm{star}}(n-st, n, s \mathbb{B}), 
\end{align*}
completing the proof of Theorem~\ref{THM:3rd-interval}. 
\end{proof}
\section{Proofs for theorems in the second interval}\label{SEC:proof-2nd-interval}
In this section, we prove Theorems~\ref{THM:2nd-interval} and~\ref{THM:2nd-interval-graph}. 
Let us start with some technical lemmas. 
\subsection{Preparations}
Using Proposition~\ref{PROP:hypergraph-KST-Zaran}, we establish the following lemma, ensuring a near-maximum $K_{s_1, \ldots, s_r}^{r}$-matching in a semibipartite $r$-graph with minimum degree of $\Omega(n^{r-1})$ on one side.

\begin{lemma}\label{LEMMA:2nd-near-perfect-K-semibipartite}
  Let $r \ge 2$ and $s_r \ge \cdots \ge s_1 \ge 2$ be integers and let $\alpha > 0$ be a real number. 
  Let $s:= s_1 + \cdots + s_r$ and $\mathbb{K}:= K_{s_1, \ldots, s_r}^{r}$. 
  The following holds for sufficiently large $n$. 
  Suppose that $\mathcal{H}$ is an $m$ by $n$ semibipartite $r$-graph on $V_1$ and $V_2$ with $m \le \frac{\alpha n}{s-s_1}$ and $d_{\mathcal{H}}(v) \geq \alpha n^{r-1}$ for every $v \in V_1$.
  Then $\mathcal{H}$ contains at least $\left \lfloor  \frac{m}{s_1} - \frac{4}{\alpha}\right \rfloor$  pairwise vertex-disjoint copies of ordered $\mathbb{K}$. 
\end{lemma}
\begin{proof}
    Let $n$ be sufficiently large and $\mathcal{H}$ be an $m$ by $n$ semibipartite $r$-graph on $V_1$ and $V_2$ with $n \geq (s-s_1)m/\alpha$ and $d_{\mathcal{H}}(v) \geq \alpha n^{r-1}$ for every $v \in V_1$. 
    We may assume that $m \ge \frac{4s_1}{\alpha}$, since otherwise $\left \lfloor  \frac{m}{s_1} - \frac{4}{\alpha}\right \rfloor \le 0$, and there is nothing to prove. 
    Let $\mathcal{K} = \{K_1, \ldots, K_{\ell}\}$ be a maximum collection of pairwise vertex-disjoint copies of ordered $K$ in $\mathcal{H}$.
    Let $B:= \bigcup_{i\in [\ell]}V(K_i)$, $B_i:= B\cap V_i$, and $V_i':= V_i\setminus B_i$ for $i\in \{1,2\}$. 
    Let $\mathcal{H}'$ be the induced subgraph of $\mathcal{H}$ on $V_1'\cup V_2'$ and note that $\mathcal{H'}$ is also semibipartite. 
    Suppose to the contrary that $\ell \le \frac{m}{s_1} - \frac{4}{\alpha}$. 
    Then $|B_1| = s_1 \ell \le m- \frac{4s_1}{\alpha}$ and $|B_2| = (s-s_1)\ell \le (s-s_1)\frac{m}{s_1} \le \frac{\alpha n}{s_1}$. 
    Therefore, $|V_1'| = m- |B_1| \ge \frac{4s_1}{\alpha}$ and 
    \begin{align*}
        |\mathcal{H}'|
        = \sum_{v\in V_1'}d_{\mathcal{H}'}(v)
        & \ge \sum_{v\in V_1'}\left(d_{\mathcal{H}}(v) - |B_2|\binom{n-1}{r-2}\right)  \\
        &  \ge |V_1'| \left(\alpha n^{r-1} - \frac{\alpha n}{s_1} \binom{n-1}{r-2}\right) \\
        & \ge |V_1'| \times \frac{\alpha}{2}n^{r-1} \\
        & = |V_1'| \times \frac{\alpha}{4}n^{r-1} 
         + |V_1'| \times \frac{\alpha}{4}n^{r-1} 
         \ge C |V_1'| n^{r-1-\frac{1}{s_1\cdots s_{r-1}}} + s_1 n^{r-1}. 
    \end{align*}
    It follows from Proposition~\ref{PROP:hypergraph-KST-Zaran} that $\mathcal{H}'$ contains an ordered copy of $K$, contradicting the maximality of $\mathcal{K}$. 
\end{proof}

If we extend the assumptions in Lemma~\ref{LEMMA:2nd-near-perfect-K-semibipartite} by assuming the number of vertices with near-maximum degree is at least some constant, then we can enhance the lemma by finding a maximum $K_{s_1, \ldots, s_r}^{r}$-matching. 
The proof relies on a simple absorption strategy, pairing a low-degree vertex with $s_1-1$ high-degree vertices.

\begin{figure}[htbp]
\centering
\tikzset{every picture/.style={line width=0.85pt}} 
\begin{tikzpicture}[x=0.75pt,y=0.75pt,yscale=-1,xscale=1]
\draw [line width=1pt]  (160.95,149.32) .. controls (160.95,143.29) and (165.84,138.4) .. (171.87,138.4) -- (524.41,138.4) .. controls (530.44,138.4) and (535.33,143.29) .. (535.33,149.32) -- (535.33,182.08) .. controls (535.33,188.11) and (530.44,193) .. (524.41,193) -- (171.87,193) .. controls (165.84,193) and (160.95,188.11) .. (160.95,182.08) -- cycle ;
\draw [line width=1pt]  (159.79,57.95) .. controls (159.79,54.11) and (162.9,51) .. (166.74,51) -- (527.22,51) .. controls (531.06,51) and (534.17,54.11) .. (534.17,57.95) -- (534.17,78.8) .. controls (534.17,82.64) and (531.06,85.75) .. (527.22,85.75) -- (166.74,85.75) .. controls (162.9,85.75) and (159.79,82.64) .. (159.79,78.8) -- cycle ;
 \draw    (305,51) -- (305,85.75) ;
\draw    (176.43,71.11) -- (176.43,171.63) ;
\draw    (192.69,71.11) -- (192.69,171.63) ;
\draw    (208.96,71.11) -- (208.96,171.63) ;
\draw    (176.43,71.11) -- (192.69,171.63) ;
\draw [fill=uuuuuu]   (176.43,71.11)  circle (1.5pt);
\draw [fill=uuuuuu]   (176.43,171.63)  circle (1.5pt);
\draw [fill=uuuuuu]   (192.69,71.11) circle (1.5pt);
\draw [fill=uuuuuu]   (192.69,171.63)  circle (1.5pt);
\draw [fill=uuuuuu]   (208.96,71.11)  circle (1.5pt);
\draw [fill=uuuuuu]   (208.96,171.63)  circle (1.5pt);
%
\draw    (192.69,71.11) -- (208.96,171.63) ;
\draw    (208.96,71.11) -- (192.69,171.63) ;
\draw    (208.96,71.11) -- (176.43,171.63) ;
\draw    (192.69,71.11) -- (176.43,171.63) ;
\draw    (176.43,71.11) -- (208.96,171.63) ;
%
\draw [fill=uuuuuu]   (487.71,71.11)  circle (1.5pt);
\draw [fill=uuuuuu]   (487.71,171.63)  circle (1.5pt);
\draw [fill=uuuuuu]   (503.97,71.11) circle (1.5pt);
\draw [fill=uuuuuu]   (503.97,171.63)  circle (1.5pt);
\draw [fill=uuuuuu]   (520.23,71.11)  circle (1.5pt);
\draw [fill=uuuuuu]   (520.23,171.63)  circle (1.5pt);
\draw    (487.71,71.11) -- (487.71,171.63) ;
\draw    (503.97,71.11) -- (503.97,171.63) ;
\draw    (520.23,71.11) -- (520.23,171.63) ;
\draw    (487.71,71.11) -- (503.97,171.63) ;
\draw    (503.97,71.11) -- (520.23,171.63) ;
\draw    (520.23,71.11) -- (503.97,171.63) ;
\draw    (520.23,71.11) -- (487.71,171.63) ;
\draw    (503.97,71.11) -- (487.71,171.63) ;
\draw    (487.71,71.11) -- (520.23,171.63) ;
\draw [fill=uuuuuu]   (440.09,72.08)  circle (1.5pt);
\draw [fill=uuuuuu]   (440.09,172.6)  circle (1.5pt);
\draw [fill=uuuuuu]   (456.35,72.08)  circle (1.5pt);
\draw [fill=uuuuuu]   (456.35,172.6)  circle (1.5pt);
\draw [fill=uuuuuu]   (472.61,72.08)  circle (1.5pt);
\draw [fill=uuuuuu]   (472.61,172.6)  circle (1.5pt);
%
\draw    (440.09,72.08) -- (440.09,172.6) ;
\draw    (456.35,72.08) -- (456.35,172.6) ;
\draw    (472.61,72.08) -- (472.61,172.6) ;
\draw    (440.09,72.08) -- (456.35,172.6) ;
\draw    (456.35,72.08) -- (472.61,172.6) ;
\draw    (472.61,72.08) -- (456.35,172.6) ;
\draw    (472.61,72.08) -- (440.09,172.6) ;
\draw    (456.35,72.08) -- (440.09,172.6) ;
\draw    (440.09,72.08) -- (472.61,172.6) ;
\draw [fill=uuuuuu]    (278.64,70.13) circle (1.5pt);
\draw [fill=uuuuuu]   (278.64,170.65)  circle (1.5pt);
\draw [fill=uuuuuu]   (294.91,70.13) circle (1.5pt);
\draw [fill=uuuuuu]   (294.91,170.65)  circle (1.5pt);
\draw [fill=uuuuuu]   (376.4,73.06)  circle (1.5pt);
\draw [fill=uuuuuu]   (376.4,166.75)  circle (1.5pt);
%
\draw    (278.64,70.13) -- (278.64,170.65) ;
\draw    (294.91,70.13) -- (294.91,170.65) ;
\draw    (278.64,70.13) -- (294.91,170.65) ;
\draw    (294.91,70.13) -- (278.64,170.65) ;
\draw    (376.4,73.06) -- (294.91,170.65) ;
\draw    (376.4,73.06) -- (376.4,166.75) ;
\draw    (294.91,70.13) -- (376.4,166.75) ;
\draw    (376.4,73.06) -- (278.64,170.65) ;
\draw    (278.64,70.13) -- (376.4,166.75) ;
\draw (128,62.78) node [anchor=north west][inner sep=0.75pt]   [align=left] {$V_1$};
\draw (128,162.33) node [anchor=north west][inner sep=0.75pt]   [align=left] {$V_2$};
\draw (225.42,30) node [anchor=north west][inner sep=0.75pt]   [align=left] {$L$};
\draw (409.9,30) node [anchor=north west][inner sep=0.75pt]   [align=left] {$S$};
\end{tikzpicture}

\caption{Supplementary diagram for Claim~\ref{LEMMA:2nd-perfect-K-semibipartite} when $F = K_{3,3}$.}
\label{fig:Lemma5.2}
\end{figure}

\begin{lemma}\label{LEMMA:2nd-perfect-K-semibipartite}
     Let $r \ge 2$ and $s_r \ge \cdots \ge s_1 \ge 2$ be integers and $\alpha > 0$ be a real number. 
    Let $s:= s_1 + \cdots + s_r$ and $\mathbb{K}:= K_{s_1, \ldots, s_r}^{r}$.  
    The following statement holds for sufficiently large $n$. 
    Suppose that $\mathcal{H}$ is an $m$ by $n$ semibipartite $r$-graph on $V_1$ and $V_2$ that satisfies 
    \begin{enumerate}[label=(\roman*)]
        \item\label{LEMMA:2nd-perfect-K-semibipartite-1} 
                $m \le \frac{\alpha n}{8(s-s_1)}$, 
        \item\label{LEMMA:2nd-perfect-K-semibipartite-2}            
                $d_{\mathcal{H}}(v) \geq \alpha n^{r-1}$ for every $v \in V_1$, and 
        \item\label{LEMMA:2nd-perfect-K-semibipartite-3}  
                $L := \left\{v \in V_1 \colon d(v) \geq \binom{n}{r-1} - \frac{\alpha n^{r-1}}{2s_1} \right\}$
        has size at least $\min\left\{\frac{5s_1(s_1-1)}{\alpha},\ \frac{s_1-1}{s_1}m\right\}$. 
    \end{enumerate}
    Then $\mathcal{H}$ contains $\left \lfloor  m/s_1\right \rfloor$  pairwise vertex-disjoint copies of ordered $\mathbb{K}$.
\end{lemma}
\begin{proof}
    Let $S:= V_1 \setminus L$, $\mathcal{H}_{S} := \mathcal{H}[S\cup V_2]$, and $\mathcal{K}_{S} = \{\mathbb{K}_1, \ldots, \mathbb{K}_{\ell}\}$ be a maximum collection of ordered copies of $\mathbb{K}$ in $\mathcal{H}_{S}$. 
    Let 
    \begin{align*}
        B_{S} := \bigcup_{i\in [\ell]}V(\mathbb{K}_i), \quad
        V_2':= V_2 \setminus B_S, \quad 
        S_1 := S\setminus B_{S}, \quad\text{and}\quad 
        m_1:= |S_1|. 
    \end{align*}
    It follows from Lemma~\ref{LEMMA:2nd-near-perfect-K-semibipartite} that $\ell \ge \max\left\{\left\lfloor \frac{|S|}{s_1} - \frac{4}{\alpha} \right\rfloor,\ 0\right\}$, and hence, 
    \begin{align}\label{equ:LEMMA:2nd-perfect-K-semibipartite-m1-upper-bound}
        m_1 
        = |S| - s_1 \ell 
        \le |S| -  s_1 \times \max\left\{\left\lfloor \frac{|S|}{s_1} - \frac{4}{\alpha} \right\rfloor,\ 0\right\}
        \le \min\left\{\frac{5s_1}{\alpha},\ |S|\right\}. 
    \end{align}
    Let us assume that $S_1:= \{v_1, \ldots, v_{m_1}\}$. 
    Let $T_1, \ldots, T_{m_1}$ be pairwise disjoint $(s_1-1)$-subsets of $L$ (the existence of such $T_i$'s is guaranteed by~\ref{LEMMA:2nd-perfect-K-semibipartite-3} and~\eqref{equ:LEMMA:2nd-perfect-K-semibipartite-m1-upper-bound}).
  Let $\mathcal{H}_i$ denote the induced subgraph of $\mathcal{H}$ on $\{v_i\} \cup T_i \cup V_2'$ for $i\in [m_1]$. 
  
  Let $B_0:= \emptyset$. 
  Suppose that we have defined $B_i \subset V_2$ of size $i(s-s_1)$ for some $i\in [0,m_1-1]$. 
  We will find an ordered copy of $\mathbb{K}$ in $\mathcal{H}_{i+1}':= \mathcal{H}_{i+1}-B_i$. 
  Indeed, notice from~\ref{LEMMA:2nd-perfect-K-semibipartite-1}  that 
  \begin{align*}
      |B_i \cup B_S| 
      \le \frac{m}{s_1}(s-s_1) 
      \le \frac{\alpha n}{4}. 
  \end{align*}
    Combined with~\ref{LEMMA:2nd-perfect-K-semibipartite-2}  and~\ref{LEMMA:2nd-perfect-K-semibipartite-3}, we see that the common link $L_{\mathcal{H}_{i+1}'}(T_{i+1}\cup \{v_{i+1}\})$ satisfies 
  \begin{align*}
      L_{\mathcal{H}_{i+1}'}(T_{i+1}\cup \{v_{i+1}\})
      & \ge d_{\mathcal{H}}(v_{i+1}) - \sum_{u\in T_{i+1}}\left(\binom{n}{r-1}-d_{\mathcal{H}}(u)\right) - |B_i \cup B_S|\binom{n-1}{r-2} \\
      & \ge \alpha n^{r-1} - (s_1-1)\frac{\alpha n^{r-1}}{2s_1} - \frac{\alpha n}{4}\binom{n-1}{r-2} 
      \ge \frac{\alpha n^{r-1}}{4}. 
  \end{align*}
  Since $n$ is sufficiently large, it follows from Theorem~\ref{THM:Erdos-hypergraph-KST} that $K_{s_2, \ldots, s_{r}}^{r-1} \subset L_{\mathcal{H}_{i+1}'}(T_{i+1}\cup \{v_{i+1}\})$. This copy of $K_{s_2, \ldots, s_{r}}^{r-1}$ together with $T_{i+1}\cup \{v_{i+1}\}$ forms an ordered copy of  $\mathbb{K}$, denoted by $\mathbb{K}'_{i+1}$,  in $\mathcal{H}_{i+1}'$ (see Figure~\ref{fig:Lemma5.2}).  
  Let $B_{i+1}:= B_i\cup (V(\mathbb{K}'_{i+1}) \cap V_2)$ and notice that $|B_{i+1}| = |B_i|+s-s_1 = (i+1)(s-s_1)$.  
  Inductively, we can find an ordered copy of $\mathbb{K}$, denoted by $\mathbb{K}_i'$, from $\mathcal{H}_{i}$ for every $i\in [m_1]$ such that $\mathbb{K}'_1, \ldots, \mathbb{K}'_{m_1}$ are pairwise vertex-disjoint. 
  Let $L':= L\setminus \bigcup_{i\in [m_i]}T_{i}$ and $V_2'':= V_2\setminus (B_{m_1} \cup B_{S})$. 
  Repeating the argument above to the induced subgraph of $\mathcal{H}$ on $L'\cup V_2''$, we can find $\lfloor m/s_1 \rfloor - m_1-\ell$ pairwise vertex-disjoint ordered copies of $\mathbb{K}$ in $\mathcal{H}[L'\cup V_2'']$, showing that $\mathcal{H}$ contains at least $\left \lfloor  m/s_1\right \rfloor$  pairwise vertex-disjoint copies of ordered $\mathbb{K}$ (see Figure~\ref{fig:Lemma5.2}).
\end{proof}

A minor adjustment to the proof of Lemma~\ref{LEMMA:2nd-perfect-K-semibipartite} results in the following simple lemma.

\begin{lemma}\label{LEMMA:2nd-perfect-K-semibipartite-a}
    Let $r \ge 2$ and $s_r \ge \cdots \ge s_1 \ge 2$ be integers and $\alpha < 1/s_1$ be a nonnegative real number. 
    Let $s:= s_1 + \cdots + s_r$ and $\mathbb{K}:= K_{s_1, \ldots, s_r}^{r}$.  
    The following statement holds for sufficiently large $n$. 
    Suppose that $\mathcal{H}$ is an $m$ by $n$ semibipartite $r$-graph on $V_1$ and $V_2$ that satisfies 
    \begin{align*}
        m \le \frac{(1-\alpha s_1)s_1}{(r-1)(s-s_1)}n \quad\text{and}\quad 
        d_{\mathcal{H}}(v)
        \ge (1-\alpha)\binom{n}{r-1} \quad\text{for all $v\in V_1$}. 
    \end{align*}
     Then $\mathcal{H}$ contains $\left \lfloor  m/s_1\right \rfloor$  pairwise vertex-disjoint copies of ordered $\mathbb{K}$.
\end{lemma}

Using Lemma~\ref{LEMMA:2nd-near-perfect-K-semibipartite}, we can derive the following crude but useful upper bound for $\mathrm{ex}(n,(t+1)F)$. 

\begin{lemma}\label{LEM:2nd-weak-bound}
    Let $r \ge 2$ and $s_r \ge \cdots \ge s_1 \ge 2$ be integers. 
    Let $s:= s_1 + \cdots + s_r$ and $\mathbb{B}:= B_{s_1, \ldots, s_r}^{r}$ be an $r$-partite $r$-graph with part sizes $s_1, \ldots, s_r$. 
    For every $\alpha > 0$ there exists  $N_0$ such that the following statement holds for all $n \ge N_0$ and $t \le \frac{\alpha n}{9s^2 r!}$. 
    Every $n$-vertex $(t+1)\mathbb{B}$-free $r$-graph $\mathcal{H}$ satisfies  
    \begin{align*}
        |\mathcal{H}|
        \le \left(s_1t + \frac{17 sr!}{\alpha} \right) \Delta(\mathcal{H})
            + \frac{\alpha s_1}{2r!} t n^{r-1} 
            + \mathrm{ex}(n-st,\mathbb{B}). 
    \end{align*}
    If, in addition, $\Delta(\mathcal{H}) \le (1-\alpha)\binom{n-1}{r-1}$ and $t \ge \frac{5}{\alpha s_1}\left(\frac{\mathrm{ex}(n,\mathbb{B})}{\binom{n-1}{r-1}} + \frac{18 sr!}{\alpha}\right)$, then
    \begin{align*}
        |\mathcal{H}| 
        \le \binom{n}{r}-\binom{n-s_1(t-1)}{r}. 
    \end{align*}
\end{lemma}
\begin{proof}
    Fix $\alpha>0$ and let $n$ be sufficiently large. 
    Let $\mathcal{H}$ be a maximum $n$-vertex $(t+1)\mathbb{B}$-free $r$-graph. 
    Let $\mathcal{B} = \{\mathbb{B}_1, \ldots, \mathbb{B}_{t}\}$ be a collection of pairwise vertex-disjoint copies of $\mathbb{B}$ in $\mathcal{H}$. 
    Let 
    \begin{align*}
        V:= V(\mathcal{H}),\ 
        V_1:= \bigcup_{i\in [t]}V(\mathbb{B}_i),\ 
        V_2:= V \setminus V_1,\quad\text{and}\quad
        L
        := \left\{v\in V_1 \colon d_{\mathcal{H}}(v) \ge \frac{\alpha s_1 n^{r-1}}{2 s r!} \right\}.
    \end{align*}
    
    \begin{claim}\label{CLAIM:2nd-L-upper-bound-LEMMA}
      We have $|L| \le  \left( t +1 + \frac{16sr!}{\alpha s_1}\right)s_1$.
    \end{claim}
    \begin{proof}
        Suppose that this is not true. 
        Let 
        \begin{align*}
            m := \left\lceil \left( t +1 + \frac{16sr!}{\alpha s_1}\right)s_1 \right\rceil
            \le 2\times \frac{\alpha n}{9s^2 r!} \times s_1 
            \le \frac{\alpha s_1 n}{4s^2 r!}
            \le \frac{\alpha (n-m)}{8s}. 
        \end{align*}
        We may assume that $|L| = m$, since otherwise we can replace $L$ by a subset of size $m$. 
        Let $\mathcal{H}'$ be the collection of edges in $\mathcal{H}$ that have exactly one vertex in $L$. 
        Notice that $\mathcal{H}'$ is semibipartite and for every $v\in L$, we have 
        \begin{align*}
            d_{\mathcal{H}'}(v)
            \ge d_{\mathcal{H}}(v) - |L|\binom{n-2}{r-2}
            \ge \frac{\alpha s_1n^{r-1}}{2s r!}
                - m\frac{r-1}{n-1}\binom{n-1}{r-1}
            \ge \frac{\alpha s_1n^{r-1}}{4s r!}. 
        \end{align*}
        Therefore, by Lemma~\ref{LEMMA:2nd-near-perfect-K-semibipartite},  the semibipartite $r$-graph $\mathcal{H}'$ contains at least 
        \begin{align*}
             \left\lfloor \frac{m}{s_1} - 4\left(\frac{\alpha s_1}{4sr!}\right)^{-1} \right\rfloor
             \ge \left\lfloor t +1 + \frac{16 s r!}{\alpha s_1}  - \frac{16 s r!}{\alpha s_1} \right\rfloor
             = t+1
        \end{align*}
        copies of ordered $\mathbb{B}$, a contradiction. 
    \end{proof}
    It follows from Claim~\ref{CLAIM:2nd-L-upper-bound-LEMMA} and some simple calculations that 
    \begin{align*}
      |\mathcal{H}| 
        & = \sum_{v\in L}d_{\mathcal{H}}(v) + \sum_{v\in V_1\setminus L}d_{\mathcal{H}}(v) + |\mathcal{H}[V_2]| \\
        & \le \left(s_1t + s_1 +  \frac{16 sr!}{\alpha} \right)  \Delta(\mathcal{H})
            + st \times \frac{\alpha s_1n^{r-1}}{2 s r!}
            + \mathrm{ex}(n-st,\mathbb{B}) \\
        & \le \left(s_1t +  \frac{17 sr!}{\alpha} \right)  \Delta(\mathcal{H})
            + \frac{\alpha s_1}{2r!} t n^{r-1} 
            + \mathrm{ex}(n-st,\mathbb{B}). 
    \end{align*}
    Assume additionally that $\Delta(\mathcal{H}) \le (1-\alpha)\binom{n-1}{r-1}$ and $t \ge \frac{5}{\alpha s_1}\left(\frac{\mathrm{ex}(n,\mathbb{B})}{\binom{n-1}{r-1}} + \frac{18 sr!}{\alpha}\right)$. 
    Then it follows from the inequality above that  
    \begin{align*}
        |\mathcal{H}| 
        & \le \left(s_1t +  \frac{17 sr!}{\alpha} \right)  (1- \alpha)\binom{n-1}{r-1}
            + \frac{\alpha s_1}{2r!} t n^{r-1} 
            + \mathrm{ex}(n-st,\mathbb{B}) \\
        & \le s_1 t \left(1- \frac{\alpha}{2}\right)\binom{n-1}{r-1} 
            - s_1 t \frac{\alpha}{2} \binom{n-1}{r-1} 
            + \frac{17 sr!}{\alpha} \binom{n-1}{r-1}  + \frac{\alpha s_1}{2r!} t n^{r-1}  + \mathrm{ex}(n,\mathbb{B}). 
    \end{align*}
    Since $n$ is large and $t \ge \frac{5}{\alpha s_1}\left(\frac{\mathrm{ex}(n,\mathbb{B})}{\binom{n-1}{r-1}} + \frac{18 sr!}{\alpha}\right)$, we have 
    \begin{align*}
        & - s_1 t \frac{\alpha}{2} \binom{n-1}{r-1} 
            + \frac{17 sr!}{\alpha} \binom{n-1}{r-1}  + \frac{\alpha s_1}{2r!} t n^{r-1} + \mathrm{ex}(n,\mathbb{B})\\
        & \le -\left(\frac{\alpha s_1}{2}t- \frac{18 sr!}{\alpha} - \frac{\alpha s_1}{2r} t-o(1) - \frac{\mathrm{ex}(n,\mathbb{B})}{\binom{n-1}{r-1}}\right)\binom{n-1}{r-1}  - s_1 \binom{n-1}{r-1}\\
        & \le -\left(\frac{\alpha s_1}{5}t- \frac{18 sr!}{\alpha} - \frac{\mathrm{ex}(n,\mathbb{B})}{\binom{n-1}{r-1}}\right)\binom{n-1}{r-1} - s_1 \binom{n-1}{r-1}
        \le - s_1 \binom{n-1}{r-1}. 
    \end{align*}
    It follows $t \le \frac{\alpha n}{9s^2 r!}$,  Facts~\ref{FACT:binom-inequality-a}, and~\ref{FACT:inequality-c} that 
    \begin{align*}
        \left(1- \frac{\alpha}{2}\right)\binom{n-1}{r-1}
        & \le \left(1- \frac{\alpha}{2}\right)\left(\frac{n}{n-s_1t-r}\right)^{r-1}\binom{n-s_1t}{r-1} \\
        & \le \left(1- \frac{\alpha}{2}\right)\left(1+ \frac{4(r-1)(s_1t+r)}{n} \right)\binom{n-s_1t}{r-1} \\
        & \le \left(1- \frac{\alpha}{2}\right)\left(1+ \frac{\alpha}{2} \right)\binom{n-s_1t}{r-1}
        \le \binom{n-s_1t}{r-1}. 
    \end{align*}
    Therefore, we obtain 
    \begin{align*}
        |\mathcal{H}| 
        \le s_1(t-1)\binom{n-s_1t}{r-1} +  s_1 \binom{n-1}{r-1} - s_1 \binom{n-1}{r-1}
        \le \binom{n}{r} - \binom{n-s_1(t-1)}{r}, 
    \end{align*}
    completing the proof of Lemma~\ref{LEM:2nd-weak-bound}. 
\end{proof}

\subsection{Hypergraphs: proof of Theorem~\ref{THM:2nd-interval}}\label{SUBSEC:proof-2nd-hypergraph}
In this subsection, we establish the validity of Theorem~\ref{THM:2nd-interval}. 
Our proof strategy begins by considering the collection $L \subset V(\mathcal{H})$ of all vertices in an $n$-vertex $(t+1)F$-free $r$-graph $\mathcal{H}$ with degree at least $\Omega(n^{r-1})$. 
We will show that a significant proportion of vertices in $L$ possess near-maximum degree. 
Consequently, leveraging Lemma~\ref{LEMMA:2nd-perfect-K-semibipartite}, we can infer that the size of $L$ is less than $s_1(t+1)$.
Subsequently, we show that the induced subgraph $\mathcal{H}-L$ cannot contain many vertex-disjoint copies of $F$. 
By employing the size constraint provided by Lemma~\ref{LEMMA:trivial-max-degree} on $\mathcal{H}-L$, we derive the desired upper bound on $|\mathcal{H}|$.

\begin{proof}[Proof of Theorem~\ref{THM:2nd-interval}]
    Fix integers $r\ge 2$, $s_r \ge \cdots \ge s_1\ge 2$, and an $r$-partite $r$-graph $B_{s_1, \ldots, s_r}^{r} =: \mathbb{B}$  with part sizes $s_1, \ldots, s_r$. 
    Let 
    \begin{align*}
        s:= \sum_{i\in [r]}s_i, \quad
        \delta_1 := \frac{1}{16es_1s}, \quad  
        \delta_2 := \frac{1}{2es}, \quad\text{and}\quad
        \varepsilon := \frac{\delta_1}{32s^2 r!}. 
    \end{align*}
    Let $n$ be sufficiently large and $t$ be an integer satisfying 
    \begin{align}\label{equ:2nd-hygp-t-range}
        \frac{20}{\delta_1}\left(\frac{\mathrm{ex}(n,\mathbb{B})}{\binom{n-1}{r-1}} + \frac{20 sr!}{\delta_1}\right)
        \le t \le \varepsilon n.  
    \end{align}
    Let $\mathcal{H}$ be a maximum $n$-vertex $(t+1)\mathbb{B}$-free $r$-graph with vertex $V$. 
    Let 
    \begin{align*}
        L_{1} := \left\{v\in V \colon d_{\mathcal{H}}(v) \ge (1-\delta_1)\binom{n-1}{r-1}\right\},\ 
        L_2  := \left\{v\in V\setminus L_1 \colon d_{\mathcal{H}}(v) \ge \delta_2 \binom{n-1}{r-1}\right\}, 
    \end{align*}
    and 
    \begin{align*}
        L & := L_1\cup L_2, \quad  
        U := V \setminus L, \quad 
        \ell:= |L|, \quad 
        \ell_1 := |L_1|, \quad 
        \ell_2:= |L_2|. 
    \end{align*}
    In addition, let 
    \begin{align*}
        \mathcal{H}_1:= \mathcal{H} - L_1,\quad 
        \mathcal{H}_2:= \mathcal{H}[U], \quad 
        n_1:= n-\ell_1, 
        \quad\text{and}\quad
         t_1 := t - \left\lfloor \ell_1/s_1 \right\rfloor. 
    \end{align*}
    \begin{claim}\label{CLAIM:2nd-H1-upper-bound}
        We have $\ell_1 < s_1(t+1)$ and $\mathcal{H}_1$ is $(t_1 + 1)\mathbb{B}$-free. 
    \end{claim}
    \begin{proof}
        Suppose that the first part is not true. 
        We may assume that $\ell_1 = s_1(t+1)$ since otherwise we can replace $L_1$ by an $s_1(t+1)$-subset. 
        Let $\mathcal{G}$ be the collection of edges in $\mathcal{H}$ that have exactly one vertex in $L_1$. 
        Notice that $\mathcal{G}$ is semibipartite, and for every $v\in L_1$ it follows from the assumptions on $t$ and $\delta_1$ that 
        \begin{align*}
            d_{\mathcal{G}}(v)
            \ge d_{\mathcal{H}}(v) - |L_1|\binom{n-2}{r-2}
            & \ge (1-\delta_1)\binom{n-1}{r-1} - s_1(t+1) \frac{r-1}{n-1} \binom{n-1}{r-1} \\
            & > \left(1- \frac{1}{16 s_1 s}\right)\binom{n-1}{r-1}. 
        \end{align*}
        Since, in addition, $\left(1-\frac{s_1}{16s_1s}\right)s_1\frac{1}{(r-1)(s-s_1)}(n-\ell_1) \ge 2\varepsilon s_1 n \ge \ell_1$ (by~\eqref{equ:2nd-hygp-t-range}), 
        it follows from Lemma~\ref{LEMMA:2nd-perfect-K-semibipartite-a} that $(t+1)\mathbb{B} \subset (t+1)K_{s_1, \ldots, s_r}^r\subset \mathcal{G}$, a contradiction.

        Now suppose to the contrary that there exists a collection  $\mathcal{B} = \{\mathbb{B}_1, \ldots, \mathbb{B}_{t_1+1}\}$ of $t_1+1$ pairwise vertex-disjoint copies of $\mathbb{B}$ in $\mathcal{H}_1$. 
        Let $B := \bigcup_{i\in [t_1+1]}V(\mathbb{B}_i)$ and $V' := V\setminus B$.
        Let $\mathcal{G}_1$ be the collection of edges in $\mathcal{H}[V']$ that contain exactly one vertex in $L_1$. 
        Similar to the argument above, we have 
        \begin{align*}
            d_{\mathcal{G}_1}(v)
            \ge d_{\mathcal{H}}(v) - |L_1\cup B_1|\binom{n-2}{r-2}
            & \ge (1-\delta_1)\binom{n-1}{r-1} - 2s(t+1) \frac{r-1}{n-1} \binom{n-1}{r-1} \\
            & > \left(1- \frac{1}{16 s_1 s}\right)\binom{n-1}{r-1}. 
        \end{align*}
        Similarly, it follows from~\eqref{equ:2nd-hygp-t-range} and  Lemma~\ref{LEMMA:2nd-perfect-K-semibipartite-a} that  $\lfloor \ell_1/s_1 \rfloor \mathbb{B} \subset \mathcal{G}_1$, which together with $\mathcal{B}$ implies that $(t+1)\mathbb{B}\subset \mathcal{H}$, a contradiction.  
    \end{proof}
    By Claim~\ref{CLAIM:2nd-H1-upper-bound}, we obtain 
    \begin{align*}
        n_1 - r 
        = n-\ell_1 -r  
        \ge n- s_1 (t+1) - r 
        \ge (1-2\varepsilon s_1)n. 
    \end{align*}
    Combined with the definition of $L_1$, Facts~\ref{FACT:binom-inequality-a}, and~\ref{FACT:inequality-c}, we obtain     \begin{align*}
        \Delta(\mathcal{H}_1)
        \le (1-\delta_1)\binom{n-1}{r-1}
        & \le (1-\delta_1) \left(\frac{n}{n_1-r}\right)^{r-1}\binom{n_1-1}{r-1} \\
        & \le (1-\delta_1) \left(1+8\varepsilon r s_1\right)\binom{n_1-1}{r-1} \\
        & \le (1-\delta_1) \left(1+\frac{\delta_1}{4}\right)\binom{n_1-1}{r-1} 
         \le \left(1-\frac{\delta_1}{2}\right)\binom{n_1-1}{r-1}. 
    \end{align*}
    Let
    \begin{align*}
        \theta
        := \frac{5}{\delta_1 s_1/2}\left(\frac{\mathrm{ex}(n_1,\mathbb{B})}{\binom{n_1-1}{r-1}} + \frac{18sr!}{\delta_1/2}\right)
        \le \frac{10}{\delta_1 s_1}\left(\frac{2 \mathrm{ex}(n,\mathbb{B})}{\binom{n-1}{r-1}} + \frac{36sr!}{\delta_1}\right)
        \le \frac{t-s+1}{s_1}. 
    \end{align*}
    If $t_1 \ge  \theta$, then 
    it follows from $t_1 \le t \le \frac{\delta_1 n}{32s^2 r!} \le \frac{\delta_1 n_1/2}{9s^2 r!}$ and Lemma~\ref{LEM:2nd-weak-bound} that 
    \begin{align*}
        |\mathcal{H}_1|
        \le \binom{n_1}{r} - \binom{n_1-s_1(t_1-1)}{r}
        = \binom{n-\ell_1}{r} - \binom{n-s_1 t}{r}. 
    \end{align*}
    Consequently, 
    \begin{align*}
        |\mathcal{H}|
        \le \binom{n}{r} -\binom{n-\ell_1}{r} + |\mathcal{H}_1|
        \le \binom{n}{r} - \binom{n-s_1 t}{r},  
    \end{align*}
    and we are done. 
    So we may assume that $t_1 \le \theta$, which implies that 
    \begin{align}\label{equ:2nd-hygp-L1-lower-bound}
        \ell_1 
        \ge s_1 \left(t - \theta\right)
        \ge s_1 \left(t - \frac{t-s+1}{s_1}\right)
        = \frac{s_1-1}{s_1} s_1(t+1). 
    \end{align}
    %
    Let 
    \begin{align*}
        t_2
        := t- \left\lfloor \ell/s_1 \right\rfloor
        \le t_1
        \le \theta. 
    \end{align*}
    \begin{claim}\label{CLAIM:2nd-L1+L2-upper-bound}
        We have $\ell_1 + \ell_2 \le s_1(t+1)-1$, and $\mathcal{H}[U]$ is $(t_2+1) \mathbb{B}$-free.
    \end{claim}
    \begin{proof}
        The proof is similar to that of Claim~\ref{CLAIM:2nd-H1-upper-bound}. 
        Suppose to the contrary that Claim~\ref{CLAIM:2nd-L1+L2-upper-bound} is not true. 
        If $\ell < s_1(t+1)$, then let $\mathcal{B}' := \left\{\mathbb{B}_1', \ldots, \mathbb{B}_{t_2+1}' \right\}$ be a collection of pairwise vertex-disjoint copies of $\mathbb{B}$ in $\mathcal{H}[U]$ and let $B':= \bigcup_{i\in [t_2+1]}V(\mathbb{B}_i')$. 
        If $\ell \ge s_1(t+1)$, then by removing vertices in $L_2$, we may assume that $\ell= s_1(t+1)$. In addition, we set $B':= \emptyset$ in this case. 
         Let $\mathcal{G}_2$ be the collection of edges in $\mathcal{H}[V\setminus B']$ that contain exactly one vertex in $L$.
        To build a contradiction, it suffices to show that $\lfloor \ell/s_1 \rfloor \mathbb{B} \subset \mathcal{G}_2$.  
        Indeed, observe that for every $v\in L_1$, we have 
        \begin{align*}
            d_{\mathcal{G}_2}(v)
            & \ge d_{\mathcal{H}}(v) - |L\cup B'| \binom{n-2}{r-2} \\
            & \ge (1-\delta_1)\binom{n-1}{r-1} - 2s(t+1)\frac{r-1}{n-1}\binom{n-1}{r-1} 
             \ge \binom{n-1}{r-1} - \frac{\delta_2 n^{r-1}}{4s_1}, 
        \end{align*}
        and for every $u\in L_2$, we have 
        \begin{align*}
            d_{\mathcal{G}_2}(u)
            & \ge d_{\mathcal{H}}(u) - |L\cup B'| \binom{n-2}{r-2} \\
            & \ge \delta_2\binom{n-1}{r-1} - 2s(t+1)\frac{r-1}{n-1}\binom{n-1}{r-1} 
             \ge \frac{\delta_2 n^{r-1}}{2}.  
        \end{align*}
        Since, in addition, $\ell \le 2s_1t \le 2\varepsilon s_1 n \le \frac{1}{2} \frac{\delta_2 n/2}{8(s-s_1)} \le \frac{\delta_2 (n-|B'|)/2}{8(s-s_1)}$, 
         it follows from~\eqref{equ:2nd-hygp-L1-lower-bound} and Lemma~\ref{LEMMA:2nd-perfect-K-semibipartite} that $\lfloor \ell/s_1 \rfloor \mathbb{B} \subset \mathcal{G}_2$, a contradiction. 
    \end{proof}
    %
    %
    Since $\mathcal{H}[U]$ is $(t_2+1)\mathbb{B}$-free and  $\Delta(\mathcal{H}[U]) \le \delta_2 \binom{n-1}{r-1} \le \frac{1}{2s}\binom{n-\ell-1}{r-1}$ (by Fact~\ref{FACT:binom-inequality-b}),  
    it follows from Lemma~\ref{LEMMA:trivial-max-degree} and simple calculations that  
    \begin{align}\label{equ:THM-2nd-hypg-HU-upper-bound}
        |\mathcal{H}[U]|
        & \le st_2 \times \frac{1}{2s}\binom{n-\ell-1}{r-1} + \mathrm{ex}(n-\ell-st_2, \mathbb{B}) \\
        & \le \binom{n-\ell}{r} - \binom{n-s_1(t+1)+1}{r-1} + \mathrm{ex}(n-s_1(t+1)+1, \mathbb{B}). \notag
    \end{align}
    It follows that 
    \begin{align*}
        |\mathcal{H}|
        & \le \binom{n}{r}-\binom{n-\ell}{r} + |\mathcal{H}[U]| \\
        & \le \binom{n}{r}-\binom{n-\ell}{r} + \binom{n-\ell}{r} - \binom{n-s_1(t+1)+1}{r-1} + \mathrm{ex}(n-s_1(t+1)+1, \mathbb{B}) \\
        & = \binom{n}{r} - \binom{n-s_1(t+1)+1}{r} + \mathrm{ex}(n-s_1(t+1)+1, \mathbb{B}), 
    \end{align*}
    completing the proof of Theorem~\ref{THM:2nd-interval}. 
\end{proof}

\subsection{Graphs: proof of Theorem~\ref{THM:2nd-interval-graph}}\label{SUBSEC:proof-2nd-graph}
In this subsection, we establish the validity of Theorem~\ref{THM:2nd-interval-graph}. 
This proof is a continuation of the argument in the previous subsection. 
To achieve the exact bound we desired, a more detailed analysis of the underlying structures is necessary.

Given a family $\mathcal{F}$ of $r$-graphs and an integer $t \ge 0$ let 
\begin{align*}
    (t+1) \mathcal{F}
    := \left\{F_1 \sqcup \cdots \sqcup F_{t+1} \colon F_i \in \mathcal{F} \text{ for all } i\in [t+1]\right\}. 
\end{align*}
We will use the following result which follows essentially from the same proof for Theorem~\ref{THM:1st-interval}.

\begin{proposition}\label{PROP:2nd-graph-one-star-is-better}
    Let $s_2 \ge s_1 \ge 2$ be integers and $\mathbb{B}:= B_{s_1, s_2}$ be an $s_1$ by $s_2$ bipartite graph. 
    Let $\mathbb{B}':= \mathbb{B}[s_2+1]$. 
    Suppose that $t \le \frac{n}{2e(s_2+1)}$. 
    Then 
    \begin{align*}
        \mathrm{ex}\left(n,(t+1)\mathbb{B}'\right)
        \le \mathrm{ex}\left(n,\mathbb{B}'\sqcup tK_{1,s_2}\right)
        \le \binom{n}{2}-\binom{n-t}{2}+ \mathrm{ex}\left(n-t,(t+1)\mathbb{B}'\right). 
    \end{align*}
\end{proposition}

\begin{proof}[Proof of Theorem~\ref{THM:2nd-interval-graph}]
    Fix integers $s_2 \ge s_1 \ge 2$ and fix an $s_1$ by $s_2$ bipartite connected graph $B_{s_1, s_2} =: \mathbb{B}$. 
     Let $s:= s_1 + s_2$, $\hat{t} := s_1(t+1)-1$, 
    \begin{align*}
        \delta_1 := \frac{1}{16es_1s}, \quad  
        \delta_2 := \frac{1}{2es}, \quad 
        \delta_3:= \frac{1}{2(s_2+1)}, \quad\text{and}\quad
        \varepsilon := \frac{\delta_1}{64s^2}. 
    \end{align*}
    Let $n$ be sufficiently large and $t$ be an integer satisfying 
    \begin{align}\label{equ:2nd-gp-t-range}
        \max \left\{\sqrt{32s_1sn},\  \frac{20 s^2 \theta}{\delta_3 s_1} \right\} 
        \le t 
        \le \varepsilon n, 
    \end{align}
    where 
    \begin{align*}
        \theta:= \frac{20}{\delta_1 s_1}\left(\frac{\mathrm{ex}(n,\mathbb{B})}{n-1} + \frac{18sr!}{\delta_1}\right). 
    \end{align*}

    Let $G$ be a maximum $n$-vertex $(t+1)\mathbb{B}$-free graph on $V$. 
    Let $L_1, L_2, L, U, \ell_1, \ell_2, \ell, n_1, t_1, t_2$ be the same as defined in the proof of Theorem~\ref{THM:2nd-interval} (with $r=2$). 
    All conclusions in the previous subsection still hold for $G$. 
    To summarize, we have  
    \begin{itemize}
        \item $s_1(t-\theta) \le \ell_1 \le \ell \le s_1(t+1)-1$, 
        \item $t_2 := t - \lfloor \ell/s_1 \rfloor \le t_1 : = t - \lfloor \ell_1/s_1 \rfloor \le \theta$, 
        \item $\Delta := \Delta(G[U]) \le \delta_2 (n-1)$, 
        \item $G[U]$ is $(t_2+1)\mathbb{B}$-free, 
        and hence, 
                \begin{align}\label{equ:THM-2nd-pg-GU-upper-bound}
                    |G[U]|
                    \le st_2 \times \delta_2 (n-1) + \mathrm{ex}(n-\ell-st_2, \mathbb{B})
                    \le \frac{\theta n}{2} +  \mathrm{ex}(n, \mathbb{B})
                    \le \theta n. 
                \end{align}  
    \end{itemize}
    We divide $U$ further by letting 
    \begin{align*} 
        S
        := \left\{u\in U \colon |N_{G}(u) \cap L| \le (1-\delta_3)|L|\right\} 
        \quad\text{and}\quad 
        W := U \setminus S. 
    \end{align*}
    In particular, since $G[S] \subset G[U]$ is $(t_2+1) \mathbb{B}$-free, it follows from Lemma~\ref{LEMMA:trivial-max-degree} that 
    \begin{align}\label{equ:2nd-graph-G[S]-upper-bound}
        |G[S]|
        \le s t_2 |S| + \mathrm{ex}\left(|S|, \mathbb{B}\right)
        \le s \theta |S| + \mathrm{ex}\left(|S|, \mathbb{B}\right). 
    \end{align}
    By the definition  of $S$, we have 
    \begin{align}\label{equ:THM-2nd-pg-G[L,U]-upper-bound}
        |G[L,U]|
        \le |W||L| + |S|\times (1-\delta_3)|L|
        = \ell (n-\ell) - \delta_3 \ell |S|. 
    \end{align}
        It follows from~\eqref{equ:THM-2nd-pg-GU-upper-bound} and~\eqref{equ:THM-2nd-pg-G[L,U]-upper-bound} that 
        \begin{align*}
            |G|
              = |G[L]| + |G[L,U]| + |G[U]| 
            & \le \binom{\ell}{2} + \ell (n-\ell) - \delta_3 \ell |S|
            + \theta n. 
        \end{align*}
        Combined with  $|G|\ge g_{2}(n,t,\mathbb{B}) \ge \binom{n}{2}-\binom{n-s_1(t+1)+1}{2} \ge \binom{\ell}{2} + \ell (n-\ell)$, 
         we obtain 
        \begin{align*}
            |S|
            \le \frac{\theta n}{s_1(t-\theta)}
            < \frac{n}{2}. 
        \end{align*}
    which combined with Proposition~\ref{PROP:Turan-ratio} implies that  
    \begin{align*}
        \frac{\mathrm{ex}\left(n,\mathbb{B}\right)}{n}
        \ge \left(1 - \frac{|S|}{n}\right)\frac{\mathrm{ex}\left(|S|,\mathbb{B}\right)}{|S|}
        \ge \frac{1}{2}\frac{\mathrm{ex}\left(|S|,\mathbb{B}\right)}{|S|}.
    \end{align*}
    Consequently, by~\eqref{equ:2nd-gp-t-range}, we obtain 
    \begin{align}\label{equ:2nd-graph-ratio}
        \frac{\delta_3 \ell}{2}|S| - \mathrm{ex}\left(|S|, \mathbb{B}\right)
          \ge \left(\frac{\delta_3}{2} \times s_1(t-\theta) - 2\frac{\mathrm{ex}\left(n, \mathbb{B}\right)}{n} \right)|S|
         \ge 0. 
    \end{align}
    In addition, 
    \begin{align}\label{equ:2nd-graph-tech-b}
        \frac{\delta_3 \ell}{2} - s \theta - 4s^2 \theta
        \ge \frac{\delta_3}{2}\times s_1(t-\theta) - s \theta - 4s^2 \theta
        > 0. 
    \end{align}
    Let 
    \begin{align*}
        t_3:= s_1(t+1)-1-\ell, \quad 
        S_1 := \left\{v\in S \colon |N_{G}(v) \cap W| \ge 4s^2 \theta \right\}, \quad\text{and}\quad
        x := |S_1|. 
    \end{align*}
    It follows from~\eqref{equ:2nd-graph-G[S]-upper-bound},~\eqref{equ:2nd-graph-ratio},~\eqref{equ:2nd-graph-tech-b} and definition of $S_1$ that 
    \begin{align*}
        |G[S]| + |G[L,S]|+ |G[S,W]| 
         & \le  s \theta |S| + \mathrm{ex}\left(|S|, \mathbb{B}\right)
            + (1-\delta_3)\ell |S| + x \Delta + 4s^2 \theta |S| \notag \\
         & \le  \ell |S| - \left(\delta_3 \ell - s \theta - 4s^2 \theta\right) |S| + \mathrm{ex}\left(|S|, \mathbb{B}\right) + \frac{xn}{2} \notag \\
         & \le \ell |S| + \frac{xn}{2} - \frac{\delta_3 \ell}{2}|S| + \mathrm{ex}\left(|S|, \mathbb{B}\right) 
          \le \ell |S| + \frac{xn}{2}. 
    \end{align*}
    Consequently, 
    \begin{align}\label{equ:2nd-graph-G[L,U]-G[S]-G[S,W]}
        |G[L,U]| + |G[S]| + |G[S,W]| 
        & = |G[L,W]|+ |G[S]|+|G[L,S]|+|G[L,W]| \notag \\
        & \le \ell |W| + \ell |S| + \frac{xn}{2}
        = \ell(n-\ell) + \frac{xn}{2}. 
    \end{align}
    
    For every set $T \subset U$ let $N_T:= \bigcap_{u\in T}N_{G}(u) \cap L$ denote the set of common neighbors of $T$ in $L$. 
    It follows from the definition of $W$ that for every $(s_2+1)$-set $T\subset W$ we have 
    \begin{align}\label{equ:THM-2nd-pg-NT-lower-bound}
        |N_T| 
        \ge \ell - \sum_{u\in T}\left(\ell - |N_{G}(u) \cap L|\right)
        \ge \left(1-\delta_3 (s_2+1)\right)\ell
        = \frac{\ell}{2}. 
    \end{align}
    %
    
\begin{figure}[htbp]
\centering
\tikzset{every picture/.style={line width=0.85pt}} 

\begin{tikzpicture}[x=0.75pt,y=0.75pt,yscale=-1,xscale=1]

\draw [line width=1pt]  (531.27,45.47) .. controls (535.38,45.47) and (538.69,48.8) .. (538.65,52.92) -- (538.45,75.25) .. controls (538.42,79.37) and (535.06,82.7) .. (530.94,82.69) -- (179.86,82.45) .. controls (175.75,82.45) and (172.45,79.11) .. (172.48,75) -- (172.68,52.66) .. controls (172.71,48.55) and (176.08,45.22) .. (180.19,45.22) -- cycle ;
\draw  [line width=1pt] (519.35,97.2) .. controls (530.94,97.21) and (540.25,106.61) .. (540.15,118.2) -- (539.6,181.14) .. controls (539.5,192.73) and (530.02,202.12) .. (518.43,202.11) -- (196.02,201.88) .. controls (184.43,201.87) and (175.12,192.47) .. (175.22,180.89) -- (175.77,117.94) .. controls (175.87,106.35) and (185.35,96.97) .. (196.94,96.98) -- cycle ;
\draw    (539.6,180) -- (175.22,180) ;
\draw    (430,181) -- (430,203) ;

\draw    (525,145) -- (510,192) ;
\draw    (510,145) -- (510,192) ;
\draw    (495,145) -- (510,192) ;
\draw    (510,61) -- (525,145) ;
\draw    (510,61) -- (510,145) ;
\draw    (510,61) -- (495,145) ;
\draw [fill=uuuuuu]    (525,145) circle (1.5pt);
\draw [fill=uuuuuu]   (510,145)  circle (1.5pt);
\draw [fill=uuuuuu]   (495,145) circle (1.5pt);
\draw [fill=uuuuuu]   (510,192)  circle (1.5pt);
\draw [fill=uuuuuu]   (510,61)  circle (1.5pt);
%
\draw    (470,145) -- (455,192) ;
\draw    (455,145) -- (455,192) ;
\draw    (440,145) -- (455,192) ;
\draw    (455,61) -- (470,145) ;
\draw    (455,61) -- (455,145) ;
\draw    (455,61) -- (440,145) ;
\draw [fill=uuuuuu]    (470,145) circle (1.5pt);
\draw [fill=uuuuuu]   (455,145)  circle (1.5pt);
\draw [fill=uuuuuu]   (440,145) circle (1.5pt);
\draw [fill=uuuuuu]   (455,61)  circle (1.5pt);
\draw [fill=uuuuuu]   (455,192)  circle (1.5pt);
%

\draw    (405,126) -- (390,162) ;
\draw    (390,126) -- (390,162) ;
\draw    (375,126) -- (390,162) ;
\draw    (390,61) -- (405,126) ;
\draw    (390,61) -- (390,126) ;
\draw    (390,61) -- (375,126) ;
\draw [fill=uuuuuu]    (405,126) circle (1.5pt);
\draw [fill=uuuuuu]   (390,126)  circle (1.5pt);
\draw [fill=uuuuuu]   (375,126) circle (1.5pt);
\draw [fill=uuuuuu]   (390,162)  circle (1.5pt);
\draw [fill=uuuuuu]   (390,61)  circle (1.5pt);
%
\draw    (355,126) -- (340,162) ;
\draw    (340,126) -- (340,162) ;
\draw    (325,126) -- (340,162) ;
\draw    (340,61) -- (355,126) ;
\draw    (340,61) -- (340,126) ;
\draw    (340,61) -- (325,126) ;
\draw [fill=uuuuuu]    (355,126) circle (1.5pt);
\draw [fill=uuuuuu]   (340,126)  circle (1.5pt);
\draw [fill=uuuuuu]   (325,126) circle (1.5pt);
\draw [fill=uuuuuu]   (340,61)  circle (1.5pt);
\draw [fill=uuuuuu]   (340,162)  circle (1.5pt);
%
\draw    (230,65) -- (240,130) ;
\draw    (230,65) -- (217,130) ;
\draw    (230,65) -- (194,130) ;
\draw    (204,65) -- (240,130) ;
\draw    (204,65) -- (217,130) ;
\draw    (204,65) -- (194,130) ;
\draw [fill=uuuuuu]    (230,65) circle (1.5pt);
\draw [fill=uuuuuu]   (204,65)  circle (1.5pt);
\draw [fill=uuuuuu]   (240,130) circle (1.5pt);
\draw [fill=uuuuuu]   (217,130)  circle (1.5pt);
\draw [fill=uuuuuu]   (194,130)  circle (1.5pt);
%
\draw    (290,65) -- (300,130) ;
\draw    (290,65) -- (277,130) ;
\draw    (290.79,65) -- (254,130) ;
\draw    (264,65) -- (300,130) ;
\draw    (264,65) -- (277,130) ;
\draw    (264,65) -- (254,130) ;
\draw [fill=uuuuuu]    (290,65) circle (1.5pt);
\draw [fill=uuuuuu]   (264,65)  circle (1.5pt);
\draw [fill=uuuuuu]   (300,130) circle (1.5pt);
\draw [fill=uuuuuu]   (277,130) circle (1.5pt);
\draw [fill=uuuuuu]   (254,130)  circle (1.5pt);
%

\draw (150,60) node [anchor=north west][inner sep=0.75pt]   [align=left] {$L$};
\draw (150,130) node [anchor=north west][inner sep=0.75pt]   [align=left] {$W$};
\draw (150,185) node [anchor=north west][inner sep=0.75pt]   [align=left] {$S$};
\draw (480,209) node [anchor=north west][inner sep=0.75pt]   [align=left] {$S_1$};
\end{tikzpicture}

\caption{Supplementary diagram for Claim~\ref{CLAIM:2nd-graph-S1-G[W]} when $F = K_{2,3}$.}
\label{fig:2nd-graph}
\end{figure}
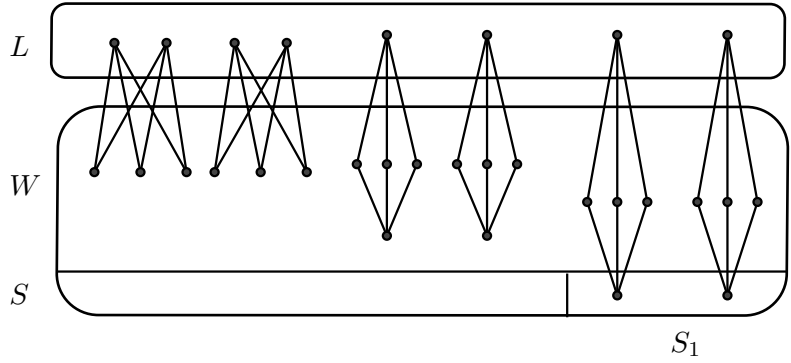

    Let $\mathbb{B}':= \mathbb{B}[s_2+1]$ and a crucial observation is that every $\mathbb{B}'$-free graph is also $K_{1,s_2}$-free. 
    \begin{claim}\label{CLAIM:2nd-graph-S1-G[W]}
        We have $x \le t_3$ and $G[W]$ is $(t_3-x+1) K_{1,s_2}$-free. 
        Moreover,  
        \begin{align}\label{equ:2nd-graph-G[W]-upper-bound-1}
            |G[W]|
            \le \frac{t_3-x}{2}n + \frac{s_2}{2}n. 
        \end{align}
    \end{claim}
    \begin{proof}
        Suppose to the contrary that there exists a $(t_3+1)$-set $\{v_1, \ldots, v_{t_3+1}\} \subset S_1$ such that $|N_{G}(v_i) \cap W| \ge 4s^2 \theta \ge 2s (t_3+1)$ for all $i\in [t_3+1]$. 
        Choose an $s_2$-set $X_i$ from $N_{G}(v_i) \cap W$ for each $i\in [t_2+1]$ such that $\{X_1, \ldots, X_{t_3+1}\}$ are pairwise disjoint. 
        For each $X_i$ choose an $(s_1-1)$-set $Y_i \subset N_{X_i}$ such that $\{Y_1, \ldots, Y_{t_3+1}\}$ are pairwise disjoint (the existence of such sets are guaranteed by~\eqref{equ:THM-2nd-pg-NT-lower-bound}). 
        Notice that $B \subset K_{s_1, s_2} \subset G[\{v_i\} \cup X_i \cup Y_i]$ for $i\in [t_3+1]$. 
        Let $Z:= \bigcup_{i\in [t_3+1]}\left(\{v_i\} \cup X_i \cup Y_i\right)$ and $G'$ be the induced subgraph of $G$ on $V\setminus Z$.
        For every $v\in L_1 \setminus Z$ we have 
        \begin{align*}
            d_{G'}(v) 
            \ge d_{G}(v) - |Z|
            \ge (1-\delta_1)(n-1) - s(t_3+1)
            \ge (1-2\delta_1)(n-1), 
        \end{align*}
        and for every $u\in L_2 \setminus Z$ we have 
        \begin{align*}
            d_{G'}(v) 
            \ge d_{G}(v) - |Z|
            \ge \delta_2 (n-1) - s(t_3+1)
            \ge \frac{\delta_2}{2}(n-1).  
        \end{align*}
        Since, in addition, $\ell \le s_1(t+1) \le 2\varepsilon s_1 n \le \frac{\delta_2 (n-\ell-|Z|)/2}{8(s-s_1)}$, 
        it follows from Lemma~\ref{LEMMA:2nd-perfect-K-semibipartite} that $G'$ contains at least 
        \begin{align*}
            \left\lfloor \frac{\ell- (s_1-1)t_3}{s_1} \right\rfloor
            = \ell - (s_1-1)(t+1)
        \end{align*}
        pairwise vertex-disjoint copies of $\mathbb{B}$. 
        This implies that $G$ contains 
        \begin{align*}
            t_3+1 + \ell - (s_1-1)(t+1) 
            = s_1(t+1)-1-\ell + 1 + \ell - (s_1-1)(t+1)
            = t+1
        \end{align*}
        pairwise vertex-disjoint copies of $\mathbb{B}$, a contradiction. 
        This proves that $x \le t_3$. 
        The proof for the $(t_3-x+1) K_{1,s_2}$-freeness of $G[W]$ is similar (see Figure~\ref{fig:2nd-graph}), so we omit the details. 

        Recall that $\Delta(G[W]) \le \Delta(G[U]) \le \delta_3 n$, so it follows from Lemma~\ref{LEMMA:trivial-max-degree} and $(t_3-x+1) K_{1,s_2}$-freeness of $G[W]$ that 
        \begin{align*}
            |G[W]|
            \le (t_3-x)(s_2+1)\times \delta_3 n + \mathrm{ex}(n,K_{1,s_2})
            \le \frac{t_3-x}{2}n+ \frac{s_2}{2}n, 
        \end{align*}
        where in the last inequality, we used the trivial bound $\mathrm{ex}(n,K_{1,s_2}) \le \frac{s_2}{2}n$. 
    \end{proof}
    %

    It follows from~\eqref{equ:2nd-graph-G[S]-upper-bound},~\eqref{equ:2nd-graph-G[L,U]-G[S]-G[S,W]}, and~\eqref{equ:2nd-graph-G[W]-upper-bound-1} that 
    \begin{align}\label{equ:2nd-graph-G-upper-bound-a}
        |G|
        & = |G[L]| + |G[L,U]| + |G[S]| + |G[S,W]| + |G[W]| \notag \\
        & \le \binom{\ell}{2} - |\overline{G[L]}| + \ell(n-\ell) + \frac{xn}{2} + \frac{t_3-x}{2}n + \frac{s_2}{2}n \notag \\
        & \le \binom{n}{2}-\binom{n-s_1(t+1)+1}{2} - t_3 \times \frac{2n}{3}+ \frac{t_3}{2}n + \frac{s_2}{2}n- |\overline{G[L]}| \notag \\
        & \le \binom{n}{2}-\binom{n-s_1(t+1)+1}{2} + \frac{s_2}{2}n- \frac{t_3 n}{6} - |\overline{G[L]}|. 
    \end{align}
    Since $|G| \ge \binom{n}{2} -\binom{n-s_1(t+1)-1}{2}$, it follows from the inequality above that 
    \begin{align}\label{equ:2nd-graph-t3-upper-bound}
        t_3 \le 3s_2, 
    \end{align}
    and 
    \begin{align}\label{equ:2nd-graph-complement-G[L]}
        |\overline{G[L]}| \le \frac{s_2}{2}n, 
    \end{align}

    \begin{claim}\label{CLAIM:2nd-graph-G[N_t]}
        The following statements hold. 
        \begin{enumerate}[label=(\roman*)]
          \item\label{CLAIM:2nd-graph-G[N_t]-1} Every set $L'\subset L$ of size at least $\sqrt{2s_1sn}$ contains a copy of $K_{s_1-1}$,
          \item\label{CLAIM:2nd-graph-G[N_t]-2} for every $T \in \binom{W}{s_2+1}$, the induced graph $G[N_T]$ contains $s t_3$ pairwise vertex-disjoint copies of $K_{s_1-1}$, and 
          \item\label{CLAIM:2nd-graph-G[N_t]-3} $G[W]$ is $(t_3-x+1) \mathbb{B}'$-free. 
        \end{enumerate}
    \end{claim}
    \begin{proof}
        Suppose to the contrary that there exists a set $L'\subset L$ of size $\sqrt{2s_1sn}$ such that $G[L']$ is $K_{s_1-1}$-free. 
        Then it follows from Tur\'{a}n's theorem that $|G[L']| \le \left(\frac{1}{2}-\frac{1}{2(s_1-1)}\right)|L'|^2$, and consequently, 
        \begin{align*}
            |\overline{G[L]}|
            \ge |\overline{G[L']}|
            \ge \binom{|L'|}{2} - \left(\frac{1}{2}-\frac{1}{2(s_1-1)}\right)|L'|^2 
            \ge \frac{|L'|^2}{4(s_1-1)} 
            > \frac{s_2 n}{2}, 
        \end{align*}
        contradicting~\eqref{equ:2nd-graph-complement-G[L]}. 
        This proves~\ref{CLAIM:2nd-graph-G[N_t]-1}. 

        Let $T \in \binom{W}{s_2+1}$. 
        By~\eqref{equ:THM-2nd-pg-NT-lower-bound},~\eqref{equ:2nd-graph-t3-upper-bound}, and the assumption on $t$, we have
        \begin{align*}
            |N_T|
            \ge \frac{\ell}{2}
            \ge \frac{s_1(t-\theta)}{2}
            \ge \frac{t}{2}
            \ge \frac{\sqrt{32s_1 sn}}{2}
            \ge \sqrt{2s_1sn} + (s_1-1)s t_3. 
        \end{align*}
        It follows from~\ref{CLAIM:2nd-graph-G[N_t]-1} and a simple greedy argument that $G[N_T]$ contains $s t_3$ pairwise vertex-disjoint copies of $K_{s_1-1}$. 

        Notice that if $T \in \binom{W}{s_2+1}$ contains a copy of some element in $\mathbb{B}'$, then by (i), there are disjoint $(s_1-1)$-sets $T'_1, \ldots, T'_{s t_3}$ in $N_{T}$ such that $T\cup T'_i$ contains a copy of $\mathbb{B}$ for all $i\in [s t_3]$. 
        Using this fact and  a similar argument as in the proof of Claim~\ref{CLAIM:2nd-graph-S1-G[W]} we obtain~\ref{CLAIM:2nd-graph-G[N_t]-3}. 
    \end{proof}

    Similar to~\eqref{equ:2nd-graph-G-upper-bound-a}, it follows from~\eqref{equ:2nd-graph-G[L,U]-G[S]-G[S,W]}, Claim~\ref{CLAIM:2nd-graph-G[N_t]}~\ref{CLAIM:2nd-graph-G[N_t]-3}, Proposition~\ref{PROP:2nd-graph-one-star-is-better}, and Fact~\ref{FACT:increasing-f(ell)} that 
    \begin{align*}
        |G|
        & \le \binom{\ell}{2}  + \ell(n-\ell) + \frac{xn}{2} + \mathrm{ex}\left(n-\ell-x, (t_3-x+1)\mathbb{B}'\right) \\
        & \le \binom{n}{2}-\binom{n-(\ell+x)}{2} + \mathrm{ex}\left(n-\ell-x, (t_3-x+1)\mathbb{B}'\right)  \\
        & \le \binom{n}{2} - \binom{n-(\ell+x+t_3-x)}{2} + \mathrm{ex}\left(n-\ell-x -(t_3-x), \mathbb{B}'\right) \\
        & = \binom{n}{2}-\binom{n-s_1(t+1)+1}{2} + \mathrm{ex}\left(n-s_1(t+1)+1, \mathbb{B}'\right),  
    \end{align*}
    completing the proof of Theorem~\ref{THM:2nd-interval-graph}. 
\end{proof}
\section{Concluding remarks}\label{SEC:Remarks}
Fix an $r$-graph $F$ on $m$ vertices. 
We say a collection $\left\{\mathcal{H}_1, \ldots, \mathcal{H}_{t+1}\right\}$ of $r$-graphs on the same vertex set $V$ has a \textbf{rainbow $F$-matching} 
if there exists a collection $\left\{S_i \colon i \in [t+1]\right\}$ of pairwise disjoint $m$-subsets of $V$ such that $F \subset \mathcal{H}_{i}[S_i]$ for all $i\in [t+1]$.
There has been considerable interest in extending some classical results to a rainbow version recently. See e.g. \cite{AH,GLMY22,HLS12,KK21,LWY22,LWY23} for some recent progress on the rainbow version of the Erd\H{o}s Matching Conjecture.
Rainbow version of theorems in the present work can be obtained without too much change to the corresponding proof. 
We omit the details and refer the reader to~{\cite[Theorem~1.8]{HLLYZ23}} for more information. 

Given (not necessary different) $r$-graphs $F_1, \ldots, F_{t+1}$, 
let $\mathrm{ex}\left(n, \bigsqcup_{i\in [t+1]}F_i\right)$ denote the maximum number of edges in an $n$-vertex $r$-graph without pairwise vertex-disjoint copies of $F_1, \ldots, F_{t}$, and $F_{t+1}$. 
One could also use proofs in the present work (with minor changes) to obtain some upper bounds for $\mathrm{ex}\left(n, \bigsqcup_{i\in [t+1]}F_i\right)$. 
We refer the reader to~{\cite[Theorem~6.1]{HLLYZ23}} for more information. 

It seems to be a very challenging problem to determine $\mathrm{ex}\left(n,(t+1)F\right)$ for all feasible values of $t$ in general. 
We tend to believe that for $F = K_{s, \ldots, s}^{r}$, $s\ge 2$, the asymptotic behavior of $\mathrm{ex}\left(n,(t+1)F\right)$ is governed by the three families $\mathcal{G}_1(n,t,F)$, $\mathcal{G}_2(n,t,F)$, and $\mathcal{G}_3(n,t,F)$ defined in the first section. 
However, for general $r$-graphs $F$, there might be other extremal constructions. 
Indeed, for $i\in [r-1]$, we say a set $S\subset V(F)$ is \textbf{$i$-independent} in $F$ if every edge in $F$ contains at most $i$ vertices in $S$. 
Let the \textbf{$i$-independent vertex covering number} $\tau_{i}(F)$ (if it exists) of $F$ be 
\begin{align*}
    \tau_{i}(F)
    := \min\left\{|S| \colon S\subset V(\mathcal{H}),\ \text{$S$ is $i$-independent, and } |S\cap e| \neq \emptyset \text{ for all } e\in \mathcal{H}\right\}. 
\end{align*}
Observe that for $1\le i < j \le r-1$, if $\tau_{i}(F)$ and $\tau_{j}(F)$ exist, then $\tau\le \tau_{j}(F) \le \tau_{i}(F)$. 
For $i\in [r]$ let 
\begin{align*}
    \mathcal{B}\left(n,m,r,i\right)
    := \left\{e\in \binom{[n]}{r} \colon 1 \le |e\cap [m]| \le i\right\}. 
\end{align*}
It is easy to see that $\mathcal{B}\left(n,(t+1)\tau_{i}(F)-1,r,i\right)$ is $(t+1)F$-free. Therefore, we have the following lower bound. 
\begin{fact}
    Let $F$ be an $r$-graph and suppose that $\tau_{i}(F)$ exists for some $i\in [r-1]$. 
    Then 
    \begin{align*}
        \mathrm{ex}\left(n,(t+1)F\right)
        \ge \sum_{j \in [i]}\binom{(t+1)\tau_{i}(F)-1}{j}\binom{n-(t+1)\tau_{i}(F)+1}{r-j}. 
    \end{align*}
\end{fact}

We hope the following natural problem could inspire further research on this topic. 

\begin{problem}\label{PROP:t+1-C4}
    Determine $\mathrm{ex}\left(n,(t+1)C_4\right)$ for large $n$ and for all $t \le n/4$. 
\end{problem}

We would like to remind the reader that constructions in $\mathcal{G}_{3}(n,t,C_4)$ are not maximal. 
Indeed, for a member $G \in \mathcal{G}_{3}(n,t,C_4)$, let $V_1 \cup V_2$ be the partition of $V(G)$ such that $G[V_1] \cong K_{2t+1}$ and $G[V_2]$ is maximum $C_4$-free.
Let $\hat{G}$ be a graph obtained from $G$ by adding a matching $M$ crossing $V_1$ and $V_2$ such that the set $V(M) \cap V_2$ is an independent set in $G[V_2]$. 
It is easy to see that $\hat{G}$ is still $(t+1)C_4$-free. 
The orthogonal polarity graphs show that the size of $M$ can have order $(n-2t-1)^{3/4}$ for some special values of $n-2t-1$ (see e.g.~\cite{Brown66,ER62,ERS66,Fur83,Fur94,MW07}). 
This construction shows that determining the exact value of $\mathrm{ex}\left(n,(t+1)C_4\right)$ when $t$ is close to $n/4$ is probably challenging and the following problem seems related. 

\begin{problem}\label{PROP:C4-free-indepence-number}
    Let $n \ge \alpha \ge 1$ be integers. 
    What is the maximum number of edges in an $n$-vertex $C_4$-free graph with independence number at least $\alpha$? 
\end{problem}

It would be interesting to know whether the lower bound $\sqrt{32s_2sn}$ for $t$ in Theorem~\ref{THM:2nd-interval-graph} is necessary in general. 
In an accompanying note~\cite{HHLLYZ23b}, we showed that this lower bound is not necessary for a special class of bipartite graphs including even cycles. 

\begin{problem}\label{PROP:2nd-graph-better-t}
    Characterize the family of bipartite graphs $B_{s_1,s_2}$ for which the lower bound $\sqrt{32s_2(s_1+s_2)n}$ for $t$ in Theorem~\ref{THM:2nd-interval-graph} can be omitted. 
\end{problem}

To refine the bound presented in Theorem~\ref{THM:2nd-interval}, it appears that one might need to address the following problem first.

\begin{problem}\label{PROB:r-1-large-vtx}
    Let $n \ge r \ge 3$, $s_r \ge \cdots \ge s_1 \ge 2$, and $m \le \binom{n-1}{r-1}$ be integers. 
    What is the maximum number of edges in an $n$-vertex $K_{s_1, \ldots, s_r}^{r}$-free $r$-graph $\mathcal{H}$ with an $(s_1-1)$-set $L \subset V(\mathcal{H})$ such that $d_{\mathcal{H}}(v) \ge \binom{n-1}{r-1} - m$ for all $v\in L$?
\end{problem}

Recall that the \textbf{shadow} of an $r$-graph $\mathcal{H}$ is 
\begin{align*}
    \partial\mathcal{H}
    := \left\{A\in \binom{V(\mathcal{H})}{r-1} \colon A\subset E \text{ for some } E\in \mathcal{H}\right\}. 
\end{align*}
The classical Kruskal-Katone Theorem~\cite{KA68,KR63} determines the maximum size of an $r$-graph with bounded shadow size. 
It seems that Problem~\ref{PROB:r-1-large-vtx} is related to the following Kruskal-Katone type problem. 
\begin{problem}\label{PROB:KK-F-free}
    Let $r \ge 3$ and $F$ be a degenerate $r$-graph. 
    Let $m \le \binom{n}{r-1}$ be an integer. 
    What is the maximum number of edges in an $n$-vertex $F$-free $r$-graph $\mathcal{H}$ with $|\partial \mathcal{H}| \le m$? 
\end{problem}
An analogous problem for nondegenerate $r$-graphs was studied systematically in~\cite{LM21a} (see also~\cite{Liu20a,LiuMuk23,LMR1,HLLMZ22,LP22} for furthrt related results). 

Theorem~\ref{THM:1st-interval} motivates the following analogous problem to~{\cite[Problem~6.4]{HLLYZ23}}.

\begin{problem}\label{PROB:1st-threshold}
    Let $r\ge 2$ be an integers and $F$ be a degenerate $r$-graph. 
    What is the largest value of $s(F)$ such that for sufficiently large $n$, 
    \begin{align*}
        \mathrm{ex}(n,(t+1)F)
        = g_{1}(n,t,F)
        \quad\text{holds for all }
        t \le \left(s(F)-o(1)\right)\frac{\mathrm{ex}(n,F)}{\binom{n-1}{r-1}}? 
    \end{align*}
\end{problem}

Theorem~\ref{THM:1st-interval} also motivates the following question for nondegenerate hypergraphs. 

\begin{problem}\label{PROB:Turan-density-upper-bound}
    Let $r \ge 4$ be an integer. 
    Find sufficient and/or necessary conditions for a nondegenerate $r$-graph $F$ to satisfy $\pi(F) < \frac{1}{v(F)}$. 
\end{problem}

By Erd{\H o}s' theorem~\cite{Erdos64}, every nondegenerate $r$-graph $F$ satisfies the lower bound $\pi(F) \ge \frac{r!}{r^r}$. 
This, together with straightforward calculations and constructions of $K_{4}^{3-}$-free $3$-graphs (see e.g.~\cite{MPS11}), shows that there is no nondegenerate $r$-graph $F$ with $\pi(F) < \frac{1}{v(F)}$ when $r\le 3$. 
\bibliographystyle{abbrv}
\bibliography{CHBipartite}
\begin{appendix}
\section{Proof of Theorem~\ref{THM:Erdos-hypergraph-KST}}\label{APPEND:hypergraph-KST}
\begin{theorem}[Erd\H{o}s~\cite{Erdos64}]\label{APPEND:THM:Erdos-hypergraph-KST}
    Suppose that $n \ge r \ge 3$ and $s_r\ge \cdots \ge s_1 \ge 1$ are integers. 
    Then 
    \begin{align*}
        \mathrm{ex}(n,K_{s_1, \ldots, s_r}^{r})
        \le \frac{(s_2 + \ldots+ s_r-r+1)^{\frac{1}{s_1}}}{r}n^{r-\frac{1}{s_1\cdots s_{r-1}}} + \frac{s_1 -1}{r} \binom{n}{r-1}. 
    \end{align*}
\end{theorem}
\begin{proof}
We prove this theorem by induction on $r$. 
The base case $r=2$ is the K{\H o}v\'{a}ri--S\'{o}s--Tur\'{a}n Theorem. 
So we may assume that $r \geq 3$. 
Let $\mathcal{H}$ be a $K_{s_1, \cdots, s_r}^{r}$-free $r$-graph on $n$ vertices.  Let $X$ denote the number of copies of $K_{s_1, 1,\ldots ,1}^{r}$ in $\mathcal{H}$. 
We use a standard double-counting argument to bound the number of edges in $\mathcal{H}$.
First, it follows from Jensen's inequality that 
\begin{align}\label{eq-Hyper-KST-lowerbound}
    X 
    & = \sum_{T \in \binom{n}{r-1}} \binom{|N(T)|}{s_1} \notag \\ 
    & \ge \binom{n}{r-1} \binom{\sum_{T \in \binom{n}{r-1}} |N(T)| / \binom{n}{r-1}}{s_1} \notag \\
    &  = \binom{n}{r-1} \binom{r|\mathcal{H}| / \binom{n}{r-1}}{s_1} 
    \geq \frac{1}{s_1!} \binom{n}{r-1} \left( \frac{r|\mathcal{H}|}{\binom{n}{r-1}} -(s_1 - 1) \right)^{s_1}.
\end{align}
From the other direction, notice that for every $S \in \binom{n}{s_1}$, the common link $L(S)$ is $K_{s_2, \ldots, s_r}^{r-1}$-free. 
It follows from  the induction hypothesis that 
\begin{align}\label{eq-Hyper-KST-link}
    |L_{\mathcal{H}}(S)| 
    &\leq \frac{(s_3 + \cdots+ s_r-r+2)^{\frac{1}{s_2}}}{r-1}n^{r-1-\frac{1}{s_2\cdots s_{r-1}}} + \frac{s_2 -1}{r-1} \binom{n}{r-2} \notag \\
    & \le \frac{s_2 + \cdots+ s_r-r+1}{r-1}  n^{r-1-\frac{1}{s_2\cdots s_{r-1}}}. 
\end{align}
It follows that 
\begin{align}\label{eq-Hyper-KST-upperbound}
    X  
    = \sum_{S \in \binom{n}{s_1}} |L(S)|  
    \le \frac{s_2 + \cdots+ s_r-r+1}{r-1} n^{r-1-\frac{1}{s_2\cdots s_{r-1}}}\binom{n}{s_1}. 
\end{align}
Combining~\eqref{eq-Hyper-KST-lowerbound} and~\eqref{eq-Hyper-KST-upperbound}, we obtain
\begin{align*}
   |\mathcal{H}| 
   & \le \frac{1}{r}\left(\frac{s_2 + \cdots+ s_r-r+1}{r-1} n^{r-1-\frac{1}{s_2\cdots s_{r-1}}}\frac{n^{s_1}}{\binom{n}{r-1}}\right)^{\frac{1}{s_1}}\binom{n}{r-1} + \frac{s_1-1}{r}\binom{n}{r-1} \\
   & \le \frac{\left(s_2 + \cdots+ s_r-r+1\right)^{\frac{1}{s_1}} }{r} n^{r-\frac{1}{s_1\cdots s_{r-1}}} + \frac{s_1 -1}{r} \binom{n}{r-1}, 
\end{align*}
completing the induction. 
\end{proof}
\section{Proof of Proposition~\ref{PROP:hypergraph-KST-Zaran}}\label{APPEND:hypergraph-KST-semibipartite}
\begin{proposition}\label{APPEND:PROP:hypergraph-KST-Zaran}
    Suppose that $r \ge 3$, $s_r\ge \cdots \ge s_1 \ge 1$, and $m, n\ge 1$ are integers. 
    Then 
    \begin{align*}
        Z(m,n,s_1, \ldots, s_r)
        \le \frac{\left(s_2 + \cdots+ s_r-r+1\right)^{\frac{1}{s_1}} }{r-1} mn^{r-1-\frac{1}{s_1\cdots s_{r-1}}} + (s_1-1)\binom{n}{r-1}. 
    \end{align*}    
\end{proposition}
\begin{proof}
Let $\mathcal{H}$ be an $m$ by $n$ semibipartite $r$-graph on $V_1$ and $V_2$ with $|V_1| = m$ and $|V_2| = n$. 
Suppose that $\mathcal{H}$ does not contain a copy of $K_{s_1, \ldots, s_r} = K[W_1, \ldots, W_r]$ with $W_1$ contained in $V_1$ and $W_2, \ldots, W_{r}$ contained in $V_2$. 
We use a standard double-counting argument to bound the size of $\mathcal{H}$. 
Let 
\begin{align*}
    X
    := \sum_{T\in \binom{V_2}{r-1}}\binom{|N(T)|}{s_1}
    = \sum_{S\in \binom{V_1}{s_1}}|L(S)|. 
\end{align*}
First, it follows from Jensen's inequality that 
\begin{align}\label{equ:KST-Zaran-X-lower-bound}
    X
    & \ge \binom{n}{r-1}\binom{\sum_{T\in \binom{V_2}{r-1}}|N(T)|/\binom{n}{r-1}}{s_1} \notag \\
    & = \binom{n}{r-1}\binom{|\mathcal{H}|/\binom{n}{r-1}}{s_1} 
     \ge \frac{1}{s_1!}\binom{n}{r-1}\left(\frac{|\mathcal{H}|}{\binom{n}{r-1}} - (s_1-1)\right)^{s_1}. 
\end{align}
From the other direction, notice that for every $S\in \binom{V_1}{s_1}$, the common link $L_{\mathcal{H}}(S) \subset \binom{V_2}{r-1}$ is $K_{s_2, \ldots, s_r}^{r-1}$-free. 
So it follows from~\eqref{eq-Hyper-KST-link} that 
\begin{align*}
    |L_{\mathcal{H}}(S)|
    \le \frac{s_2 + \cdots+ s_r-r+1}{r-1} n^{r-1-\frac{1}{s_2\cdots s_{r-1}}}. 
\end{align*}
Consequently, 
\begin{align}\label{equ:KST-Zaran-X-upper-bound}
    X
    & \le \binom{m}{s_1} \cdot  \frac{s_2 + \cdots+ s_r-r+1}{r-1} n^{r-1-\frac{1}{s_2\cdots s_{r-1}}} \notag \\
    & \le \frac{1}{s_1!} \frac{s_2 + \cdots+ s_r-r+1}{r-1} n^{r-1-\frac{1}{s_2\cdots s_{r-1}}} m^{s_1}. 
\end{align}
Combining~\eqref{equ:KST-Zaran-X-lower-bound} and~\eqref{equ:KST-Zaran-X-upper-bound} we obtain 
\begin{align*}
    |\mathcal{H}|
    & \le \left(\frac{(s_2 + \cdots+ s_r-r+1) m^{s_1} n^{r-1-\frac{1}{s_2 \cdots s_{r-1}}}}{(r-1)\binom{n}{r-1}}\right)^{\frac{1}{s_1}}\binom{n}{r-1} + (s_1-1)\binom{n}{r-1} \\
    & \le \left(\frac{(s_2 + \cdots+ s_r-r+1) m^{s_1} n^{r-1-\frac{1}{s_2 \cdots s_{r-1}}}}{r-1}\right)^{\frac{1}{s_1}}\left(\frac{n^{r-1}}{(r-1)!}\right)^{1-\frac{1}{s_1}} + (s_1-1)\binom{n}{r-1} \\
    & \le \frac{(s_2 + \cdots+ s_r-r+1)^{\frac{1}{s_1}}}{r-1} mn^{r-1-\frac{1}{s_1\cdots s_{r-1}}}  + (s_1-1) \binom{n}{r-1}, 
\end{align*}
    completing the proof of Proposition~\ref{PROP:hypergraph-KST-Zaran}. 
\end{proof}
\end{appendix}
\end{document}